\documentclass[11pt]{amsart}

\RequirePackage{etex}
\usepackage[utf8]{inputenc}
\usepackage[T1]{fontenc}
\usepackage[english]{babel}

\usepackage[mathscr]{euscript}
\usepackage{caption}
\usepackage{subcaption}
\usepackage{datetime}
\usepackage{url}
\usepackage{verbatim}
\usepackage{float}
\usepackage{lipsum}
\usepackage{setspace}

\usepackage{titletoc}
\usepackage{textcomp}
\usepackage{anysize}
\usepackage{fancyhdr}
\usepackage{calrsfs}

\usepackage{amssymb}
\usepackage{amsmath,mathtools}
\usepackage{amsfonts}
\usepackage{amscd}
\usepackage{amsthm}

\usepackage{latexsym}
\usepackage{enumerate}
\usepackage[shortlabels]{enumitem}
\usepackage{array}
\usepackage{stackrel}
\usepackage{stmaryrd}
\usepackage{stackengine}

\usepackage{graphicx}
\graphicspath{ {Images/} }
\usepackage{color,graphics}
\usepackage{tikz}
	\usetikzlibrary{arrows,automata,patterns, decorations.pathmorphing, decorations.pathreplacing, decorations.markings,tikzmark} 
\usepackage{tikz-cd}
	\usetikzlibrary{matrix,arrows,calc,automata,trees}
\usepackage[all,cmtip]{xy}
\usepackage{pgfplots,caption}
\pgfplotsset{compat=1.9}
\usepackage{ftnxtra}

\usepackage{fnpos}

\usepackage{lscape}
\usepackage[active]{srcltx}
\usepackage{rotating}

\usepackage[hypertexnames=false]{hyperref}

\makeatletter
\renewcommand{\tocsection}[3]{%
  \indentlabel{\@ifnotempty{#2}{\bfseries\ignorespaces#1 #2\quad}}\bfseries#3}
\renewcommand{\tocsubsection}[3]{%
  \indentlabel{\@ifnotempty{#2}{\ignorespaces#1 #2\quad}}#3}

\newcommand\@dotsep{4.5}
\def\@tocline#1#2#3#4#5#6#7{\relax
  \ifnum #1>\c@tocdepth 
  \else
    \par \addpenalty\@secpenalty\addvspace{#2}%
    \begingroup \hyphenpenalty\@M
    \@ifempty{#4}{%
      \@tempdima\csname r@tocindent\number#1\endcsname\relax
    }{%
      \@tempdima#4\relax
    }%
    \parindent\z@ \leftskip#3\relax \advance\leftskip\@tempdima\relax
    \rightskip\@pnumwidth plus1em \parfillskip-\@pnumwidth
    #5\leavevmode\hskip-\@tempdima{#6}\nobreak
    \leaders\hbox{$\m@th\mkern \@dotsep mu\hbox{.}\mkern \@dotsep mu$}\hfill
    \nobreak
    \hbox to\@pnumwidth{\@tocpagenum{\ifnum#1=1\bfseries\fi#7}}\par
    \nobreak
    \endgroup
  \fi}
\AtBeginDocument{%
\expandafter\renewcommand\csname r@tocindent0\endcsname{0pt}
}
\def\l@subsection{\@tocline{2}{0pt}{2.5pc}{5pc}{}}
\makeatother

\fancypagestyle{my_own_style}{
		\fancyhf{}
		\fancyfoot[C]{\thepage}
}

\newenvironment{spaceleft}
	{\list{}{\rightmargin0pt \leftmargin-1.2cm}
		\item\relax}
	{\endlist}

\newcommand{\N}{{\mathbb N}}
\newcommand{\Z}{{\mathbb Z}}
\newcommand{\Q}{{\mathbb Q}}
\newcommand{\C}{{\mathbb C}}
\newcommand{\R}{{\mathbb R}}
\newcommand{\T}{{\mathbb T}}
\newcommand{\Ps}{{\mathbb P}}
\newcommand{\V}{{\mathbb V}}
\newcommand{\K}{{\mathbb K}}
\newcommand{\F}{{\mathbb F}}

\DeclareMathAlphabet{\pazocal}{OMS}{zplm}{m}{n}

\newcommand{\calA}{{\pazocal A}}
\newcommand{\calB}{{\pazocal B}}
\newcommand{\calC}{{\pazocal C}}
\newcommand{\calD}{{\pazocal D}}
\newcommand{\calE}{{\pazocal E}}

\newcommand{\calG}{{\pazocal G}}

\newcommand{\calM}{{\pazocal M}}
\newcommand{\calN}{{\pazocal N}}
\newcommand{\calO}{{\pazocal O}}
\newcommand{\calP}{{\pazocal P}}
\newcommand{\calQ}{{\pazocal Q}}
\newcommand{\calR}{{\pazocal R}}
\newcommand{\calS}{{\pazocal S}}

\newcommand{\calU}{{\pazocal U}}

\newcommand{\calW}{{\pazocal W}}
\newcommand{\calX}{{\pazocal X}}

\newcommand{\pazB}{{\mathcal B}}

\newcommand{\pazE}{{\mathcal E}}
\newcommand{\pazF}{{\mathcal F}}

\newcommand{\pazP}{{\mathcal P}}
\newcommand{\pazQ}{{\mathcal Q}}

\newcommand{\pazT}{{\mathcal T}}

\newcommand{\gotR}{{\mathfrak R}}

\newcommand{\ra}{\rightarrow}

\newcommand{\act}{\curvearrowright}

\newcommand{\ol}{\overline}
\newcommand{\ul}{\underline}
\newcommand{\wh}{\widehat}
\newcommand{\wt}{\widetilde}

\newcommand{\tr}{\operatorname{tr}}
\newcommand{\Tr}{\operatorname{Tr}}
\newcommand{\rk}{\operatorname{rk}}
\newcommand{\Rk}{\operatorname{Rk}}
\newcommand{\Id}{\operatorname{Id}}

\newcommand{\soc}{\operatorname{soc}}

\newcommand{\xddots}{%
  \raise 4pt \hbox {.}
  \mkern 6mu
  \raise 1pt \hbox {.}
  \mkern 6mu
  \raise -2pt \hbox {.}
}

\numberwithin{equation}{section}

\theoremstyle{plain}

\newtheorem{theorem}{Theorem}[section]
\newtheorem*{theorem*}{Theorem}
\newtheorem{lemma}[theorem]{Lemma}
\newtheorem{proposition}[theorem]{Proposition}
\newtheorem{corollary}[theorem]{Corollary}

\theoremstyle{definition}

\newtheorem{definition}[theorem]{Definition}
\newtheorem*{definition*}{Definition}
\newtheorem{example}[theorem]{Example}

\newtheorem{remark}[theorem]{Remark}

\newtheorem*{remark*}{Remark}

\newtheorem*{question*}{Question}
\newtheorem*{sac*}{Strong Atiyah Conjecture}

\newtheorem*{assumption*}{Assumption}
\newtheorem{observation}[theorem]{Observation}

\title[The group algebra of the lamplighter]{Approximating the group algebra of the lamplighter by infinite matrix products}
\author{Pere Ara}
\address[P. Ara]{Departament de Matem\`atiques, Edifici Cc, Universitat Aut\`onoma de Barcelona, 08193 Cerdanyola del Vall\`es (Barcelona), Spain, and}
\address{Centre de Recerca Matem\`atica, Edifici Cc, Campus de Bellaterra, 08193 Cerdanyola del Vall\`es (Barcelona), Spain.}
\email{para@mat.uab.cat}
\author{Joan Claramunt}
\address[J. Claramunt]{Department of Mathematics and Statistics, Lancaster University, Bailrigg, Lancaster LA1 4YW, United Kingdom.}
\email{j.claramunt@lancaster.ac.uk}


\thanks{Both authors were partially supported by DGI-MINECO-FEDER through the grants MTM2017-83487-P and PID2020-113047GB-I00, and by the Generalitat de Catalunya through the grant 2017-SGR-1725. The first named author was also partially supported by the Spanish State Research Agency, through the Severo Ochoa and Mar\'ia de Maeztu Program for Centers and Units of Excellence in R$\&$D (CEX2020-001084-M). The second named author was also partially supported by DGI-MINECO-FEDER through the Grant BES-2015-071439, by the research funding Brazilian agency CAPES and by the ERC Starting Grant “Limits of Structures in Algebra and Combinatorics” No. 805495.}

\date{\today}

\begin{document}

\pagestyle{plain}
 
\begin{abstract}
 In this paper, we introduce a new technique in the study of the $*$-regular closure of some specific group algebras $KG$ inside $\calU(G)$, the $*$-algebra of unbounded operators affiliated to the group von Neumann algebra $\calN(G)$. 
 The main tool we use for this study is a general approximation result for a class of crossed product algebras of the form $C_K(X)\rtimes _T \Z$, where $X$ is a totally disconnected compact metrizable space, $T$ is a homeomorphism of $X$, and $C_K(X)$ stands for the algebra of locally constant functions on $X$ with values on an arbitrary field $K$. The connection between this class of algebras and a suitable class of group algebras is provided by Fourier transform. Utilizing this machinery, we study an explicit approximation for the lamplighter group algebra. This is used in another paper by the authors to obtain a whole family of $\ell^2$-Betti numbers arising from the lamplighter group, most of them transcendental. 
\end{abstract}

\maketitle

\tableofcontents

\normalsize

\section{Introduction}\label{section-introduction.Atiyah}
For a discrete group $G$ and a subfield $K$ of $\C$ closed under complex conjugation, the group algebra $KG$ can be naturally seen as a $*$-subalgebra of the von Neumann algebra $\calN (G)$ of $G$. The following problem, known as the {\it Strong Atiyah Conjecture} (SAC), was solved in the negative by Grigorchuk and Żuk \cite{GZ} (see also \cite{GLSZ} and \cite{DiSc}):

{\bf Strong Atiyah Conjecture:} If $T$ is an $m \times  n$ matrix over $K G$, then $$\dim_{\calN(G)}(\text{ker}\,T) \in \sum_H \frac{1}{|H|}\Z ,$$
where $H$ ranges over the family of all the finite subgroups of $G$, and $\text{ker}\, T$ denotes the kernel of the operator $\ell^2(\Gamma)^n \to \ell^2(\Gamma )^m$ given by left multiplication by $T$ (see e.g. \cite[page 369]{Luck}).

Here $\dim_{\calN(G)}$ stands for the von Neumann dimension, and the real number $\dim_{\calN(G)}(\text{ker}\,T)$ is called the {\it $\ell^2$-Betti number of $T$}.  The first counterexample to the SAC was found in \cite{GZ} and is given by the {\it lamplighter group}
$$\Gamma = \Z_2 \wr \Z = \Big(\bigoplus_{\Z}\Z_2\Big) \rtimes \Z,$$
where $\Z$ acts on $\bigoplus_{\Z}\Z_2$ by translation.  More precisely, denoting by $t$ the generator corresponding to $\Z$ and by $a_i$ the generator corresponding to the $i^{\text{th}}$ copy of $\Z_2$ in $G$, it was shown in \cite{GZ} that the self-adjoint element $T=s+s^*$, where $s=\frac{1}{2}(1+a_0)t$, satisfies that $\dim_{\calN(\Gamma)}(\text{ker}\,T) = 1/3$. This gives a counterexample for the SAC, because all the finite subgroups of $\Gamma$ have order a power of $2$, but does not solve the original question posed by Atiyah which asks whether all the $\ell^2$-Betti numbers $\dim_{\calN(G)}(\text{ker}\,T)$ are rational whenever $T \in \Z G$. This question was solved, also in the negative by Austin \cite{Aus} in a non-constructive way, and then other authors gave further counterexamples, see \cite{Gra14}, \cite{Gra16} and \cite{PSZ}. It is important to remark that Grabowski gave in \cite{Gra16} an example of an irrational (indeed transcendental) $\ell^2$-Betti number associated to the lamplighter group. Using, among other things, the techniques developed in the present paper, we build in \cite{AC3} a large family of real numbers, most of them transcendental, which arise as $\ell^2$-Betti numbers of the lamplighter group. We emphasize that the techniques developed here are essential for the construction in \cite{AC3} of the first examples of irrational algebraic virtual $\ell^2$-Betti numbers associated to the lamplighter group (see \cite[Question 1.7]{Gra14} and \cite[Problem 4]{Gra16}). 

The von Neumann algebra $\calN (G)$ can be naturally embedded in the $*$-algebra $\calU (G)$ of unbounded operators affiliated to $\calN (G)$. This algebra can be realized algebraically as the classical ring of quotients of $\calN (G)$, meaning that every element in $\calU (G)$ can be represented as a ``fraction'' $ab^{-1}$ where $a,b\in \calN (G)$ and $b$ has trivial kernel. The study of $\ell^2$-Betti numbers can be done by using the notion of a Sylvester matrix rank function. Such a function assigns to each matrix (of finite size) $A$ over a ring $R$ a real number $\rk (A)$ satisfying some axioms, and it is a generalization of the usual rank in matrices over a field.  (See Section \ref{section-preliminaries} for the precise definitions of these concepts.)

The canonical trace $\tr$ on $\calN (G)$ induces a canonical Sylvester matrix rank function $\rk$ on $\calU (G)$ by the rule $\rk(A)=\tr(P)$, where $A\in M_n(\calU (G))$ and  $P\in M_n(\calN (G))$ is its support projection. We have the formula $\rk (A) = n- \dim_{\calN(G)}(\text{ker}\,A)$ connecting the rank of $A$ with the von Neumann dimension of the kernel of $A$ (thought as an operator acting on $\ell^2(G)^n$).

The $*$-algebra $\calU (G)$ is a $*$-regular ring (see \cite{Berb} and \cite{Rei}) and therefore there is a smallest $*$-regular subalgebra $\calR_{KG}$ of $\calU (G)$ containing $KG$ (\cite{AG} and \cite{LS12}). The algebra $\calR_{KG}$ is called the {\it $*$-regular closure} of $KG$ in $\calU (G)$. The structure of $\calR_{KG}$ has been recently investigated in connection with the validity of the Strong Atiyah Conjecture and the L\"uck Approximation Conjecture in several papers, including \cite{AG}, \cite{Jaik-survey}, \cite{Jaik} and \cite{JL2}. An interesting result in this respect is given by Jaikin-Zapirain in \cite[Lemma 6.2]{Jaik}, which implies that for any $*$-subring $S$ of a $*$-regular ring $\calU$ such that $\calU$ agrees with the $*$-regular closure of $S$ in $\calU$, and every Sylvester matrix rank function $\rk$ on $\calU$, the subgroup $\calG (S)$ of $(\R,+)$ generated  by $\{\rk(A) \mid A\in M_n(S)\}$ coincides with $\phi (K_0(\calU))$, where $\phi $ is the state of $K_0(\calU)$ induced by $\rk$. 
 Applying this to the canonical rank function $\rk$ restricted to $\calR_{KG}$ one immediately recovers the known fact that $\calG (KG)=\Z$ if and only if $\calR_{KG}$ is a division ring (in which case, obviously, it must agree with the division closure of $KG$ in $\calU (G)$, see \cite[Lemma 3]{Schick}). Another fact one can easily show from the Jaikin-Zapirain result is that if the SAC holds for $G$, and there is a bound on the orders of the finite subgroups of $G$, then $\calR_{KG}$ must be an Artinian semisimple ring (see \cite[Theorem 1.5]{LS12} for a stronger related result).

In this paper we begin a systematic study of the $*$-regular algebra $\calR_{K\Gamma}$ of the lamplighter group $\Gamma$, and indeed of a much more general class of $*$-regular closures, as follows. We consider groups of the form $G = H\rtimes_{\rho} \Z$, the semidirect product of a countable discrete torsion abelian group $H$ by an action $\rho$ of $\Z$ on $H$. The first step in our strategy consists in using the Fourier transform to express the group algebra $KG$ in the form $C_K(X)\rtimes_{\hat{\rho}} \Z$, where $X = \widehat{H}$ is the Pontryagin dual of $H$, $C_K(X)$ is the algebra of continuous functions $f\colon X\to K$, where $K$ is endowed with the {\it discrete topology}, and $\hat{\rho}\colon  \Z\act \widehat{H}$ is the dual action. This process works smoothly for any field $K$ of characteristic $0$ containing all the $n^{\text{th}}$ roots of $1$, where $n$ ranges over the orders of the elements of $H$. In characteristic $p>0$ we have to impose in addition the condition that $p$ is coprime with all the orders of elements of $H$. (See Section \ref{section-group.alg} for details.)

Since $H$ is a discrete countable torsion abelian group, the space $X = \widehat{H}$ is a totally disconnected compact metrizable space. We can thus generalize the above setting by studying $*$-algebras $\calA = C_K(X)\rtimes_T \Z$, where $T$ is a homeomorphism of a totally disconnected infinite compact metrizable space $X$ and $K$ is any field endowed with a positive definite involution. This is precisely the situation studied in \cite{AC}, where it was shown that given a full $T$-invariant ergodic measure $\mu$ on $X$, there is a {\it unique} Sylvester matrix rank function $\rk$ on $\calA$ such that $\rk (\chi_U)= \mu (U)$ for each clopen subset $U$ of $X$. Here $\chi_U\in C_K(X)$ denotes the characteristic function of $U$. It is worth to mention here that, using the methods of \cite{virili} and \cite{Li2}, it is possible to obtain a Sylvester matrix rank function on $\calA$ satisfying the above condition. Although the method of \cite{AC} is quite different from the methods employed in those articles, these rank functions must agree by the uniqueness result mentioned above. Moreover by \cite{AC} the completion $\gotR_{\rk}$ of $\calA$ with respect to the metric topology induced by $\rk$ is a $*$-regular ring, which is $*$-isomorphic to the von Neumann continuous regular factor $\calM_K$. We show in Proposition \ref{proposition-measure.H} that for a group algebra $KG$, where $G = H\rtimes_{\rho}\Z$ is as above, the canonical rank function $\rk_{KG}$ corresponds under Fourier transform to the rank function $\rk_{\wh{\mu}}$, where $\wh{\mu}$ is the Haar measure on $\widehat{H}$. If in addition $\rk_{KG}$ is extremal in the compact convex set of Sylvester matrix rank functions, then $\wh{\mu}$ is ergodic and we are exactly in the setting of \cite{AC}. Moreover we show in Theorem \ref{theorem-identify.*reg.closure} that, with the above hypothesis, there is a rank-preserving $*$-isomorphism  $\calR_{KG}\cong \calR_{\calA}$, where $\calR_{\calA}$ is the $*$-regular closure of $\calA=C_K(\widehat{H})\rtimes _{\hat{\rho}} \Z$ in $\gotR_{\rk}$. It is well-known that the Haar measure on the group $\bigoplus _{\Z} \Z_2$ is a full  invariant ergodic measure and so the lamplighter group $\Gamma $ falls under the umbrella of our theory.

In the following, we fix a totally disconnected infinite compact metrizable space $X$ and a homeomorphism $T$ of $X$. We also fix a full $T$-invariant ergodic measure $\mu$ on $X$, and write $\calA:=C_K(X)\rtimes _T \Z$, where $K$ is a field with positive definite involution.  The basic idea to obtain approximating algebras for $\calA$ was developed in the setting of $C^*$-algebras by Putnam et al, mainly in the case where $T$ is a minimal homeomorphism, see e.g. \cite{HPS92}, \cite{P89}, \cite{P90}. We do not require in this paper that $T$ is minimal. Given a non-empty clopen subset $E$ of $X$ and a (finite) partition $\calP$ of $X \backslash E$ into clopen subsets, let $\calA (E,\calP)$ be the unital $*$-subalgebra of $\calA$ generated by the partial isometries $\chi_Zt$, where $Z\in \calP$. Then $\calA (E,\calP)$ can be embedded in a (possibly infinite) direct product $\gotR$ of finite matrix algebras over $K$. Now fixing a point $y\in X$ and considering sequences $(E_n,\calP_n)$ such that $E_n\supseteq E_{n+1}$ for all $n$, $\bigcap_{n\ge 1} E_n = \{ y\}$, each partition $\calP_{n+1}\cup \{E_{n+1}\}$ refines $\calP_n\cup \{E_n\}$, and such that $\bigcup_{n\ge1} (\calP_n\cup\{E_n\})$ generate the topology of $X$, we obtain a sequence of approximating $*$-algebras $\calA_n=\calA(E_n, \calP_n)$, each of them embedded in a direct product of finite matrix algebras $\gotR_n$. We set $\calA_{\infty}=\bigcup_{n\ge 1} \calA_n$ and $\gotR_{\infty}=\bigcup_{n\ge 1}\gotR_n$.

With this notation at hand, we can summarize our main results in the following theorem (see Propositions \ref{proposition-lim.of.reg.closure} and \ref{proposition-periodic.point.Rinfty}, and Theorem \ref{theorem-periodic.point.RA}).

\begin{theorem}
Let $\calA = C_K(X)\rtimes_T \Z$, $T$ and $\mu$ as stated above.  Then the nested sequence of approximating $*$-subalgebras $\calA_n$ of $\calA$ satisfies the following properties:
\begin{enumerate}
\item[(i)] For $n\in \N$, each algebra $\calA_n$ can be embedded in a $*$-regular ring $\gotR_n$ which is a (possibly infinite) direct product of finite matrix algebras over $K$.
\item[(ii)] There are natural inclusion maps $\gotR_n\hookrightarrow \gotR_{n+1}$ so that the following diagram is commutative:
\begin{equation*}
\vcenter{
\xymatrix{
\calA_n \ar@{^{(}->}[r] \ar@{^{(}->}[d] & \calA_{n+1} \ar@{^{(}->}[d] \ar@{^{(}->}[r] & \calA_{n+2} \ar@{^{(}->}[d] \ar@{^{(}->}[r] & \cdots \ar@{^{(}->}[r] & \calA_{\infty} \ar@{^{(}->}[d] \ar@{^{(}->}[r] & \calA \ar@{^{(}->}[d] \\
\calR_n \ar@{^{(}->}[r] \ar@{^{(}->}[d] & \calR_{n+1} \ar@{^{(}->}[d] \ar@{^{(}->}[r] & \calR_{n+2} \ar@{^{(}->}[d] \ar@{^{(}->}[r] & \cdots \ar@{^{(}->}[r] & \calR_{\infty} \ar@{^{(}->}[d] \ar@{^{(}->}[r] & \calR_{\calA} \ar@{^{(}->}[d] \\
\gotR_n \ar@{^{(}->}[r] & \gotR_{n+1} \ar@{^{(}->}[r] & \gotR_{n+2} \ar@{^{(}->}[r] & \cdots \ar@{^{(}->}[r] & \gotR_{\infty} \ar@{^{(}->}[r] & \gotR_{\rk},
}
}
\end{equation*} 
where, $\gotR_{\rk}$ is the rank-completion of $\calA$, $\calR_{\calA}$ is the $*$-regular closure of $\calA$ in $\gotR_{\rk}$, and for each $n\in \N \cup \{\infty\}$, $\calR_n$ is the $*$-regular closure of $\calA_n$ in $\gotR_n$.
\item[(iii)] Assume that $y$ is a periodic point for $T$ of period $l$. Then there is an ideal $I$ of $\calA_{\infty}$ which is also an ideal of $\calA$ such that
$$\calA_{\infty}/I\cong M_l(K) \quad \text{ and } \quad \calA/I\cong M_l(K[t,t^{-1}]).$$
\item[(iv)] The extension $\widetilde{I}$ of $I$ to $\calR_{\infty}$ satisfies that $\calR_{\infty}/\widetilde{I}\cong M_l(K)$. Moreover the ideal $\widetilde{I}$ can be further extended to an ideal $\ol{I}$ of $\calR_{\calA}$ such that $\calR_{\calA}/\ol{I}\cong M_l(K(t))$. In particular, $\ol{I}$ is a maximal ideal of the $*$-regular closure $\calR_{\calA}$.
\end{enumerate}
\end{theorem}

Parts (i) and (iii) come from \cite{AC} and are listed here for completeness. Observe that (iii) means that the $*$-subalgebra $\calA_{\infty}$ is large in $\calA$. Similarly (iv) indicates that the $*$-regular subalgebra $\calR_{\infty}$ is large in the $*$-regular algebra $\calR_{\calA}$. 

Using this result we can infer the following corollary, which provides useful information on the ideal structure of $\calR_{\calA}$:

\begin{corollary}\label{corollary-maximal.ideals}
With the above hypothesis and notation, suppose that $T$ has some periodic point. Then there is an injective map from the set of finite orbits for the action of $T$ to the set of maximal ideals of $\calR_{\calA}$. In particular, $\calR_{\calA}$ is not a simple ring.
\end{corollary}

See Lemma \ref{lemma-orbits.ideals} for the proof of Corollary \ref{corollary-maximal.ideals}. The algebra $\calR_{\calA}$ might be simple if $T$ has no periodic points, see \cite{AC3}.

Although the above theorem reduces in principle the study of the $*$-regular closure $\calR_{\calA}$ to the determination of the $*$-regular subalgebras $\calR_n$, it is a challenging problem to elucidate the structure of these algebras. For a given $*$-subalgebra $\calB:=\calA(E,\calP)$ associated to a pair $(E,\calP)$ as above, we start the investigation of the $*$-regular closure  $\calR_{\calB}$  of $\calB$ in $\gotR$. For this, we are guided by the study in \cite{AG} of a particular case, which corresponds to the choices $E_0 = [\ul{0}]$ and $\calP_0=\{[\ul{1}]\}$, where $\calA= C_K(\{0,1\}^{\Z})\rtimes_T \Z\cong K \Gamma$ is the lamplighter group algebra and $[\ul{\epsilon}] = \{(x_n) \in \{0,1\}^{\Z} \mid x_0 = \epsilon\}$ for $\epsilon = 0,1$. In that paper, the structure of the $*$-regular closure $\calR_0$ of $\calA_0:=\calA(E_0,\calP_0)$ in its corresponding direct product of finite matrix algebras $\gotR_0$ was completely determined, see \cite[Theorem 6.13]{AG}. For a general algebra $\calB = \calA(E,\calP)$, we find in the present paper a $*$-subalgebra $\calE$ of $\calR_{\calB}$ which provides a generalization of the construction in \cite{AG} (see Subsection \ref{subsection-analysis.A.G.example} for the exact relation between both constructions). We expect that $\calE$ will generally be only a proper $*$-subalgebra of $\calR_{\calB}$, but there are indications that the algebra $\calE$ is very large, see below.

The algebra $\calB$ has a description as a {\it partial crossed product} over $\Z$ \cite{Exel}, and we can write $\calB = \bigoplus_{i\in \Z} \calB_it^i$, where $\calB_0$ is a $*$-regular subalgebra of $C_K(X)$ and each $\calB_i=e_i\calB_0$ is an ideal of $\calB_0$ generated by an idempotent $e_i$ (see \cite{AC} and Subsection \ref{subsection-study.RB}). We observe that we can restrict the injective representation map $\pi\colon \calB\hookrightarrow \gotR$ to a representation $\pi_+\colon \calB_+=\bigoplus_{i\ge 0}\calB_it^i \to \gotR$ by lower triangular matrices. Considering now a skew partial formal power series algebra 
$$\calB_0[[t;T]] := \Big\{ \sum_{i\ge 0} b_it^i \mid b_i\in \calB_i \text{ for all } i\ge 0 \Big\} ,$$
we can extend $\pi_+$ to an injective representation $\pi_+\colon \calB_0[[t;T]]\to \gotR$ by lower triangular matrices.   Similarly we obtain a representation $\pi_-\colon \calB_0[[t^{-1}, T^{-1}]]\to \gotR$ by upper triangular matrices. The algebra $\calD_+$ is defined as the {\it division closure} of $\calB_+$ in $\calB_0[[t;T]]$; that is, $\calD_+$ is the smallest subalgebra of $\calB_0[[t;T]]$ containing $\calB_+$ and closed under inversion. Similarly we obtain a corresponding subalgebra $\calD_-$ of $\calB_0[[t^{-1},T^{-1}]]$. It turns out that $\pi_+(\calD_+)\subseteq \calR_{\calB}$ and $\pi_-(\calD_-)\subseteq \calR_{\calB}$, and that $\pi_+(\calD_+)^*=\pi_-(\calD_-)$  (see Proposition \ref{proposition-charac.division.closure}). Since $\calR_{\calB}$ is closed under the involution, the $*$-algebra $\calD$ generated by $\pi_+(\calD_+)$  must also be contained in $\calR_{\calB}$. 
The situation is summarized in the following diagram:

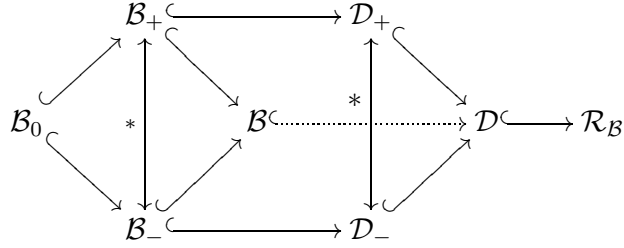
\begin{equation*}
\xymatrix{
	& \calB_+ \ar@{^{(}->}[rr] \ar@{^{(}->}[rd] \ar@{<->}[dd]_{*} & & \calD_+ \ar@{^{(}->}[rd] \ar@{<->}[dd]_{\rotatebox{90}{\text{ } \text{ } \text{ }\rotatebox{-90}{$*$}}} & & \\
	\calB_0 \ar@{^{(}->}[ru] \ar@{^{(}->}[rd] & & \calB \ar@{^{(}.>}[rr] & & \calD \ar@{^{(}->}[r] & \calR_{\calB} \\
	& \calB_- \ar@{^{(}->}[rr] \ar@{^{(}->}[ru] & & \calD_- \ar@{^{(}->}[ru] & &
}
\end{equation*}

Finally by using the somewhat technical notion of {\it special terms} (see Definition \ref{definition-special.terms}) we are able to enlarge the algebra $\calD$ by considering a certain $*$-regular subalgebra $\Psi (\calQ)p_E$
of $p_E\gotR p_E$, where $p_E:=\pi (\chi_E)$. The $*$-algebra $\calE$ is then defined as the subalgebra of $\gotR$ generated by $\calD$ and  $\Psi (\calQ)p_E$ (Definition \ref{definition-algebraQ}). It is shown in Proposition \ref{proposition-algebraQ.inside.calR} that we have inclusions of $*$-algebras $\calB\subseteq \calD\subseteq \calE \subseteq \calR_{\calB}$. 

In the last section of the paper, we give a specific approximation sequence $\{\calA_n\}$ for the lamplighter group algebra $\calA := C_K(X)\rtimes _T \Z \cong K \Gamma$, where here $X=\{0,1\}^{\Z}$ and $T$ is the shift $T(x)_i= x_{i+1}$ for $x\in X$. We use as a sequence of approximations the algebras $\calA_n:= \calA(E_n,\calP_n)$ associated to partitions given by cylinder sets   
$$[\epsilon_{-n} \ldots \underline{\epsilon_0} \cdots \epsilon_n]=\{ x=(x_i) \in X \mid x_{-n} = \epsilon_{-n},..., x_0= \epsilon_0,...,x_n = \epsilon_n \},$$
where $\epsilon_{-n},\dots ,\epsilon_n\in \{0,1\}$. For $n=0$, we recover the algebra $\calA_0$ studied in \cite{AG}. It is worth mentioning here that, for  $n\ge 1$, the algebra $\calR_n:=\calR_{\calA_n}$ contains a copy of the algebra $K_{\text{rat}}\langle  X\rangle $ of non-commutative rational series in infinitely many indeterminates (see Proposition \ref{proposition-rational.series}). This has potential applications for the computation of $\ell^2$-Betti numbers for the lamplighter, as explained in \cite{AC3}.

\medskip

The paper is organized as follows. Section \ref{section-preliminaries} contains preliminary definitions and results, and Section \ref{section-approx.crossed.product} contains details on the basic construction used in the paper, which is fully developed in \cite{AC}. We undertake in Section \ref{section-*.regular.closure} the general study of the $*$-regular closure $\calR_{\calA}$ of a crossed product algebra $\calA=C_K(X)\rtimes _T \Z$, where $X$ is a totally disconnected compact metrizable space, and $T$ is a homeomorphism of $X$. We first obtain the results about the approximation of $\calR_{\calA}$ by a sequence of $*$-regular algebras $\calR_n$, including the study about the difference between $\calR_{\calA}$ and $\calR_{\infty}=\bigcup_{n\ge 1} \calR_n$ in case $y$ is a periodic point of $T$. In Subsection \ref{subsection-study.RB}, we start our general study of the $*$-regular algebra
$\calR_{\calB}$ associated to an algebra $\calB = \calA(E,\calP)$, by identifying several interesting subalgebras therein. In Section \ref{section-group.alg}, we explain the connection between the crossed product algebras $C_K(X) \rtimes_T \Z$ and group algebras $KG$, where $G = H \rtimes_{\rho} \Z$ is a semidirect product with $H$ a countable torsion abelian group. The connection uses essentially the Fourier transform, but, since we impose only minimal conditions on our base field $K$, we need to work out some additional details. Finally Section \ref{section-lamplighter.algebra} contains our study of the group algebra $K\Gamma$ of the lamplighter group $\Gamma$. Using a concrete sequence $\{(E_n,\calP_n) \mid  n\in \Z^+\}$ of partitions of $X=\{0,1\}^{\Z}$, we are able to concretely compute several of the objects introduced in Section \ref{section-*.regular.closure} in the general setting. In particular we show in Proposition \ref{proposition-rational.series} that, for $n\ge 1$, the $*$-regular algebras $\calR_n$ contain a copy of the algebra of non-commutative rational series in infinitely many indeterminates. We conclude the paper by determining the exact relation of our theory with the algebra studied in \cite[Section 6]{AG}.

\section{Background and preliminaries}\label{section-preliminaries}

Here we collect background definitions, concepts, and results needed during the course of the paper.

\subsection{Von Neumann regular rings, \texorpdfstring{$*$}{}-regular rings and rank functions}\label{subsection-regular.rings.*.rank.functions}

A ring $R$ is called \textit{regular}, or \textit{von Neumann regular}, if for every element $x \in R$ there exists $y \in R$ such that $x = xyx$. In this case, the element $e = xy$ is an idempotent and generates the same right ideal as $x$. In fact, a characterization for regular rings is that every finitely generated right ideal of $R$ is generated by a single idempotent \cite[Theorem 1.1]{Goo91}.

A \textit{$*$-regular ring} is a regular ring $R$ endowed with an involution $*$ which is \textit{proper}, meaning that the equation $x^*x = 0$ implies $x = 0$. The involution is called \textit{positive definite} in case, for each $n \geq 1$, the equation $\sum_{i=1}^n x_i^*x_i = 0$ implies $x_i = 0$ for all $1 \leq i \leq n$. If $R$ is a $*$-regular ring with a positive definite involution, then $M_n(R)$ is also a $*$-regular ring when endowed with the $*$-transpose involution.

In a $*$-regular ring $R$ every principal right/left ideal is generated by unique projections. Specifically, for every $x \in R$ there exist unique projections $e, f \in R$ (usually denoted by $\mathrm{LP}(x)$ and $\mathrm{RP}(x)$, and termed the \textit{left} and \textit{right projections} of $x$, respectively) satisfying that $xR = eR$ and $Rx = Rf$. In this setting, there also exists a unique element $\ol{x} \in fRe$ such that $x\ol{x} = e$ and $\ol{x}x = f$. The element $\ol{x}$ is called the \textit{relative inverse} of $x$. We refer the reader to \cite{Ara87, Berb} for further information on $*$-regular rings.


For any subset $S \subseteq R$ of a unital $*$-regular ring $R$, there exists a smallest unital $*$-regular subring of $R$ containing $S$ (\cite[Proposition 6.2]{AG}, see also \cite[Proposition 3.1]{LS12}). This $*$-regular ring is called the \textit{$*$-regular closure} of $S$ in $R$, and is denoted by $\calR(S,R)$. In fact,
$$\calR(S,R) = \bigcup_{n \geq 0} \calR_n(S,R),$$
where $\calR_0(S,R)$ is the unital $*$-subring of $R$ generated by the set $S$, and $\calR_{n+1}(S,R)$ is generated by $\calR_n(S,R)$ and the relative inverses in $R$ of the elements of $\calR_n(S,R)$. It was observed in \cite{Jaik} that $\calR_{n+1}(S,R)$ can be described as the subring of $R$ generated by the elements of $\calR_n(S,R)$ and the relative inverses of the elements of the form $x^*x$ for $x \in \calR_n(S,R)$.

A \textit{pseudo-rank function} on a unital regular ring $R$ is a map $\rk : R \ra [0,1]$ that satisfies the following properties:
\begin{enumerate}[a),leftmargin=1cm]
\item $\rk(0) = 0$, $\rk(1) = 1$;
\item $\rk(xy) \leq \min\{\rk(x),\rk(y)\}$ for every $x,y \in R$;
\item if $e, f \in R$ are orthogonal idempotents, then $\rk(e+f) = \rk(e) + \rk(f)$.
\end{enumerate}
If $\rk$ satisfies the additional property
\begin{enumerate}[d),leftmargin=1cm]
\item $\rk(x) = 0$ only if $x = 0$,
\end{enumerate}
then $\rk$ is called a \textit{rank function} on $R$. For general properties of pseudo-rank functions over regular rings one can consult \cite[Chapter 16]{Goo91}.

Every (pseudo-)rank function $\rk$ on a regular ring $R$ defines a (pseudo-)metric $d$ on $R$ by the rule $d(x,y) = \rk(x-y)$. Since the ring operations are continuous with respect to this (pseudo-)metric, one can consider the completion $\ol{R}$ of $R$ with respect to $d$. It is also a regular ring, and $\rk$ can be uniquely extended continuously to a rank function $\ol{\rk}$ on $\ol{R}$ such that $\ol{R}$ is also complete with respect to the metric induced by $\ol{\rk}$. 


Regular rings are also of great interest since every (pseudo-)rank function $\rk$ on $R$ can be uniquely extended to a (pseudo-)rank function on matrices over $R$ (see e.g. \cite[Corollary 16.10]{Goo91}). This is no longer true if we do not assume $R$ to be regular. The definition that seems to fit in the general setting is the notion of a Sylvester matrix rank function.

\begin{definition}\label{definition-sylvester.rank}
Let $R$ be a unital ring, and set $M(R) = \bigcup_{n \geq 1} M_n(R)$. A \textit{Sylvester matrix rank function} on $R$ is a map $\rk : M(R) \ra \R^+$ satisfying the following conditions:
\begin{enumerate}[a),leftmargin=1cm]
\item ${\rm rk} (M)= 0$ if $M$ is a zero matrix, and ${\rm rk}(1)= 1$;
\item ${\rm rk} (M_1M_2) \leq \min\{{\rm rk}(M_1), {\rm rk}(M_2)\}$ for any matrices $M_1$ and $M_2$ of appropriate sizes;
\item ${\rm rk} \begin{pmatrix} M_1 & 0 \\ 0 & M_2 \end{pmatrix} = {\rm rk} (M_1) + {\rm rk}(M_2) $ for any matrices $M_1$ and $M_2$;
\item ${\rm rk} \begin{pmatrix} M_1 & M_3 \\ 0 & M_2  \end{pmatrix} \ge {\rm rk}(M_1) + {\rm rk}(M_2)$ for any matrices $M_1$, $M_2$ and $M_3$ of appropriate sizes.
\end{enumerate}
\end{definition}

Sylvester matrix rank functions were first introduced by Malcolmson in \cite{Mal} in order to study homomorphisms to division rings.
For more theory and properties about Sylvester matrix rank functions we refer the reader to \cite{Jaik}, \cite{Li1}, \cite{Li2} and \cite[Part I, Chapter 7]{Sch}.  

We denote by $\Ps(R)$ the compact convex set of Sylvester matrix rank functions on $R$. By \cite[Proposition 16.20]{Goo91} this space coincides with the space of pseudo-rank functions on $R$ when $R$ is a regular ring.

As in the case of pseudo-rank functions on a regular ring, a Sylvester matrix rank function $\rk$ on a unital ring $R$ gives rise to a pseudo-metric by the rule $d(x,y) = \rk(x-y)$. We call the Sylvester matrix rank function \textit{faithful} if its kernel, defined as the set of all element $x \in R$ with zero rank, is exactly $\{ 0 \}$. In this case $d$ becomes a metric on $R$. Again, the ring operations are continuous with respect to $d$, so one can consider the completion $\ol{R}$ of $R$ with respect to $d$, and $\rk$  can be uniquely extended continuously to a Sylvester matrix rank function $\ol{\rk}$ on $\ol{R}$.

A very useful result connecting the $*$-regular closure $\calR(S,R)$ and possible values of a Sylvester matrix rank function defined on $R$ is given in the following proposition, which can be thought of as an analogue of the classical Cramer's rule.

\begin{proposition}[Corollary 6.2 of \cite{Jaik}]\label{proposition-cramer.rule.reg.closure}
Let $S$ be a unital $*$-subring of a $*$-regular ring $R$, and let $\calR = \calR(S,R)$ be the $*$-regular closure of $S$ in $R$.

Then for any matrices $r_1,...,r_k \in M_{n \times m}(\calR)$, there exists a matrix $M \in M_{a \times b}(S)$ and matrices $A_1,...,A_k \in M_{n \times b}(S)$ such that, for any other square-matrices $t_1,...,t_k \in M_n(S)$ and any Sylvester matrix rank function $\rk$ defined on $\calR$,
$$\rk(t_1 r_1 + \cdots + t_k r_k) = \rk \begin{pmatrix}
M \\
t_1 A_1 + \cdots + t_k A_k
\end{pmatrix} - \rk(M).$$
In particular, any Sylvester matrix rank function on $\calR$ is completely determined by its values on matrices over $S$.
\end{proposition}

\subsection{Division closure and rational closure}\label{subsection-noncommutative.localization} In this subsection we recall the well-known concepts of the division closure and the rational closure.
We also introduce the notion of universal localization in a special case. See \cite{Cohn85} for a detailed treatment.

Let $R$ be a unital subring of a ring $T$. The {\it division closure} $\calD (R,T)$ of $R$ in $T$ is the smallest subring of $T$ containing $R$ and closed under inverses in $T$. So we have $R\subseteq \calD (R,T)\subseteq T$ and $d^{-1}\in \calD (R,T)$ whenever $d\in \calD (R,T)$ is invertible in $T$. The {\it rational closure} of $R$ in $T$ is the smallest subring $\calR at (R,T)$ of $T$ containing $R$ and closed under inverses of square matrices. We thus have $R\subseteq \calR at (R,T)\subseteq T$, and whenever a matrix $A\in M_n(\calR at (R,T))$ is invertible in $M_n(T)$, then all the entries of $A^{-1}$ belong to $\calR at (R,T)$. 

We summarize in the next lemma some properties concerning the rational, division and $*$-regular closures. Its proof is straightforward, so we omit it.

\begin{lemma}\label{lemma-div.rat.closure}
Let $R$ be a unital subring of $T$. Then the following properties hold:
\begin{enumerate}[i),leftmargin=1cm]
\item $\calD(\calD(R,T),T) = \calD(R,T) = \calD(R,\calD(R,T))$.
\item $\calR at( \calR at(R,T), T) = \calR at(R,T) = \calR at(R,\calR at(R,T))$.
\item $\calD(R,T) \subseteq \calR at(R,T)$.
\end{enumerate}
Moreover, if $T$ is $*$-regular, then the $*$-regular closure of $R$ in $T$, which we denote by $\calR (R,T)$, contains the rational closure $\calR at(R,T)$.
\end{lemma}

We will need in Section \ref{section-lamplighter.algebra} the notion of universal localization of a ring, but only with respect to elements of the ring.
Let $\Sigma$ be a subset of a unital ring $R$. Then the {\it universal localization} of $R$ with respect to $\Sigma$ is the ring $\Sigma^{-1}R$ obtained by universally adjoining to $R$ inverses of elements of $\Sigma$. There is a canonical ring homomorphism $\iota_{\Sigma}\colon R\to \Sigma^{-1}R$ satisfying the following universal property: for each ring homomorphism $\varphi\colon R\to S$ such that $\varphi (s)$ is invertible for all $s\in \Sigma$, there exists a unique homomorphism $\tilde{\varphi}\colon \Sigma^{-1} R\to S$ such that $\varphi= \tilde{\varphi}\circ \iota_{\Sigma}$.

\subsection{\texorpdfstring{$\ell^2$}{}-Betti numbers for group algebras and the Atiyah Conjecture}\label{subsection-Betti.numbers.Atiyah.problem}
Here we recall some basic facts on $\ell^2$-Betti numbers associated to a group. We refer the reader to \cite{Luck} for more information on this subject.

For a discrete countable group $G$ and a subfield $K$ of $\C$ closed under complex conjugation we consider the group $*$-algebra $KG$ acting on the Hilbert space $\ell^2(G)$ by left multiplication. 
We denote by $\calN(G)$ the weak-completion of $\C G$ in $\calB(\ell^2(G))$, which is commonly known as the \textit{group von Neumann algebra of $G$}. An equivalent algebraic definition can be given: it consists exactly of those bounded operators $T : \ell^2(G) \ra \ell^2(G)$ that are $G$-equivariant, i.e. $T(\xi g) = T(\xi)g$ for $\xi\in \ell^2(G)$ and $g\in G$. 

The algebra $\calN(G)$ is endowed with a normal, positive and faithful trace, defined as
$$\tr_{\calN(G)}(T) := \langle T(\xi_e),\xi_e \rangle_{\ell^2(G)} \quad \text{for } T \in \calN(G),$$
where $\{\xi_g\}_{g\in G}$ is the canonical orthonormal basis of $\ell^2(G)$.
Note that for an element $T = \sum_{\gamma \in G} a_{\gamma} \gamma \in \C G$, its trace is simply the coefficient $a_e$.

All the above constructions can be extended to $k \times k$ matrices: the $*$-algebra $M_k(K G)$ acts faithfully on $\ell^2(G)^k$ by left multiplication. We denote by $\calN_k (G)$ the weak-completion of $M_k(\C G)$ inside $\calB (\ell^2(G)^k)$, which coincides with $M_k(\calN (G))$. The previous trace can be extended to an (unnormalized) trace in $\calN_k (G)$ by setting, for a matrix $T = (T_{ij}) \in \calN_k(G)$,
$$\Tr_{\calN_k (G)}(T) := \sum_{i=1}^k \tr_{\calN (G)}(T_{ii}).$$

Every matrix operator $A \in M_k(K G)$ gives rise to an \textit{$\ell^2$-Betti number}, in the following way. Consider $A$ as an operator $A : \ell^2(G)^k \ra \ell^2(G)^k$ acting on the left, and take $p_A \in \calN_k (G)$ to be the projection onto the kernel of $A$. One can then define the von Neumann dimension $\dim_{\calN(G)}(\text{ker}\,A)$ of $\text{ker}\,A$ as the trace of the projection $p_A$.

\begin{definition}\label{definition-l2bettinumber.grouprings}
A real positive number $r$ is called an \textit{$\ell^2$-Betti number arising from $G$ with coefficients in $K$} if for some integer $k \geq 1$, there exists a matrix operator $A \in M_k(K G)$ such that
$$\dim_{\calN(G)}(\text{ker}\,A) := \Tr_{\calN_k (G)}(p_A) = r.$$
We denote the set of all $\ell^2$-Betti numbers arising from $G$ with coefficients in $K$ by $\calC(G,K)$. It should be noted that this set is always a subsemigroup of $(\R^+,+)$. The subgroup of $(\R,+)$ generated by $\calC(G,K)$ will be denoted throughout the article by $\calG(G,K)$.
\end{definition}

It is also possible to define the von Neumann dimension by means of a \textit{rank function}, as follows. Let $\calU (G)$ be the classical ring of quotients of $\calN (G)$. It is a $*$-regular ring possessing a Sylvester matrix rank function $\rk$ defined by
$$\rk (A) := \Tr_{\calN_k(G)}(\mathrm{LP}(A)) = \Tr_{\calN_k(G)}(\mathrm{RP}(A))$$
for any matrix $A \in M_k(\calU(G))$, where $\mathrm{LP}(A)$ and $\mathrm{RP}(A)$ are the left and right projections of $A$ inside $M_k(\calU(G))$, respectively. In particular, we obtain by restriction a Sylvester matrix rank function $\rk_{KG}$ on $K G$ , and we have the equality
\begin{equation}\label{equation-vN.dim.rank.UG}
\dim_{\calN(G)}(\text{ker}\,A) = k - \rk_{KG} (A)
\end{equation}
for each matrix operator $A \in M_k(K G)$.

\section{A dynamical approximation for crossed product algebras}\label{section-approx.crossed.product}

Let us recall the general construction used in \cite{AC} for providing an essentially unique Sylvester matrix rank function on an algebraic crossed product algebra $\calA := C_K(X) \rtimes_T \Z$, where $T : X \ra X$ is a homeomorphism of a totally disconnected, compact metrizable space $X$, which we also assume to be infinite (e.g. one can take $X$ to be the Cantor space). Here $K$ is an arbitrary field with an involution $\lambda \mapsto \ol{\lambda}$, and $C_K(X)$ is the $*$-algebra of continuous functions $f : X \ra K$ where $K$ is endowed with the discrete topology; equivalently, is the $*$-algebra of locally constant functions $f : X \ra K$. For the construction, a $T$-invariant, ergodic and full probability measure $\mu$ on $X$ is also needed. We refer the reader to \cite[Section 3]{AC}. This construction is used throughout the paper.

For $\emptyset \neq E \subseteq X$ any clopen subset and $\calP$ any (finite) partition of the complement $X \backslash E$ into clopen subsets, define $\calB$ to be the unital $*$-subalgebra of $\calA$ generated by the family of partial isometries $\{\chi_Z t\}_{Z \in \calP}$. By \cite[Lemma 3.4]{AC}, the $*$-algebra $\calB_0 = C_K(X) \cap \calB$ is linearly spanned by the unit $1$ and the projections of the form $\chi_C$, being $C$ a non-empty clopen subset of $X$ of the form
\begin{equation}\label{equation-form.of.bs}
T^{-r}(Z_{-r}) \cap \cdots \cap Z_0 \cap \cdots \cap T^{s-1}(Z_{s-1}),
\end{equation}
where $Z_{-r},...,Z_0,...,Z_{s-1} \in \calP$ and $r,s \geq 0$. Here $\chi_U$ denotes the characteristic function of the clopen subset $U \subseteq X$. We have the decomposition
$$\calB = \bigoplus_{i \in \Z} \calB_0 (\chi_{X \backslash E}t)^i = \bigoplus_{i \in \Z} \calB_i t^i$$
with $\calB_i = \chi_{X \backslash(E \cup \cdots \cup T^{i-1}(E))}\calB_0$ and $\calB_{-i} = \chi_{X \backslash(T^{-1}(E) \cup \cdots \cup T^{-i}(E))}\calB_0$ for $i > 0$. By \cite[Lemma 3.8]{AC}, for $i > 0$ the algebra $\calB_i$ is linearly spanned by the terms $\chi_C$, where $C$ is of the form \eqref{equation-form.of.bs} with $s \geq i$, and similarly $\calB_{-i}$ is linearly spanned by the terms $\chi_C$, where $C$ is of the form \eqref{equation-form.of.bs} with $r \geq i$.

There exists a quasi-partition of $X$ (i.e. a countable family of non-empty, pairwise disjoint clopen subsets whose union has full measure) given by the $T$-translates of clopen subsets $W$ of the form
\begin{equation}\label{equation-Wexpression}
W = E \cap T^{-1}(Z_1) \cap \cdots \cap T^{-k+1}(Z_{k-1}) \cap T^{-k}(E) \quad \text {for some $k \geq 1$ and $Z_i \in \calP$},
\end{equation}
whenever these are non-empty. In fact, if we write $|W| := k$ (the length of $W$) and $\V := \{W \neq \emptyset \text{ as above}\}$, then for a fixed $W \in \V$ the characteristic functions $\{\chi_W, ..., \chi_{T^{|W|-1}(W)}\}$ belong to $\calB$, and moreover the set of elements $e_{ij}(W) = (\chi_{X \backslash E}t)^i \chi_W (t^{-1} \chi_{X \backslash E})^j$, $0 \leq i,j \leq |W|-1$, forms a set of $|W| \times |W|$ matrix units in $\calB$ (i.e. they satisfy $e_{ij}(W) e_{ts}(W) = \delta_{j,t} e_{is}(W)$ for all allowable indices $i,j,t,s$). In addition, the element $h_W := e_{00}(W) + \cdots + e_{|W|-1,|W|-1}(W)$ is central in $\calB$ and, by \cite[Proposition 3.11]{AC}, we have
\begin{equation}\label{equation-iso.central}
h_W \calB \cong M_{|W|}(K).
\end{equation}
From now on, and slightly abusing language, we will identify $h_W \calB$ with $M_{|W|}(K)$ whenever convenient. In this way one constructs an injective $*$-representation $\pi : \calB \hookrightarrow \prod_{W \in \V} M_{|W|}(K) =: \gotR$ defined by $\pi(a) = (h_W \cdot a)_W$ \cite[Proposition 3.13]{AC}.

The $*$-algebra $\calB$ corresponding to $(E,\calP)$ as above will be denoted by $\calA(E,\calP)$.\\

Take now $\{E_n\}_{n \geq 1}$ to be a decreasing sequence of clopen sets of $X$ together with a family $\{\calP_n\}_{n \geq 1}$, being each $\calP_n$ a (finite) partition of $X \backslash E_n$ into clopen subsets, such that:
\begin{enumerate}[a),leftmargin=1cm]
\item the intersection of all the $E_n$ consists of a single point $y \in X$;
\item $\calP_{n+1} \cup \{E_{n+1}\}$ is a partition of $X$ finer than $\calP_n \cup \{E_n\}$;
\item $\bigcup_{n \geq 1} (\calP_n \cup \{E_n\})$ generates the topology of $X$.
\end{enumerate}
By writing $\V_n$ for the (non-empty) sets \eqref{equation-Wexpression} corresponding to the pair $(E_n,\calP_n)$ and setting $\calA_n := \calA(E_n,\calP_n)$, we get injective $*$-representations $\pi_n$ of $\calA_n$ into $\gotR_n := \prod_{W \in \V_n} M_{|W|}(K)$, in such a way that the diagrams
\begin{equation}
\vcenter{
	\xymatrix{
	\calA_n \ar@{^{(}->}[r]^{\iota_n} \ar@{^{(}->}[d]^{\pi_n} & \calA_{n+1} \ar@{^{(}->}[d]^{\pi_{n+1}} \ar@{^{(}->}[r]^{\iota_{n+1}} & 	\calA_{n+2} \ar@{^{(}->}[d]^{\pi_{n+2}} \ar@{^{(}->}[r] & \cdots \ar@{^{(}->}[r] & \calA_{\infty} \ar@{^{(}->}[d]^{\pi_{\infty}} \\
	\gotR_n \ar@{^{(}->}[r]^{j_n} & \gotR_{n+1} \ar@{^{(}->}[r]^{j_{n+1}} & \gotR_{n+2} \ar@{^{(}->}[r] & \cdots \ar@{^{(}->}[r] & \gotR_{\infty}
	}
}\label{diagram-comm.diag.1}
\end{equation}
commute. Here $\iota_n$ is the natural embedding $\iota_n(\chi_Z t) = \sum_{Z'} \chi_{Z'}t$ where the sum is taken with respect to all $Z' \in \calP_{n+1}$ satisfying $Z' \subseteq Z$, the maps $j_n : \gotR_n \hookrightarrow \gotR_{n+1}$ are the embeddings given in \cite[Proposition 4.2]{AC}, and $\calA_{\infty}, \gotR_{\infty}$ are the inductive limits of the direct systems $(\calA_n,\iota_n), (\gotR_n,j_n)$, respectively. In fact, the algebra $\calA_{\infty}$ can be explicitly described in terms of the crossed product, as follows. For $U \subseteq X$ an open set, denote by $C_{c,K}(U)$ the ideal of $C_K(X)$ generated by the characteristic functions $\chi_V$, where $V$ ranges over the clopen subsets $V \subseteq X$ contained in $U$. By \cite[Lemma 4.3]{AC}, $\calA_{\infty}$ coincides with the $*$-subalgebra of $\calA = C_K(X) \rtimes_T \Z$ generated by $C_K(X)$ and $C_{c,K}(X \backslash \{y\}) t$. This will be used in the next section when describing the \textit{$*$-regular closure} of $\calA$.

The following easy example might be useful to understand the concepts introduced above. We will develop the main example of this paper, corresponding to the group algebra of the {\it lamplighter group}, in Section \ref{section-lamplighter.algebra}.  

\begin{example}[Odometer]\label{example-odometer}
Let $X$ be the compact space $X = \prod_{i \in \N}\{0,1\}$ and let $T : X \ra X$ be the homeomorphism given by the odometer, namely for $x = (x_i)_i \in X$, $T$ is given by
$$T(x) = \begin{cases} (1,x_2,x_3,...) & \text{ if } x_1 = 0, \\ (0,...,0,1,x_{n+2},...) & \text{ if } x_i = 1 \text{ for } 1 \leq i \leq n \text{ and } x_{n+1} = 0, \\ (0,0,...) & \text{ if } x_i = 1 \text{ for all } i \in \N. \end{cases}$$
Note that the odometer action is just addition of $(1,0,0,...)$ by carry-over. The algebra $\calO := C_K(X) \rtimes_T \Z$ is known as the \textit{odometer algebra}.

Given $\epsilon_{1},...,\epsilon_l \in \{0,1\}$, the cylinder set $\{ x = (x_i) \in X \mid x_{1} = \epsilon_{1},...,x_l = \epsilon_l \}$ will be denoted by $[\epsilon_1 \cdots \epsilon_l]$. It is then clear that a basis for the topology of $X$ is given by the collection of clopen sets consisting of all the cylinder sets.

We have a natural measure $\mu$ on $X$ given by the usual product measure, where we take the $\big( \frac{1}{2}, \frac{1}{2} \big)$-measure on each component $\{0,1\}$. It is well-known (cf. \cite[Section VIII.4]{Dav}) that $\mu$ is an ergodic, full and $T$-invariant probability measure on $X$.

For $n \geq 1$, we take $E_n = [0 \cdots 0]$ (with $n$ zero's) for the sequence of clopen sets, whose intersection gives the point $y = (0,0,0,...) \in X$. For each $n\ge 1$, we take the partition $\calP_n$ of the complement $X \backslash E_n$ to be the obvious one, namely
$$\calP_n = \{[\epsilon_1 \epsilon_2 \cdots \epsilon_n] \text{ } \mid \text{ } \epsilon_i \in \{0,1\}\} \backslash \{E_n\}.$$
Write $\calO_n := \calA(E_n,\calP_n)$. The quasi-partition $\ol{\calP}_n$ is really simple in this case: write $Z_{n,l} = T^l(E_n)$ for $1 \leq l \leq 2^n-1$. Note that these clopen sets form exactly the partition $\calP_n$, and that $T(Z_{n,2^n-1}) = E_n$. Therefore there is only one possible $W \in \V_n$, which has length $2^n$ and is given by
$$W = E_n \cap T^{-1}(Z_{n,1}) \cap T^{-2}(Z_{n,2}) \cap \cdots \cap T^{-2^n+1}(Z_{n,2^n-1}) \cap T^{-2^n}(E_n) = E_n.$$
The representations $\pi_n : \calO_n \ra \gotR_n$ become $*$-isomorphisms, where here $\gotR_n = M_{2^n}(K)$. Under this identification, the embeddings $\iota_n : \calO_n \hookrightarrow \calO_{n+1}$ become the block-diagonal embeddings
$$j_n : M_{2^n}(K) \hookrightarrow M_{2^{n+1}}(K), \quad x \mapsto \text{diag}(x,x).$$
So here $\calO_{\infty} \cong \varinjlim_n M_{2^n}(K)$.\qed
\end{example}

Returning to the general situation, observe that we can define a Sylvester matrix rank function $\rk_n$ on each $\gotR_n$ by the rule
$$\rk_n(M) = \sum_{W \in \V_n} \mu(W) \Rk(M_W) \quad \text{ for } M = (M_W)_W \in \gotR_n,$$
being $\Rk$ the usual rank of matrices. These rank functions are compatible with respect to the embeddings $j_n$, so they give rise to a well-defined rank function $\rk_{\infty}$ on $\gotR_{\infty}$. Using this rank function it is possible to define a Sylvester matrix rank function over $\calA$, unique with respect to a certain compatibility property concerning the measure $\mu$, as the following theorem states.

\begin{theorem}\cite[Theorem 4.7 and Proposition 4.8]{AC}\label{theorem-rank.function}
If $\gotR_{\rk}$ denotes the rank-completion of $\gotR_{\infty}$ with respect to its rank function $\rk_{\infty}$, then we can embed $\calA \hookrightarrow \gotR_{\rk}$ in such a way that it coincides with the embedding $\calA_{\infty} \hookrightarrow \gotR_{\infty}$ when restricted to $\calA_{\infty}$, and the element $t$ is sent to $\lim_n \pi_n(\chi_{X \backslash E_n}t)$.

Moreover, the Sylvester matrix rank function $\rk_{\calA}$ induced by restriction of $\ol{\rk_{\infty}}$ (the extension of $\rk_{\infty}$ to $\gotR_{\rk}$) on $\calA$ is extremal, and unique with respect to the following property:
$$\rk_{\calA}(\chi_U) = \mu(U) \quad \text{ for every clopen subset } U \subseteq X.$$

Finally, the rank-completion of $\calA$ with respect to $\rk_{\calA}$ gives back $\gotR_{\rk}$.
\end{theorem}

\section{The \texorpdfstring{$*$}{}-regular closure \texorpdfstring{$\calR_{\calA}$}{} and its ideal structure}\label{section-*.regular.closure}


In this section we undertake the study of the $*$-regular closure of the crossed product algebra $\calA = C_K(X) \rtimes_T \Z$ inside $\gotR_{\rk}$ (see Theorem \ref{theorem-rank.function}) when $(K,-)$ is a field with positive definite involution. In this situation, the direct product $\gotR_n = \prod_{W \in \V_n} M_{|W|}(K)$ has a structure of $*$-regular $K$-algebra, where each matrix algebra $M_{|W|}(K)$ is endowed with the $*$-transpose involution. Since this structure is compatible with the transition homomorphisms, we obtain a stucture of $*$-regular algebra on $\gotR_{\infty}$, and thus on its completion $\gotR_{\rk}$. Moreover, $\calA$ sits naturally as a $*$-subalgebra of $\gotR_{\rk}$. See \cite[Theorem 4.9]{AC} for more details.

Let $\calR_{\calA} := \calR(\calA,\gotR_{\rk})$ denote the $*$-regular closure of $\calA$ inside $\gotR_{\rk}$.
The structure of this $*$-algebra is closely related to the possible values of the rank function $\rk_{\calA}$; in other words, it is related with the set
\begin{equation}\label{equation-definition.CA}
\calC(\calA) := \rk_{\calA}\Big(\bigcup_{n \geq 1} M_n(\calA)\Big) \subseteq \R^+.
\end{equation}
As in Definition \ref{definition-l2bettinumber.grouprings}, this set has the structure of a semigroup, inherited from $(\R^+,+)$. We will denote by $\calG(\calA)$ the subgroup of $(\R,+)$ generated by $\calC(\calA)$. The set $\calC(\calA)$ can be thought of as the set of \textit{complementary} $\ell^2$-Betti numbers arising from $\calA$ (cf. Definition \ref{definition-l2bettinumber.grouprings} and \eqref{equation-vN.dim.rank.UG}).

 The exact relation between $\calG(\calA)$ and $\calR_{\calA}$ is given in the following proposition, which is motivated by Proposition \ref{proposition-cramer.rule.reg.closure}. Recall that, for a unital ring $\mathfrak A$, the abelian group $K_0(\mathfrak A)$ is the Grothendieck group of the monoid of isomorphism classes of finitely generated projective $\mathfrak A$-modules, see for instance \cite{Ros}.

\begin{proposition}\label{proposition-motivation.values.ranks}
The group $\calG(\calA)$ coincides with the subgroup of $(\R,+)$ generated by the set
$$\rk_{\calR_{\calA}}(\calR_{\calA}) = \{\rk_{\calR_{\calA}}(r) \mid r \in \calR_{\calA}\},$$
where $\rk_{\calR_{\calA}}$ is the restriction of $\ol{\rk_{\infty}}$ to $\calR_{\calA}$ (see Theorem \ref{theorem-rank.function}). Equivalently, it coincides with the image of the state
$$\phi : K_0(\calR_{\calA}) \ra \R, \quad [p] - [q] \mapsto \rk_{\calR_{\calA}}(p) - \rk_{\calR_{\calA}}(q).$$
\end{proposition}
\begin{proof}
Write $\calS_1$ for the subgroup generated by $\rk_{\calR_{\calA}}(\calR_{\calA})$. By Proposition \ref{proposition-cramer.rule.reg.closure}, we have the inclusion $\calS_1 \subseteq \calG(\calA)$.

For the other inclusion note first that, since $\calR_{\calA}$ is a $*$-regular ring with positive definite involution, each matrix algebra $M_n(\calR_{\calA})$ is $*$-regular too. So for each $A \in M_n(\calR_{\calA})$ there exists a projection $P \in M_n(\calR_{\calA})$ such that $\rk_{\calR_{\calA}}(A) = \rk_{\calR_{\calA}}(P)$. We conclude that $\calC(\calA)$ is contained in the set of positive real numbers of the form $\rk_{\calR_{\calA}}(P)$, where $P$ ranges over matrix projections with coefficients in $\calR_{\calA}$. Now each such projection $P$ is equivalent to a diagonal one \cite[Proposition 2.10]{Goo91}, that is, one of the form $\text{diag}(p_1,...,p_r)$
for some projections $p_1,...,p_r \in \calR_{\calA}$, so that $\rk_{\calR_{\calA}}(P) = \rk_{\calR_{\calA}}(p_1) + \cdots + \rk_{\calR_{\calA}}(p_r) \in \calS_1$, and $\calG(\calA) \subseteq \calS_1$.

Now the last part of the proposition follows easily, since $\phi(K_0(\calR_{\calA})) = \calS_1$.
\end{proof}

Before continuing, it is worth to mention that we can completely determine the rank-completion of $\calR_{\calA}$: it is the well-known \textit{von Neumann continuous factor} $\calM_K$, defined as the completion of $\varinjlim_n M_{2^n}(K)$ with respect to its unique rank function (see \cite{AC2,Elek16} for details). Here the direct limit is taken with respect to the block diagonal embeddings $x \mapsto \text{diag}(x,x)$.

\begin{proposition}\label{proposition-reg.closure.completion}
With the above notation, we have that $\ol{\calR_{\calA}}^{\rk_{\calR_{\calA}}} \cong \calM_K$ as $*$-algebras.
\end{proposition}
\begin{proof}
Here $\calM_K$ has the involution induced from the $*$-transpose involution on each matrix algebra $M_{2^n}(K)$. Since $\calA \subseteq \calR_{\calA} \subseteq \gotR_{\rk}$ and $\ol{\calA}^{\rk_{\calA}} = \gotR_{\rk} \cong \calM_K$ due to \cite[Theorems 4.7 and 4.9]{AC}, the result follows.
\end{proof}

Our strategy is to make use of our sequence $\{\calA_n\}_{n \geq 1}$ of approximating algebras from Section \ref{section-approx.crossed.product} to approximate $\calR_{\calA}$ in a suitable way. As already mentioned, in our present setting the algebras $\gotR_n$, $\gotR_{\infty}$ and $\gotR_{\rk}$ become $*$-regular algebras, and all the connecting maps in the commutative diagram \eqref{diagram-comm.diag.1} become $*$-homomorphisms. 
We denote by $\calR_n = \calR(\calA_n, \gotR_n)$ the $*$-regular closure of $\calA_n$ inside $\gotR_n$. Similarly, $\calR_{\infty} = \calR(\calA_{\infty}, \gotR_{\infty})$.

\begin{proposition}\label{proposition-lim.of.reg.closure}
We have inclusions $\calR_n \subseteq \calR_{n+1}$, and moreover $\bigcup_{n \geq 1} \calR_n = \calR_{\infty}$. Therefore the diagram \eqref{diagram-comm.diag.1} extends to a commutative diagram
\begin{equation}
\vcenter{
	\xymatrix{
	\calA_n \ar@{^{(}->}[r] \ar@{^{(}->}[d] & \calA_{n+1} \ar@{^{(}->}[d] \ar@{^{(}->}[r] & \calA_{n+2} \ar@{^{(}->}[d] \ar@{^{(}->}[r] & \cdots \ar@{^{(}->}[r] & \calA_{\infty} \ar@{^{(}->}[d] \ar@{^{(}->}[r] & \calA \ar@{^{(}->}[d] \\
	\calR_n \ar@{^{(}->}[r] \ar@{^{(}->}[d] & \calR_{n+1} \ar@{^{(}->}[d] \ar@{^{(}->}[r] & \calR_{n+2} \ar@{^{(}->}[d] \ar@{^{(}->}[r] & \cdots \ar@{^{(}->}[r] & \calR_{\infty} \ar@{^{(}->}[d] \ar@{^{(}->}[r] & \calR_{\calA} \ar@{^{(}->}[d] \\
	\gotR_n \ar@{^{(}->}[r] & \gotR_{n+1} \ar@{^{(}->}[r] & \gotR_{n+2} \ar@{^{(}->}[r] & \cdots \ar@{^{(}->}[r] & \gotR_{\infty} \ar@{^{(}->}[r] & \gotR_{\rk}.
	}
}\label{diagram-comm.diag.2}
\end{equation}
\end{proposition}
\begin{proof}
Since $\calA_n \subseteq \calA_{n+1} \cap \gotR_n \subseteq \calR_{n+1} \cap \gotR_n \subseteq \gotR_n$, and $\calR_{n+1} \cap \gotR_n$ is $*$-regular, we have $\calR_n \subseteq \calR_{n+1} \cap \gotR_n \subseteq \calR_{n+1}$. In particular, this shows the commutativity of the left sides of the diagram. The proof for the right sides is similar: $\calA_{\infty} \subseteq \calA \cap \gotR_{\infty} \subseteq \calR_{\calA} \cap \gotR_{\infty} \subseteq \gotR_{\infty}$, and since $\calR_{\calA} \cap \gotR_{\infty}$ is $*$-regular, we have again $\calR_{\infty} \subseteq \calR_{\calA} \cap \gotR_{\infty} \subseteq \calR_{\calA}$.

To prove the equality $\bigcup_{n \geq 1} \calR_n = \calR_{\infty}$, note that each $\calA_n \subseteq \calR_n \subseteq \bigcup_{n \geq 1} \calR_n \subseteq \gotR_{\infty}$, hence $\calA_{\infty} \subseteq \bigcup_{n \geq 1} \calR_n \subseteq \gotR_{\infty}$. It is easy to check, using that $\calR_n \subseteq \calR_{n+1}$, that $\bigcup_{n \geq 1} \calR_n$ is $*$-regular, so by definition $\calR_{\infty} \subseteq \bigcup_{n \geq 1} \calR_n$. The other inclusion is trivial because each $\calR_n \subseteq \calR_{\infty}$.
\end{proof}

The following lemma gives some examples of elements that appear inside $\calR_{\calA}$.

\begin{lemma}\label{lemma-invert.poly.RA}
Take $p(x) = \lambda_0 + \lambda_1 x + \cdots + \lambda_k x^k \in K[x]$ a polynomial with $\lambda_0 \neq 0$. Then $p(t) \in \calA$ is invertible in $\calR_{\calA}$. Moreover, $\calR_{\calA}$ contains a copy of the rational function field $K(t)$.
\end{lemma}
\begin{proof}
Inside $\gotR_{\rk}$ we identify the element $t \in \calA$ with $\lim_n \pi_n(\chi_{X \backslash E_n}t)$ (cf. Theorem \ref{theorem-rank.function}). Hence $p(t) = \lim_n  p(\pi_n(\chi_{X \backslash E_n}t))$. Note that $\pi_n(\chi_{X \backslash E_n}t) = ( h_W \cdot \chi_{X \backslash E_n}t)_W$. We compute
\begin{align*}
h_W \cdot \chi_{X \backslash E_n}t = e_{10}(W) + \cdots + e_{|W|-1,|W|-2}(W) =: u_W,
\end{align*}
so
$$p(\pi_n(\chi_{X \backslash E_n}t) ) = (\lambda_0 \Id_{W} + \lambda_1 u_W + \cdots + \lambda_k u_W^k)_W.$$
These are all lower triangular matrices inside each matrix algebra $M_{|W|}(K)$, and in fact invertible since $\lambda_0 \neq 0$. Hence $p(\pi_n(\chi_{X \backslash E_n}t) )$ is invertible inside $\gotR_n \subseteq \gotR_{\rk}$ for each $n \geq 1$, and so is its limit $\lim_n p(\pi_n(\chi_{X \backslash E_n}t) ) = p(t)$. Since $t$ is already invertible in $\calR_{\calA}$, it follows that $K(t) \subseteq \calR_{\calA}$.
\end{proof}

In \cite{AC3}, the $*$-regular closure of the odometer algebra (cf. Example \ref{example-odometer}) is explicitly computed.

\subsection{Difference between \texorpdfstring{$\calR_{\calA}$}{} and \texorpdfstring{$\calR_{\infty}$}{} assuming existence of a periodic point}\label{subsection-difference.RA.Rinfty}

We determine in this subsection the exact relationship between $\calR_{\infty}$ and $\calR_{\calA}$ in case $y$ is a periodic point. To start, we recall the following proposition from \cite{AC}, which determines how big is the subalgebra $\calA_{\infty}$ inside the algebra $\calA$ in this case of interest.

\begin{proposition}\cite[Proposition 4.5]{AC}\label{proposition-periodic.point.A}
Let us assume the above notation. Suppose that $y$ is a periodic point for $T$ with period $l$. Let $I$ be the ideal of $\calA$ generated by $C_{c,K}(X \backslash \{y, \dots, T^{l-1}(y)\})$. Then:
\begin{enumerate}[(i),leftmargin=1cm]
\item $I$ is also an ideal of $\calA_{\infty}$, and we have $*$-algebra isomorphisms
$$\calA / I \cong M_l(K[s,s^{-1}]), \qquad \calA_{\infty} / I \cong M_l(K).$$
\item There exists some $M \geq 0$ such that for each $n \geq M$ there is exactly one $W_n \in \V_n$ of length $l$ and containing $y$, and such that the isomorphism $h_{W_n} \calA_n \cong M_l(K)$ given in \eqref{equation-iso.central} coincides with the restriction of the projection map $q : \calA_{\infty} \ra \calA_{\infty}/I$ on $h_{W_n} \calA_n$. That is, the diagram
\begin{equation*}
\vcenter{
	\xymatrix{
	\calA_{\infty} \ar[r]^{q} & \calA_{\infty}/I \ar[d]^{\cong} \\
	h_{W_n} \calA_n \ar[r]^{\cong} \ar@{^{(}->}[u] & M_l(K)
	}
}\label{diagram-comm.diag.3}
\end{equation*}
commutes, where the right isomorphism comes from $(i)$. Moreover, $h_W \in I$ for all $W \in \V_n$, $W \neq W_n$.
\item $\calA_n / (I \cap \calA_n) \cong \calA_{\infty} / I \cong M_l(K)$ and $(1-h_{W_n}) \calA_n = I \cap \calA_n$ for every $n \geq M$.
\end{enumerate}
\end{proposition}

\noindent We continue with a proposition concerning the structure of $\calR_{\infty}$.

\begin{proposition}\label{proposition-periodic.point.Rinfty}
Let us assume the same notation as in Proposition \ref{proposition-periodic.point.A}. Let $\wt{I}$ be the ideal of $\calR_{\infty}$ generated by $I$. Then:
\begin{enumerate}[(i),leftmargin=1cm]
\item $\wt{I} = \bigcup_{n \geq M} (1 - h_{W_n}) \calR_n$, and there is a $*$-isomorphism
$$\calR_{\infty} / \wt{I} \cong M_l(K).$$
\item If $\calR$ denotes the $*$-subalgebra of $\calR_{\calA}$ generated by $\wt{I}, h_{W_M}\calA_M$ and $K[t,t^{-1}]$, then $\wt{I}$ is also an ideal of $\calR$, $\calA$ is contained in $\calR$, and there is a $*$-isomorphism
$$\calR / \wt{I} \cong M_l(K[t^l,t^{-l}]).$$
\end{enumerate}
\end{proposition}

\begin{remark}\label{remark-periodic.point.Rinfty}
Since the ideal $\wt{I}$ is already $*$-regular and the quotient $\calR / \wt{I}$ is close to be $*$-regular, the $*$-subalgebra $\calR$ is not the $*$-regular closure $\calR_{\calA}$, but is a good approximation to it. We will see later how to enlarge $\calR$ in order to obtain the whole $\calR_{\calA}$. First, we prove Proposition \ref{proposition-periodic.point.Rinfty}.
\end{remark}

\begin{proof}[Proof of Proposition \ref{proposition-periodic.point.Rinfty}]
$(i)$ For $n \geq M$, let $I_n = I \cap \calA_n$ and let $\wt{I}_n$ be the ideal of $\calR_n$ generated by $I_n$.
\begin{enumerate}[label=\underline{Claim \arabic*}:,leftmargin=*,labelindent=0em]
\item $\wt{I} = \bigcup_{n \geq M} \wt{I}_n$. 
\begin{spaceleft}
\underline{Proof:} Clearly each $\wt{I}_n \subseteq \wt{I}$, so $\bigcup_{n \geq M} \wt{I}_n \subseteq \wt{I}$. For the other inclusion, recall that $I \subseteq \calA_{\infty}$ by Proposition \ref{proposition-periodic.point.A}. If we take $a$ to be an element of $\wt{I}$, then $a = \sum_{j=1}^m r_j b_j s_j$ for some $r_j, s_j \in \calR_{\infty}$ and $b_j \in I \subseteq \calA_{\infty}$. There exists then an index $n_0 \geq M$ such that $r_j, s_j \in \calR_{n_0}$ and $b_j \in I \cap \calA_{n_0} = I_{n_0}$ for all $1 \leq j \leq m$. Therefore $a = \sum_{j=1}^m r_j b_j s_j \in \calR_{n_0} I_{n_0} \calR_{n_0} = \wt{I}_{n_0}$, and we obtain the inclusion $\wt{I} \subseteq \bigcup_{n \geq M}\wt{I}_n$.\quad\qedsymbol
\end{spaceleft}

\item $\wt{I}_n = (1 - h_{W_n}) \calR_n$.
\begin{spaceleft}
\underline{Proof:} Since $I_n = (1-h_{W_n})\calA_n$ due to Proposition \ref{proposition-periodic.point.A}, and taking into account that $(1 - h_{W_n})$ is central in $\calR_n$, we compute
$$\wt{I}_n = \calR_n I_n \calR_n = \calR_n (1 - h_{W_n}) \calA_n \calR_n = \calR_n (1 - h_{W_n}) \calR_n = (1 - h_{W_n}) \calR_n,$$
as required.\quad\qedsymbol
\end{spaceleft}
\end{enumerate}
Using Claims $1$ and $2$, we get $\wt{I} = \bigcup_{n \geq M}\wt{I}_n = \bigcup_{n \geq M}(1-h_{W_n})\calR_n$. In particular, we see that the ideal $\wt{I}$ is $*$-regular.
\begin{enumerate}[label=\underline{Claim 3}:,leftmargin=*,labelindent=0em]
\item For $m \geq n \geq M$, we have isomorphisms $\calR_n / \wt{I}_n \cong \calR_m / \wt{I}_m \cong M_l(K)$, via
$$e_{ij}(W_n) + \wt{I}_n \mapsto e_{ij}(W_m) + \wt{I}_m \mapsto e_{ij}.$$
\begin{spaceleft}
\underline{Proof:} To see this, note first that $h_{W_n}\calR_n \cong M_l(K)$ through $e_{ij}(W_n) \mapsto e_{ij}$, since $\calA_n \subseteq \calR_n \subseteq \gotR_n$ and $h_{W_n} \calA_n = h_{W_n} \gotR_n \cong M_l(K)$. Now each $h_{W_n}$ is a central idempotent in $\calR_n$, so we have decompositions
$$\calR_n = (1-h_{W_n})\calR_n \oplus h_{W_n}\calR_n = \wt{I}_n \oplus h_{W_n} \calR_n.$$
Hence $\calR_n / \wt{I}_n \cong h_{W_n}\calR_n \cong M_l(K) \cong h_{W_m} \calR_m \cong \calR_m / \wt{I}_m$ through the cited maps.\quad\qedsymbol 
\end{spaceleft}
\end{enumerate}
Fix $n \geq M$ and consider the composition $\calR_n \hookrightarrow \calR_{\infty} \ra \calR_{\infty} / \wt{I}$. Since $\wt{I}_n \subseteq \wt{I}$, we get a $*$-homomorphism
$$\calR_n / \wt{I}_n \ra \calR_{\infty} / \wt{I}, \quad r + \wt{I}_n \mapsto r + \wt{I}.$$
From Claim $3$ it follows easily that, for $n, m \geq M$, the diagram
\begin{equation*}
\xymatrix{
\calR_n / \wt{I}_n \ar[d]_*[@!90]{\cong} \ar[r] & \calR_{\infty} / \wt{I} \\
\calR_m / \wt{I}_m \ar[ru] &
}
\end{equation*}
is commutative. This proves surjectivity of $\calR_n / \wt{I}_n \ra \calR_{\infty} / \wt{I}$, and injectivity also follows from this and Claim $1$. Thus we obtain the desired $*$-isomorphism $\calR_{\infty} / \wt{I} \cong \calR_n / \wt{I}_n \cong M_l(K)$.

$(ii)$ We show first that $\wt{I}$ is stable under multiplication by elements of $K[t,t^{-1}]$.

\begin{enumerate}[label=\underline{Claim 4}:,leftmargin=*,labelindent=0em]
\item $t \wt{I} = \wt{I}$.
\begin{spaceleft}
\underline{Proof:} Let's prove the inclusion $t \wt{I} \subseteq \wt{I}$, so take $a \in (1 - h_{W_n}) \calR_n$ for some $n \geq M$. Since $a = (1 - h_{W_n})a$, it is enough to show that $t(1 - h_{W_n}) \in \wt{I}$. But
$$t(1 - h_{W_n}) = t \chi_{X \backslash (W_n \cup T(W_n) \cup \cdots \cup T^{l-1}(W_n))} = \chi_{X \backslash (T(W_n) \cup T^2(W_n) \cup \cdots \cup T^l(W_n))}t$$
and $\chi_{X \backslash (T(W_n) \cup T^2(W_n) \cup \cdots \cup T^l(W_n))} \in C_{c,K}(X \backslash \{y\})$, so by the description of $\calA_{\infty}$ given in Section \ref{section-approx.crossed.product} we deduce that $t(1 - h_{W_n}) \in \calA_{\infty}$, hence $t(1 - h_{W_n}) \in \wt{I}$. To show the other inclusion it is enough to show that $t^{-1} \wt{I} \subseteq \wt{I}$, which in turn will follow once we show that $t^{-1}(1-h_{W_n}) \in \wt{I}$. This is obvious because $1 - h_{W_n} \in C_{c,K}(X \backslash \{y\})$.\quad\qedsymbol 

\end{spaceleft}
\end{enumerate}

\noindent As a consequence, we have $p(t) \wt{I}, \wt{I} p(t) \subseteq \wt{I}$ for any Laurent polynomial $p(t)$ in $t$, so $\wt{I}$ is an ideal of $\calR$.

\begin{enumerate}[label=\underline{Claim 5}:,leftmargin=*,labelindent=0em]
\item $\wt{I}$ is a proper ideal of $\calR$.
\begin{spaceleft}
\underline{Proof:} This follows from the fact that $\rk_{\calA}(1 - h_{W_n}) < 1$ for all $n \geq M$.\quad\qedsymbol
\end{spaceleft}
\end{enumerate}

\begin{enumerate}[label=\underline{Claim 6}:,leftmargin=*,labelindent=0em]
\item $\calA \subseteq \calR$.
\begin{spaceleft}
\underline{Proof:} We show by induction that $h_{W_n} \calA_n \subseteq \calR$ for all $n \geq M$. For $n = M$, this follows from the definition of $\calR$. Now assume that $h_{W_n} \calA_n \subseteq \calR$ for some $n \geq M$. Under the quotient map $\calA_{\infty} \ra \calA_{\infty} / I \cong M_l(K)$ the matrix units $e_{ij}(W_n)$ correspond to the matrix units $e_{ij}$ (Proposition \ref{proposition-periodic.point.A}), hence the differences $e_{ij}(W_{n+1}) - e_{ij}(W_n)$ belong to $I \subseteq \wt{I} \subseteq \calR$ for all $0 \leq i,j \leq l-1$. Since $e_{ij}(W_n) \in h_{W_n}\calA_n \subseteq \calR$, we deduce that $e_{ij}(W_{n+1}) \in \calR$, and so the whole algebra $h_{W_{n+1}}\calA_{n+1}$ lies inside $\calR$. Therefore $h_{W_n}\calA_n \subseteq \calR$ for all $n \geq M$. Hence
$$\calA_n = (1 - h_{W_n})\calA_n \oplus h_{W_n}\calA_n \subseteq \wt{I} + h_{W_n} \calA_n \subseteq \calR$$
for all $n \geq M$, so $\calA_{\infty} \subseteq \calR$. In particular $C_K(X) \subseteq \calR$, and since $K[t,t^{-1}] \subseteq \calR$ already, we obtain $\calA \subseteq \calR$, as claimed.\quad\qedsymbol
\end{spaceleft}
\end{enumerate}

\begin{enumerate}[label=\underline{Claim 7}:,leftmargin=*,labelindent=0em]
\item We have a $*$-isomorphism $\calR / \wt{I} \cong M_l(K[t^l,t^{-l}])$.
\begin{spaceleft}
\underline{Proof:} Since $1 - h_{W_M} \in \wt{I}$, the family $\{ e_{ij}(W_M) + \wt{I} \mid 0 \leq i,j \leq l-1 \}$ is a complete system of matrix units for $\calR / \wt{I}$, so there is an isomorphism $\calR / \wt{I} \cong M_l(T)$, being $T$ the centralizer of the family $\{ e_{ij}(W_M) + \wt{I} \mid 0 \leq i,j \leq l-1 \}$ in $\calR / \wt{I}$. The isomorphism is given explicitly by
$$s \mapsto \sum_{i,j=0}^{l-1} s_{ij} e_{ij}, \quad \text{ with } \quad s_{ij} = \sum_{k=0}^{l-1} e_{ki}(W_M) \cdot s \cdot e_{jk}(W_M) \in T,$$
which is also a $*$-isomorphism. We thus only need to prove that $T = K[t^l,t^{-l}]$. The inclusion $K[t^l,t^{-l}] \subseteq T$ is clear since, using that $e_{ij}(W) = \chi_{T^i(W)}t^{i-j}$ for $0 \leq i,j \leq |W|-1$, we get
$$t^l e_{ij}(W_M) - e_{ij}(W_M)t^l = ( \chi_{T^{i+l}(W_M)} - \chi_{T^i(W_M)} ) t^{i-j+l}$$
which belongs to $I \subseteq \wt{I}$ due to the fact that $\chi_{T^{i+l}(W_M)} - \chi_{T^i(W_M)}\in C_{c,K}(X \backslash \{y, \dots, T^{l-1}(y)\})$ since $y$ is a periodic point of period $l$ which belongs to $W_M$ (see Proposition \ref{proposition-periodic.point.A}). Therefore $\calR / \wt{I} \cong M_l(T) \supseteq M_l(K[t^l,t^{-l}])$. In order to prove the equality it is enough to check that the element $t + \wt{I} \in \calR/\wt{I}$ belongs to $M_l(K[t^l,t^{-l}])$ under the previous isomorphism. Note that we can write
$$t + \wt{I} = t h_{W_M} + \wt{I} = \sum_{i=0}^{l-2} e_{i+1,i}(W_M) + t^l e_{0,l-1}(W_M) + \wt{I}$$
which is mapped to the element $\sum_{i=0}^{l-2} e_{i+1,i} + t^l e_{0,l-1}$ under the previous isomorphism, and so it belongs to $M_l(K[t^l,t^{-l}])$.
We have thus obtained the desired $*$-isomorphism $\calR / \wt{I} \cong M_l(K[t^l,t^{-l}])$.\quad\qedsymbol
\end{spaceleft}
\end{enumerate}
This concludes the proof of the proposition.
\end{proof}

We now want to describe the $*$-regular closure $\calR_{\calA}$ of $\calA$ in $\gotR_{\rk}$ in terms of the $*$-algebra $\calR$ introduced in Proposition \ref{proposition-periodic.point.Rinfty}(ii). For that we need a couple of technical lemmas, together with a definition.

\begin{definition}\label{definition-local.unit}
Let $R$ be a non-unital ring. We say that a family $E \subseteq R$ of idempotents is a \textit{left local unit for $R$} if for every $r_1,...,r_n \in R$ there exists an idempotent $e \in E$ such that
$$e r_i = r_i \quad \text{for all } 1 \leq i \leq n.$$
The concept of \textit{right local unit} is defined analogously. A \textit{local unit} will be a right and left local unit.
\end{definition}
Note that, in the case that $R$ is a ring endowed with an involution $*$ and $E$ is a left local unit for $R$, then $E^* = \{e^* \mid e \in E\}$ is a right local unit for $R$.\\

Recall that $\wt{I}$ is the ideal of $\calR_{\infty}$ generated by $I$. We write $\calS_0$ for the $*$-subalgebra of $\calR_{\calA}$ generated by $\wt{I}, h_{W_M}\calA_M$ and $K(t)$ (compare with $\calR$ from Proposition \ref{proposition-periodic.point.Rinfty}). It may be the case that $\wt{I}$ is not an ideal of $\calS_0$ anymore; nevertheless, we have the following result.

\begin{lemma}\label{lemma-local.unit.ideal}
Denote by $\ol{I}_0$ the ideal of $\calS_0$ generated by $\wt{I}$, and consider
$$E = \{ p(t^l)^{-1}(1-h_{W_n})p(t^l) \in \ol{I}_0 \mid p(t) \in K[t] \backslash \{0\}, n \geq M\}.$$
Then $E$ is a left local unit for $\ol{I}_0$.
\end{lemma}
\begin{proof}
Since $\wt{I}$ is closed under the involution, it follows that $\ol{I}_0$ is a $*$-ideal of $\calS_0$.

Note that every element of $\ol{I}_0$ is a sum of elements of the form
\begin{equation}\label{Ch2-equation-element.bigprod}
p_1(t) q_1(t)^{-1} e_{i_1,j_1}(W_M) \cdots p_s(t) q_s(t)^{-1} e_{i_s,j_s}(W_M) p_{s+1}(t) q_{s+1}(t)^{-1} (1 - h_{W_n}) y,
\end{equation}
where $p_k,q_k \in K[t]\backslash \{0\}$, $0 \leq i_k,j_k \leq l-1$, $n \geq M$ and $y \in \calS_0$. Since $\wt{I}$ is stable under multiplication by $K[t,t^{-1}]$, the product $p_{s+1}(t) (1 - h_{W_n})$ belongs to $\wt{I}$, so we can assume that $p_{s+1}(t) = 1$.

\begin{enumerate}[label=\underline{Claim}:,leftmargin=*,labelindent=0em]
\item Each element of the form \eqref{Ch2-equation-element.bigprod} can be further written as a sum of elements of the form
$$q(t^l)^{-1} (1 - h_{W_n}) \wt{y} \quad \text{ for some } q \in K[t^l] \backslash \{0\}, n \geq M \text{ and } \wt{y} \in \calS_0.$$
\begin{spaceleft}
\underline{Proof:} Since the field extension $K(t) / K(t^l)$ has degree $l$, with basis $\{1,t,...,t^{l-1}\}$, we can write $q_{s+1}(t)^{-1}$ as
$$q_{s+1}(t)^{-1} = \sum_{i = 0}^N t^i g_i(t^l)^{-1}$$
for some $N \geq 0$ and polynomials $g_i \in K[t^l] \backslash \{0\}$. Thus we can assume that $q_{s+1}$ is a polynomial in $t^l$.

Recall that, modulo the ideal $\wt{I}$, the matrix units $e_{ij}(W_M)$ commute with the element $t^l$. As a consequence the element $b_s := q_{s+1}(t^l) e_{i_s,j_s}(W_M) - e_{i_s,j_s}(W_M) q_{s+1}(t^l)$ belongs to $\wt{I}$, so there exists an integer $n_s \geq M$ such that $b_s = (1 - h_{W_{n_s}}) b_s$. Therefore
$$e_{i_s,j_s}(W_M) q_{s+1}(t^l)^{-1} - q_{s+1}(t^l)^{-1} e_{i_s,j_s}(W_M) = q_{s+1}(t^l)^{-1}(1 - h_{W_{n_s}}) b_s q_{s+1}(t^l)^{-1},$$
so that
\begin{align*}
p_1(t) q_1(t&)^{-1} e_{i_1,j_1}(W_M) \cdots p_s(t) q_s(t)^{-1} e_{i_s,j_s}(W_M) q_{s+1}(t^l)^{-1}(1 - h_{W_n}) y \\
& = p_1(t) q_1(t)^{-1} e_{i_1,j_1}(W_M) \cdots p_s(t) q_s(t)^{-1} q_{s+1}(t^l)^{-1} e_{i_s,j_s}(W_M) (1 - h_{W_n}) y \\
& \quad + p_1(t) q_1(t)^{-1} e_{i_1,j_1}(W_M) \cdots p_s(t) q_s(t)^{-1} q_{s+1}(t^l)^{-1} (1 - h_{W_{n_s}}) b_s q_{s+1}(t^l)^{-1} (1 - h_{W_n}) y.
\end{align*}
Since $e_{i_s,j_s}(W_M) \in \calR_M \subseteq \calR_n$ and $1 - h_{W_n}$ is central in $\calR_n$, the first term becomes
$$p_1(t) q_1(t)^{-1} e_{i_1,j_1}(W_M) \cdots p_s(t) \wt{q}_s(t)^{-1} (1 - h_{W_n}) y'$$
with $\wt{q}_s(t) = q_s(t) q_{s+1}(t^l) \in K[t] \backslash \{0\}$ and $y' = e_{i_s,j_s}(W_M) y \in \calS_0$, and the second term becomes
$$p_1(t) q_1(t)^{-1} e_{i_1,j_1}(W_M) \cdots p_s(t) \wt{q}_s(t)^{-1} (1 - h_{W_{n_s}}) y''$$
with now $y'' = b_s q_{s+1}(t^l)^{-1} (1 - h_{W_n}) y \in \calS_0$. Again, due to the fact that $K[t,t^{-1}] \wt{I} \subseteq \wt{I}$, we can assume that $p_s = 1$ in each of these terms. Now the claim follows
by induction on $s$.\quad\qedsymbol
\end{spaceleft}
\end{enumerate}
Let now $x_1,...,x_n \in \ol{I}_0$. By the above claim, we can assume that each $x_i$ is a monomial of the form $q_i(t^l)^{-1}(1 - h_{W_{n_i}}) y_i$ with $q_i \in K[t^l] \backslash \{0\}$, $n_i \geq M$ and 
$y_i \in \calS_0$. Consider the polynomial $q := q_1 \cdots q_n \in K[t^l] \backslash \{ 0 \}$. We see that, for each $1 \leq i \leq n$, the result of multiplying $x_i$ by $q$ to the left is always an element of the form $\wt{x}_i y_i$, where $\wt{x}_i \in \wt{I}$. Therefore there exists $N \geq M$ such that $(1 - h_{W_N}) q(t^l)x_i = q(t^l)x_i$ for all $1 \leq i \leq n$. The lemma follows by taking the idempotent $e := q(t^l)^{-1}(1 - h_{W_N})q(t^l)$.
\end{proof}

As a consequence of Lemma \ref{lemma-local.unit.ideal}, the ideal $\ol{I}_0$ must be a proper ideal of $\calS_0$, since for $e \in E$ we have $\rk_{\calR_{\calA}}(e) < 1$.

At this moment we could argue as in the proof of Proposition \ref{proposition-periodic.point.Rinfty} and compute the quotient $\calS_0/\ol{I}_0$. It turns out that this quotient is $*$-isomorphic to $M_l(K(t^l))$, which is a $*$-regular ring. The problem we encounter now is that the ideal $\ol{I}_0$ may not be $*$-regular. To fix this, we consider the non-unital subalgebra of $\calR_{\calA}$ generated by $\ol{I}_0$ and the relative inverses $\ol{x}$ of elements $x \in \ol{I}_0$, denoted by $\ol{I}_1$. It is in fact a $*$-subalgebra because of the equality $\ol{x^*} = \ol{x}^*$.

From now on, we let $\pazP$ be the set of all the left projections ${\rm LP}(e)$, for $e \in E$. So for each $p \in \mathcal P$ there is an idempotent $e \in E$ such that $p = {\rm LP}(e)$; in particular $ep = p$ and $pe = e$. Note that $\mathcal P \subseteq \ol{I}_1$.

\begin{lemma}\label{lemma-ideal.with.rel.inv}
The following statements hold:
\begin{enumerate}[i),leftmargin=1cm]
\item The set $\pazP$ is a local unit for $\ol{I}_1$.
\item If $\calS_1$ denotes the $*$-subalgebra of $\calR_{\calA}$ generated by $\ol{I}_1, h_{W_M} \calA_M$ and $K(t)$, then $\ol{I}_1$ is a proper ideal of $\calS_1$, and there is a $*$-isomorphism
$$\calS_1 / \ol{I}_1 \cong M_l(K(t^l)).$$
\end{enumerate}
\end{lemma}
\begin{proof}
For $i)$, let $x_1,...,x_n \in \ol{I}_1$. We can assume that each $x_i$ is a monomial of one of the forms

\begin{align*}
& (\text{I}) \quad r_1 \ol{r_2} \cdots \quad \text{ with } r_i \in \ol{I}_0; \qquad (\text{II}) \quad \ol{r_1} r_2 \cdots \quad \text{ with } r_i \in \ol{I}_0.
\end{align*}
Consider the sets
$$J_1 = \{r \in \ol{I}_0 \mid r \text{ appears as a first term in one of the } x_i \},$$
$$J_2 = \{r^* \in \ol{I}_0 \mid \ol{r} \text{ appears as a first term in one of the } x_i \},$$
so that $J = J_1 \cup J_2$ is a finite subset of $\ol{I}_0$. By Lemma \ref{lemma-local.unit.ideal} there exists an idempotent $e \in E$ such that $er = r$ for all $r \in J$. Take $p = {\rm LP}(e) \in \mathcal P$, so $pe = e$ and $ep = p$. Now for an element $r \in J_1$, we compute
$$pr = per = er = r.$$
Also for an element $r \in \ol{I}_0$ such that $r^* \in J_2$, we compute $pr^* = per^* = er^* = r^*$, so by taking $*$ we have $rp = r$. Multiplying to the left by the relative inverse $\ol{r}$ we get $(\ol{r}r)p = \ol{r}r$, which is a projection. Hence $\ol{r}r = (\ol{r}r)^* = (\ol{r}rp)^* = p\ol{r}r$, and
$$p\ol{r} = p\ol{r}r\ol{r} = \ol{r}r\ol{r} = \ol{r}.$$
We deduce from these computations that $px_i = x_i$ for all $1 \leq i \leq n$. Since $\pazP = \pazP^*$, part $i)$ follows.

$ii)$. By $i)$, it is immediate to check that $\ol{I}_1 \subseteq \calS_1$ is proper. To prove that it is an ideal, it is enough to show that $p(t) \ol{I}_1, \ol{I}_1 p(t) \subseteq \ol{I}_1$ for all $p(t) \in K(t)$ and that $e_{ij}(W_M)\ol{I}_1, \ol{I}_1e_{ij}(W_M) \subseteq \ol{I}_1$ for all $0 \leq i,j \leq l-1$. By taking $*$, we only need to show that $p(t) \ol{I}_1, e_{ij}(W_M) \ol{I}_1 \subseteq \ol{I}_1$.

Let $p(t) \in K(t)$ and $a \in \ol{I}_1$. We can assume that $a$ is a monomial of the form either (I) or (II). In the first case, $a = ra'$ for some $r \in \ol{I}_0$ and $a' \in \ol{I}_1$; then $p(t) a = p(t) r a' \in \ol{I}_1$ since $p(t) r \in \ol{I}_0$. In the second case, $a = \ol{r}a'$ for $r \in \ol{I}_0$ and $a' \in \ol{I}_1$. Consider $p \in \pazP$ such that $p \ol{r} = \ol{r}$ and $e \in E$ such that $p = {\rm LP}(e)$. Then $p(t) e \in \ol{I}_0$, so that
$$p(t) a = p(t) \ol{r} a' = p(t) p \ol{r} a' = p(t) e p \ol{r} a' = p(t) e a \in \ol{I}_1,$$
as required. Similar computations can be used to show that $e_{ij}(W_M) \ol{I}_1 \subseteq \ol{I}_1$.

The rest of the proof follows exactly the same arguments as in the proof of Proposition \ref{proposition-periodic.point.Rinfty}.
\end{proof}

We are now ready to determine the $*$-regular closure $\calR_{\calA}$.

\begin{theorem}\label{theorem-periodic.point.RA}
Following the previous assumptions and caveats, we define $\ol{I}_m$ to be the non-unital subalgebra of $\calR_{\calA}$ generated by $\ol{I}_{m-1}$ and the relative inverses of elements of $\ol{I}_{m-1}$, starting from our $\ol{I}_0$. Let also $\calS_m$ be the $*$-subalgebra of $\calR_{\calA}$ generated by $\ol{I}_m, h_{W_M} \calA_M$ and $K(t)$. Then:
\begin{enumerate}[(1),leftmargin=1cm]
\item $\ol{I}_m$ admits the set $\pazP$ as a local unit;
\item $\ol{I}_m$ is a proper $*$-ideal of $\calS_m$, and $\calS_m / \ol{I}_m \cong M_l(K(t^l))$ for all $m \geq 0$;
\item $\ol{I}_{\infty} = \bigcup_{m \geq 0} \ol{I}_m$ is a proper $*$-regular ideal of $\calR_{\calA}$, and
$$\calR_{\calA} / \ol{I}_{\infty} \cong M_l(K(t^l)).$$
Moreover, $\calR_{\calA}$ is generated as a $*$-algebra by $\ol{I}_{\infty}, h_{W_M} \calA_M$ and $K(t)$.
\end{enumerate}
\end{theorem}
\begin{proof}
We observe that each $\ol{I}_m$ is also a $*$-subalgebra of $\calR_{\calA}$. $(1)$ and $(2)$ follows easily by induction, taking into account that the same arguments as in the proof of Lemma \ref{lemma-ideal.with.rel.inv} apply here.

For $(3)$, let $\calS_{\infty}$ be the $*$-subalgebra of $\calR_{\calA}$ generated by $\ol{I}_{\infty}, h_{W_M} \calA_M$ and $K(t)$. Clearly $\ol{I}_{\infty} \subseteq \calS_{\infty}$ is proper since each $\ol{I}_m \subseteq \calS_m$ is so. To prove that $\ol{I}_{\infty}$ is an ideal of $\calS_{\infty}$ it is enough to show that $K(t) \ol{I}_{\infty} \subseteq \ol{I}_{\infty}$ and that $e_{ij}(W_M) \ol{I}_{\infty} \subseteq \ol{I}_{\infty}$ for all $0 \leq i,j \leq l-1$. For the first inclusion, take $p(t) \in K(t)$ and $a \in \ol{I}_{\infty}$. Then $a \in \ol{I}_m$ for some $m \geq 0$, so by $(2)$ we have $p(t)a \in \ol{I}_m \subseteq \ol{I}_{\infty}$. The second inclusion is obtained analogously.

By construction of our sequence $\{\ol{I}_m\}_{m \geq 0}$, it is straightforward to show that $\ol{I}_{\infty}$ is $*$-regular too. 
Therefore $\ol{I}_{\infty}$ is a $*$-regular ideal of $\calS_{\infty}$ and, just as before, its quotient $\calS_{\infty} / \ol{I}_{\infty}$ is $*$-isomorphic to $M_l(K(t^l))$, which is $*$-regular. It follows from \cite[Lemma 1.3]{Goo91} that $\calS_{\infty}$ is $*$-regular. Since $\calA_{\infty} \subseteq \calS_{\infty}$ by Claim 6 of Proposition \ref{proposition-periodic.point.Rinfty} and $t \in \calS_{\infty}$, we get $\calA \subseteq \calS_{\infty} \subseteq \calR_{\calA}$. We conclude that $\calS_{\infty} = \calR_{\calA}$.
\end{proof}

In conclusion, in the case that there exists a periodic point $y \in X$ of finite period $l$, we have been able to determine part of the ideal structure of the $*$-regular closure $\calR_{\calA}$: for each such point $y \in X$ one can apply the above process to construct a \textit{maximal} ideal $\ol{I}_{\infty}(y)$ of $\calR_{\calA}$, thus proving that, in particular, $\calR_{\calA}$ is not simple. In fact, the construction of the ideal $\ol{I}_{\infty}(y)$ not only depends on the periodic point $y \in X$, but on the whole orbit $\calO(y) = \{y,T(y),...,T^{l-1}(y)\}$. This defines a correspondence
$$\calO(y) \mapsto \ol{I}_{\infty}(y)$$
between the whole set of orbits of periodic points in $X$ and maximal ideals of $\calR_{\calA}$. The next lemma shows that this correspondence is in fact injective.

\begin{lemma}\label{lemma-orbits.ideals}
Let $x, y \in X$ be two periodic points of periods $l_1,l_2$, respectively (not necessarily equal). Suppose that $x \notin \calO(y)$. Then the maximal ideals $\ol{I}_{\infty}(x)$ and $\ol{I}_{\infty}(y)$ of $\calR_{\calA}$ are different.
\end{lemma}
\begin{proof}
Since $\calO(x) \cap \calO(y) = \emptyset$, we can find a clopen subset $U \subseteq X$ such that $\calO(x) \cap U = \emptyset$ but $\calO(y) \cap U \neq \emptyset$. We can in fact assume that $y \in U$.

Since $\calO(x) \cap U = \emptyset$, we have $\chi_U \in C_{c,K}(X \backslash \{x,...,T^{l_1-1}(x)\}) \subseteq \ol{I}_{\infty}(x)$. Assume for contradiction that $\ol{I}_{\infty}(x) = \ol{I}_{\infty}(y)$, so $\chi_U \in \ol{I}_{\infty}(y)$. By (1) of Theorem \ref{theorem-periodic.point.RA} there exists $p \in \pazP$ satisfying $p \chi_U = \chi_U$. Hence we can find a non-zero polynomial $p(t) \in K[t]$ and $N \geq M$ such that $p = {\rm LP}(e)$ with $e = p(t^{l_2})^{-1}(1-h_{W_N})p(t^{l_2})$. Here the collection $\{W_n\}_{n \geq M}$ is taken with respect to the point $y \in X$ (see Proposition \ref{proposition-periodic.point.A}). In particular $e \chi_U = e p \chi_U = p \chi_U = \chi_U$, so $h_{W_N} p(t^{l_2}) \chi_U = 0$.

Write $p(x) = a_0 + a_1 x + \cdots + a_m x^m$ for some $a_i \in K$. If we consider the non-empty clopen set $V := U \cap T^{l_2}(U) \cap \cdots \cap T^{l_2 \cdot m}(U)$, we obtain
\begin{align*}
0 & = \chi_V \cdot h_{W_N} p(t^{l_2}) \chi_U \\
& = \chi_V \cdot (a_0 h_{W_N} \chi_U + a_1 h_{W_N} \chi_{T^{l_2}(U)} t^{l_2} + \cdots + a_m h_{W_N} \chi_{T^{l_2 \cdot m}(U)} t^{l_2 \cdot m}) \\
& = a_0 h_{W_N} \chi_{V} + a_1 h_{W_N} \chi_{V} t^{l_2} + \cdots + a_m h_{W_N} \chi_{V} t^{l_2 \cdot m} \\
& = h_{W_N} \chi_{V} p(t^{l_2}).
\end{align*}
Since $p(t^{l_2})$ is invertible inside $\calR_{\calA}$, necessarily $h_{W_N} \chi_{V} = 0$. This is a contradiction because $y \in W_N \cap V$.
\end{proof}

It is therefore reasonable to think that, in order to uncover the whole structure of $\calR_{\calA}$ in the case of existence of a periodic point, it is crucial to understand the structure of the ideals $\ol{I}_{\infty}$, and in particular the structure of $\calR_{\infty} = \bigcup_{n \geq 1} \calR_n$, which in turn can be studied by studying their pieces $\calR_n$. Therefore, in the next section we will concentrate on uncovering part of the structure of the $*$-regular closure $\calR_n$, for a fixed $n$.


\subsection{The \texorpdfstring{$*$}{}-regular closure \texorpdfstring{$\calR_{\calB}$}{}}\label{subsection-study.RB}

We return to the general setting we had in Section \ref{section-approx.crossed.product}, with the extra hypothesis that $K$ is now a field with a positive definite involution $\lambda \mapsto \ol{\lambda}$. We fix a clopen subset $E$ of $X$ and a partition $\calP$ of $X \backslash E$ into clopen subsets. Recall that $\calB$ denotes the unital $*$-subalgebra of $\calA$ generated by the partial isometries $\{\chi_Z t\}_{Z \in \calP}$, and we write $\calB = \bigoplus_{i \in \Z} \calB_i t^i$ with $\calB_0 = C_K(X) \cap \calB$,
$$\calB_i = \chi_{X \setminus (E \cup T(E) \cup \cdots \cup T^{i-1}(E))} \calB_0 \quad \text{and} \quad \calB_{-i} = \chi_{X \setminus (T^{-1}(E) \cup \cdots \cup T^{-i}(E))} \calB_0 \qquad \text{for } i > 0.$$
We also write $\pi$ for the map $\pi \colon \calB \to \gotR$ given by $\pi(b) = (h_W \cdot b)_W$, where $\gotR = \prod_{W \in \V} M_{|W|}(K)$.

We aim to follow the same steps as in \cite[Section 6]{AG} to study the $*$-regular closure $\calR_{\calB} := \calR(\calB,\gotR)$\footnote{Note that, in the notation used in Section \ref{section-*.regular.closure}, $\calR_{\calB} = \calR_n$ in case $\calB$ is one of the $*$-subalgebras $\calA_n$.}. However, the situation here is much more involved, and we are only able to determine a (large) $*$-subalgebra of $\calR_{\calB}$.

The first step is to consider, from $\calB$, a skew partial power series ring $\calB_0[[t;T]]$ by taking infinite formal sums
$$\sum_{i \geq 0} b_i (\chi_{X \backslash E} t)^i = \sum_{i \geq 0} b_i t^i , \text{ where } b_i \in \calB_i \text{ for all } i \geq 0.$$
It is worth to point out here that, for $i>0$, the coefficients $b_i$ are restricted to belong to the generally proper ideal $\calB_i$ of $\calB_0$. Similarly we can consider $\calB_0 [[t^{-1}; T^{-1}]]$. Now, given a $W \in \V$, only a finite number of terms in the infinite sum $\sum_{i \geq 0} b_i t^i$ can be non-zero in the factor corresponding to $W$, since the product $h_W \cdot (\chi_{X \backslash E} t)^i$ is exactly zero for $i \geq |W|$. We have a similar situation for $\calB_0 [[t^{-1}; T^{-1}]]$. In this way we obtain faithful representations
$$ \pi_{+} : \calB_0 [[t;T]] \ra \gotR, \quad b \mapsto (h_W \cdot b)_W \quad \text{and} \quad \pi_{-} : \calB_0 [[t^{- 1};T^{- 1}]] \ra \gotR, \quad b \mapsto (h_W \cdot b)_W$$
by lower (resp. upper) triangular matrices. We will be mainly interested in the first one $\pi_+$.

We have the following key property.

\begin{lemma}\label{lemma-invertibility.PS}
Let $x=\sum _{i \geq 0} b_i t^i \in \calB_0[[t;T]]$. Then $x$ is invertible in $\calB_0[[t;T]]$ if and only if $b_0$ is invertible in $\calB_0 $. Analogously for the elements of $\calB_0[[t^{-1},T^{-1}]]$.
\end{lemma}
\begin{proof}
Assume first that $x = \sum_{i \geq 0} b_i t^i$ is invertible in $\calB_0[[t;T]]$. There exists then $y = \sum_{i \geq 0} b'_i t^i$ in $\calB_0[[t;T]]$ such that $xy = yx = 1$. In particular $b_0 b'_0 = 1$, and so $b_0$ is invertible in $\calB_0$.

Conversely, assume that $b_0$ is invertible in $\calB_0$. We can then assume that $b_0 = 1$, so that $x =1-y$, where the order of $y$ in $t$ is greater than or equal to $1$. We then have
$$x^{-1}= (1-y)^{-1} = 1 + y + y^2 + \cdots \in \calB_0[[t;T]],$$
hence $x$ is invertible in $\calB_0[[t;T]]$. The same arguments work for elements of $\calB_0[[t^{-1},T^{-1}]]$.
\end{proof}

We now introduce the following definitions.

\begin{definition}\label{definition-divisions.closure}
\hspace{2em}
\begin{enumerate}[a),leftmargin=1cm]
\item We denote by $\calB_+$ the algebra of elements of $\calB$ supported in non-negative degrees in $t$, that is $\calB_+ = \bigoplus_{i \geq 0} \calB_i t^i$. Clearly $\calB_+ \subseteq \calB_0[[t;T]]$. The division closure of $\calB_+$ in $\calB_0[[t;T]]$ will be denoted by $\calD_+$.
\item We set $\calB_- = \bigoplus_{i \geq 0} \calB_{-i} t^{-i}$. Again $\calB_- \subseteq \calB_0[[t^{-1},T^{-1}]]$, and we denote by $\calD_-$ the division closure of $\calB_-$ in $\calB_0[[t^{-1};T^{-1}]]$.
\end{enumerate}
\end{definition}


\noindent In order to study the division closures $\calD_+$ and $\calD_-$, we need the following known lemma.
    
\begin{lemma}[cf. \cite{Kei}]\label{lemma-B0.regular}
Let $\calS$ be a unital $*$-subalgebra of $C_K(X)$ generated by a family of characteristic functions of the form $\{\chi_C\}_C$, where $C$ are clopen subsets of $X$. Then $\calS$ is a $*$-regular ring, and every non-zero element of $\calS$ can be expressed in the form
$$\sum_{i=1}^n \lambda_i \chi_{K_i},$$
where $\lambda_i \in K \backslash \{0\}$ for all $1 \leq i \leq n$, and $\{K_i\}_{i=1}^n$ are mutually disjoint clopen subsets of $X$ such that $\chi_{K_i} \in \calS$ for all $1 \leq i \leq n$.

In particular, the $*$-subalgebra $\calB_0$ of $\calB$ is $*$-regular, and $b_0 \in \calB_0$ is invertible if and only if, when writing $b_0$ in the above form, the family $\{K_i\}_{i=1}^n$ constitutes a partition of $X$.
\end{lemma}

\begin{proof}
If $a= \sum_{i=1}^n \lambda _i \chi_{K_i}$ is as in the statement, then $\ol{a} = \sum_{i=1}^n \lambda_i^{-1} \chi_{K_i}$ is the relative inverse of $a$, hence $\calS$ is $*$-regular. 
     
We show that each element of $\calS$ can be written in the stated form. It is clear that each element of $\calS$ is a $K$-linear combination of functions of the form $\chi_{L_i}$, where $L_i$ is a clopen subset of $X$ and $\chi_{L_i} \in \calS$, since every product $\chi_{C_1} \chi_{C_2} \cdots \chi_{C_t}$ belongs to $\calS$ and equals $\chi_{L}$, where $L = C_1 \cap \cdots \cap C_t$ is clopen.
Therefore, every non-zero element $a$ of $\calS$ can be written as $a = \sum_{i=1}^n \lambda_i \chi_{L_i}$, with $\{L_i\}_{i=1}^n$ clopen subsets of $X$ such that $\chi_{L_i} \in \calS$. We now show that this sum can be chosen to be an orthogonal sum. This is done by induction on $n$.

The result is clear for $n=1$, so assume that $n \geq1$, that $a = \sum_{i=1}^{n+1} \lambda_i \chi_{L_i}$ with $\{L_i\}_{i=1}^{n+1}$ clopen subsets of $X$ such that $\chi_{L_i} \in \calS$, and that $\sum_{i=1}^n \lambda_i \chi_{L_i} = \sum_{j=1}^m \mu_j \chi_{K_j}$ where now $\{K_j\}_{j=1}^m$ are mutually disjoint clopen subets of $X$ such that $\chi_{K_j} \in \calS$. We compute
$$a = \sum_{j=1}^m \mu_j \chi_{K_j} + \lambda_{n+1} \chi_{L_{n+1}} = \sum_{j=1}^m (\mu_j + \lambda_{n+1}) \chi_{K_j \cap L_{n+1}} + \sum_{j=1}^m \mu_j \chi_{K_j \backslash L_{n+1}} + \lambda_{n+1} \chi_{L_{n+1} \backslash (K_1 \cup \cdots \cup K_m)}.$$
Since the clopen sets $\{K_j \cap L_{n+1}\}_{j=1}^m \cup \{ K_j \backslash L_{n+1} \}_{j=1}^m \cup \{ L_{n+1} \backslash (K_1 \cup \cdots \cup K_m)\}$ are clearly disjoint and 
all their characteristic functions belong to $\calS$, this completes the induction step.

Now, since $\calB_0$ is generated by a family of characteristic functions of the above form, it is $*$-regular.

To conclude, assume that $b_0 = \sum_{i=1}^n \lambda_i \chi_{K_i} \in \calB_0$ is invertible in $\calB_0$. Then necessarily its inverse must be its relative inverse $\ol{b_0} = \sum_{i=1}^n \lambda_i^{-1} \chi_{K_i}$, and so $1 = b_0 \ol{b_0} = \sum_{i=1}^n \chi_{K_i}$. This tells us that $\{K_i\}_{i=1}^n$ forms a partition of $X$. The converse is easily verified, with also $b_0^{-1} = \ol{b_0}$.
\end{proof}


\begin{proposition}\label{proposition-charac.division.closure}
With the preceding notation, we have:
\begin{enumerate}[(i),leftmargin=1cm]
\item $\calD_+$ coincides with the rational closure of $\calB_+$ in $\calB_0[[t;T]]$, and similarly $\calD_-$ coincides with the rational closure of $\calB_-$ in $\calB_0[[t^{-1};T^{-1}]]$.
\item $\pi_+(\calD_+)$ is the division closure of $\pi_+(\calB_+)$ in $\gotR$, and similarly $\pi_-(\calD_-)$ is the division closure of $\pi_-(\calB_-)$ in $\gotR$.
\item $\pi_+(\calD_+) \subseteq \calR_{\calB}$, and similarly $\pi_-(\calD_-) \subseteq \calR_{\calB}$. 
\item $\pi_+(\calD_+)^* = \pi_-(\calD_-)$. 
\end{enumerate}
\end{proposition}

\begin{proof}
$(i)$ This is a standard observation (see e.g. \cite[Observation 1.18]{AB}).


$(ii)$ Recall that $\pi_+$ is an injective homomorphism from $\calB_0[[t;T]]$ into $\gotR$. We first show that $\pi_+(\calB_0[[t;T]])$ is division closed in $\gotR$. For this, let $x = \sum_{i \geq 0} b_i t^i$ be an element in $\calB_0[[t;T]]$ such that $\pi_+(x)$ is invertible in $\gotR$. Observe that each component of $\pi_+(x)$ is an invertible matrix, with diagonal coming exclusively from elements of $\calB_0$. It follows that $\pi_+(b_0) = \pi(b_0)$ must be invertible in $\gotR$. But since $\calB_0$ is regular (Lemma \ref{lemma-B0.regular}), there exists $\ol{b}_0$ in $\calB_0$ such that $b_0 \ol{b}_0 b_0 = b_0$. Applying $\pi$ and taking into account that $\pi(b_0)$ is invertible in $\gotR$, we get that $\pi(b_0)^{-1} = \pi(\ol{b}_0)$, and so $b_0$ is in fact invertible in $\calB_0$. It follows from Lemma \ref{lemma-invertibility.PS} that $x$ is invertible in $\calB_0[[t;T]]$, as required.

Now we use the following general fact: if $R \subseteq S \subseteq T$ are unital embeddings of unital rings, and $S$ is division closed in $T$, then the division closure of $R$ in $T$ equals the division closure of $R$ in $S$, that is $\calD(R,T) = \calD(R,S)$. 
Using this and the fact just proved that $\pi_+(\calB_0[[t;T]])$ is division closed in $\gotR$, we deduce that
$$\calD(\pi_+(\calB_+),\gotR) = \calD(\pi_+(\calB_+),\pi_+(\calB_0[[t;T]])) = \pi_+( \calD(\calB_+,\calB_0[[t;T]]) ) = \pi_+(\calD_+),$$
as desired. Analogous arguments give that $\calD(\pi_-(\calB_-),\gotR) = \pi_-(\calD_-)$.

$(iii)$ By $(ii)$, we have that $\pi_+(\calD_+) = \calD(\pi_+(\calB_+),\gotR)$ which is contained in the division closure of $\pi(\calB)$ in $\gotR$, $\calD(\pi(\calB),\gotR)$. This last one is contained in $\calR_{\calB}$ by Lemma \ref{lemma-div.rat.closure}, hence $\pi_+(\calD_+) \subseteq \calR_{\calB}$. Similarly $\pi_-(\calD_-) \subseteq \calR_{\calB}$.

$(iv)$ First observe that $\pi_+(\calB_+)^* = \pi_-(\calB_-)$ and $\pi_+(\calB_0[[t;T]])^* = \pi_-(\calB_0[[t^{-1},T^{-1}]])$. The reason is that, for $x = \sum_{i \geq 0} b_i t^i \in \calB_0[[t;T]]$ (resp. $\in \calB_+$), we have
$$\pi_+(x)^* = \pi_-\Big( \sum_{i \geq 0} t^{-i} b_i^* \Big) = \pi_-\Big( \sum_{i \geq 0} T^{-i}(b_i^*) t^{-i} \Big).$$
The element $b_i^*$ is computed in the $*$-algebra $\calB_0$. Also, by the description of $\calB$ as a partial crossed product (Proposition 3.7 of \cite{AC}), it follows that $T^{-i}(b_i^*) \in \calB_{-i}$ and so $\sum_{i \geq 0} T^{-i}(b_i^*) t^{-i} \in \calB_0[[t^{-1};T^{-1}]]$ (resp. $\in \calB_-$). 
Analogous arguments show the other inclusion(s).

Now,
\begin{equation*}
\pi_+(\calD_+)^* = \calD( \pi_+(\calB_+)^* , \pi_+(\calB_0[[t;T]])^*) = \calD( \pi_-(\calB_-), \pi_-(\calB_0[[t^{-1};T^{-1}]]) ) = \pi_-(\calD_-),
\end{equation*}
as required.
\end{proof}

We have two subalgebras $\pi_+(\calD_+)$ and $\pi_-(\calD_-) = \pi_+(\calD_+)^*$ of $\calR_{\calB}$. We will write $\calD$ for the $*$-subalgebra of $\calR_{\calB}$ generated by $\pi_+(\calD_+)$, which coincides with the subalgebra generated by $\pi_+(\calD_+)$ and $\pi_-(\calD_-)$. Intuitively, we obtain $\calD$ by first adjoining all possible inverses of elements of $\calB_+$ and then taking adjoints in $\gotR$.

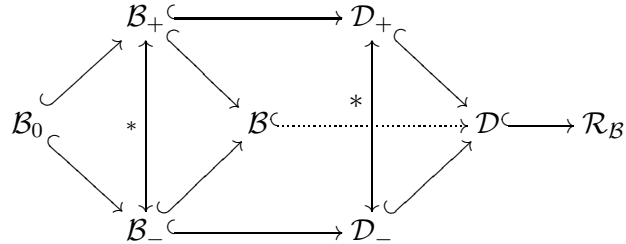
\begin{equation*}
\xymatrix{
 & \calB_+ \ar@{^{(}->}[rr] \ar@{^{(}->}[rd] \ar@{<->}[dd]_{*} & & \calD_+ \ar@{^{(}->}[rd] \ar@{<->}[dd]_{\rotatebox{90}{\text{ } \text{ } \text{ }\rotatebox{-90}{$*$}}} & & \\
 \calB_0 \ar@{^{(}->}[ru] \ar@{^{(}->}[rd] & & \calB \ar@{^{(}.>}[rr] & & \calD \ar@{^{(}->}[r] & \calR_{\calB} \\
 & \calB_- \ar@{^{(}->}[rr] \ar@{^{(}->}[ru] & & \calD_- \ar@{^{(}->}[ru] & &
}
\end{equation*}

Note that $\calD$ is indeed contained in the division closure $\calD(\calB,\gotR)$ of $\calB$ in $\gotR$.  

Our plan now is to adjoin to $\calD$ the relative inverses of (some) elements of $\calD$, in a controlled way. This is done as follows.

\begin{definition}[Special terms inside $\calB$]\label{definition-special.terms}
\text{ }
\begin{enumerate}[a),leftmargin=1cm]
\item A monomial $b_i t^i \in \calB_+$ (for $i > 0$) is said to be \textit{special} if the coefficient $b_i \in \calB_i$ is exactly of the form $\chi_S$, with
\begin{equation}\label{equation-special.form1}
S = T^{i-1}(Z'_{i-1}) \cap T^{i-2}(Z'_{i-2})\cap \cdots \cap Z'_0, \qquad \qquad (Z'_0,\ldots , Z'_{i-1}\in \calP ),
\end{equation}
that is $s=i$ and $r=0$ in \eqref{equation-form.of.bs}, and moreover we ask that $E \cap T^{-i}(S) \cap T^{-i-1}(E) \neq \emptyset$. The set of clopen subsets $S \subseteq X$ of the form \eqref{equation-special.form1} with $E \cap T^{-i}(S) \cap T^{-i-1}(E) \neq \emptyset$ will be denoted by $\calW_i$.
\item Similarly, a monomial $b_{-j} t^{-j} \in \calB_-$ (for $j > 0$) is said to be \textit{special} if the coefficient $b_{-j} \in \calB_{-j}$ is exactly of the form $\chi_{S'}$, with
\begin{equation}\label{equation-special.form2}
S' = T^{-1}(Z'_{-1}) \cap T^{-2}(Z'_{-2}) \cap \cdots \cap T^{-j}(Z'_{-j}), \qquad \qquad (Z'_{-1},\ldots , Z'_{-j}\in \calP),
\end{equation}
that is $r=j$ and $s=0$ in \eqref{equation-form.of.bs}, and moreover we ask that $E \cap S' \cap T^{-j-1}(E) \neq \emptyset$. The set of clopen subsets $S' \subseteq X$ of the form \eqref{equation-special.form2} with $E \cap S' \cap T^{-j-1}(E) \neq \emptyset$ will be denoted by $\calW_{-j}$.
\item If $E\cap T^{-1}(E)\ne \emptyset$, there is one term in $\calB_0$ which we call \textit{special}, namely the element $\chi_{S_0 \cup T^{-1}(S_1)}$, where
\begin{equation*}\label{equation-special.form3}
S_0 = E \cup \Big( \bigcup_{\substack{Z \in \calP \\ Z \cap T^{-1}(E) \neq \emptyset}} Z \Big) \quad \text{ and } \quad S_1 = E \cup \Big( \bigcup_{\substack{Z \in \calP \\ T^{-1}(Z) \cap E \neq \emptyset}} Z \Big).
\end{equation*}
In case $E\cap T^{-1}(E) \ne \emptyset$, we set $\calW_0=\{ S_0 \cup T^{-1}(S_1) \}$. If $E\cap T^{-1}(E)= \emptyset$, then there is no special term of degree $0$ and so $\calW_0 =\emptyset $. 
\end{enumerate}
\end{definition}

It is clear that, for $i \geq 1$, the set $\calW_i$ is in bijection with the set of all $W \in \V$ having length $i+1$, through the map
$$S \mapsto W(S) := E \cap T^{-i}(S) \cap T^{-i-1}(E).$$
The inverse map will be written as $W \mapsto S(W)$, so that $S(W(S)) = S$ and $W(S(W)) = W$. Analogously, for $j \geq 1$, the same set of all $W \in \V$ having length $j+1$ is in bijection with $\calW_{-j}$ via
$$S' \mapsto W(S') := E \cap S' \cap T^{-j-1}(E).$$
Again, the inverse map will be denoted by $W \mapsto S'(W)$. When $E\cap T^{-1}(E)\ne \emptyset$, the set $\calW_0$ contains only one element, namely the clopen $S_0 \cup T^{-1}(S_1)$, and therefore is in bijective correspondence with the set consisting of the only clopen $W \in \V$ with length $1$, namely $W = E \cap T^{-1}(E)$.  We will use the notation $W \mapsto S(W) := S_0 \cup T^{-1}(S_1)$ in this case. When $E \cap T^{-1}(E)= \emptyset$ there is no $W\in \mathbb V$ having length $1$ and correspondingly $\calW_0= \emptyset$. 

Note that, by construction, the element $b_0 = \chi_{S_0 \cup T^{-1}(S_1)}$ serves as a unit among the special terms, in the sense that
$$b_0 \cdot b_i t^i = b_i t^i = b_i t^i \cdot b_0 \quad \text{ and } \quad b_0 \cdot b_{-j} t^{-j} = b_{-j} t^{-j} = b_{-j} t^{-j} \cdot b_0,$$
for $b_i t^i, b_{-j}t^{-j}$ special terms of degrees $i,j \geq 1$, respectively.

The special terms are exactly detected by the representation $\pi : \calB \ra \gotR$, as follows.
 
\begin{lemma}\label{lemma-special.terms}
With the previous notation,
\begin{enumerate}[i),leftmargin=1cm]
\item For $i>0$, let $b_it^i = \chi_S t^i$ be a special term, with $S$ as in \eqref{equation-special.form1}. Then $h_W \cdot b_it^i = e_{i,0}(W)$, where $W = W(S)$. Moreover, if $W' \neq W$ is of length $k \geq 1$, then the component of $e_{k-1,0}(W')$ in $h_{W'} \cdot b_it^i$ is $0$.
\item For $j>0$, let $b_{-j}t^{-j} = \chi_{S'}t^{-j}$ be a special term, with $S'$ as in \eqref{equation-special.form2}. Then $h_W \cdot b_{-j}t^{-j} = e_{0,j}(W)$, where $W = W(S')$. Moreover, if $W' \neq W$ is of length $k \geq 1$, then the component of $e_{0,k-1}(W')$ in $h_{W'} \cdot b_{-j}t^{-j}$ is $0$.
\item Suppose that $E\cap T^{-1}(E)\ne \emptyset$.  Let $b_0 = \chi_{S_0 \cup T^{-1}(S_1)}$ be the special term of degree $0$. Then $h_W \cdot b_0 = e_{0,0}(W)$, where $W = E \cap T^{-1}(E)$. Moreover, if $W \neq E \cap T^{-1}(E)$, then the components of $e_{0,0}(W)$ and $e_{|W|-1,|W|-1}(W)$ in $h_W \cdot b_0$ are exactly $1$.
\end{enumerate}
\end{lemma}

\begin{proof}
We will only prove $i)$, being the other ones analogous. Take $W = W(S) = E \cap T^{-i}(S) \cap T^{-i-1}(E)$, and note that $T^l(W) \cap S = \emptyset$ for $0 \leq l \leq i-1$. For $l=i$, it gives $T^i(W) \cap S = T^i(W)$. Hence
$$h_W \cdot b_i t^i = \sum_{l=0}^i \chi_{T^l(W) \cap S} t^i = \chi_{T^i(W)} t^i = (\chi_{X \backslash E} t)^i \chi_W = e_{i,0}(W).$$
For the second part, it is enough to show that the product $e_{k-1,k-1}(W') \big( h_{W'} \cdot b_i t^i \big) e_{00}(W')$ is zero. This is a straightforward computation:
$$e_{k-1,k-1}(W') \Big( h_{W'} \cdot b_i t^i \Big) e_{00}(W') = \chi_{T^{k-1}(W')} \Big( \sum_{l=0}^{k-1} \chi_{T^l(W') \cap S} t^i \Big) \chi_{W'} = \chi_{T^{k-1}(W') \cap T^i(W') \cap S} t^i.$$
But $W = E \cap T^{-i}(S) \cap T^{-i-1}(E)$, so $T^{k-1}(W') \cap T^i(W') \cap S \subseteq T^{k-1}(W') \cap T^i(W' \cap W)$ which is empty for $W' \neq W$. The result follows.
\end{proof}

Consider the subset $\calS[[t;T]]$ of $\calB_0[[t;T]]$ consisting of those elements
$$\sum_{i \geq 0} b_i (\chi_{X \backslash E}t)^i = \sum_{i \geq 0} b_i t^i$$
such that each $b_i \in \calB_i$ belongs to $\text{span}\{ \chi_S \mid S \in \calW_i\}$. It is always a linear subspace of $\calB_0 [[t;T]]$, but it might not be a subalgebra. We will, however, see in Section \ref{subsection-lamplighter.alg.reg.closure} that in the special case of $\calA$ being the lamplighter group algebra, $\calS[[t;T]]$ is indeed an algebra, and even an integral domain.

As we shall see, $\calS[[t;T]]$ certainly is a subalgebra of $\calB_0[[t;T]]$ when it is endowed with the multiplicative structure given by the Hadamard product $\odot$, defined by the rule
$$\Big( \sum_{i\geq 0} b_i t^i \Big) \odot \Big( \sum_{j\geq 0} b'_j t^j \Big) := \sum_{i\geq 0} (b_i b'_i) t^i.$$

\begin{observation}\label{observation-hadamard.prod}
Each $b_i$ belongs to the linear span of all the $\chi_S$ with $S \in \calW_i$, hence they can be written as $b_i = \sum_{S \in \calW_i} \lambda_S \chi_S$ for $\lambda_S \in K$. Let $b_i, b'_i$ be two given elements of this form:
$$b_i = \sum_{S \in \calW_i}\lambda_S \chi_S, \quad b'_i = \sum_{S' \in \calW_i} \mu_{S'} \chi_{S'}.$$
Since the sets $S \in \calW_i$ are of the form $T^{i-1}(Z'_{i-1}) \cap T^{i-2}(Z'_{i-2}) \cap \cdots \cap Z'_0$ for $i > 0$ or $S_0 \cup T^{-1}(S_1)$ in case $i = 0$, we see that for $S,S' \in \calW_i$, $S \cap S' = \emptyset$ if they are different. Therefore the Hadamard product of $\calS[[t;T]]$ can be written as
$$\Big( \sum_{i\geq 0} b_i t^i \Big) \odot \Big( \sum_{j\geq 0} b'_j t^j \Big) = \sum_{i\geq 0} (b_i b'_i) t^i = \sum_{i\geq 0} \Big( \sum_{S \in \calW_i} \lambda_S \mu_S \chi_{S} \Big) t^i.$$
\end{observation}

We can also define an involution on $\calS[[t;T]]$ by
$$\overline{\sum_{i\geq 0} \Big( \sum_{S \in \calW_i} \lambda_S \chi_S \Big) t^i} := \sum_{i\geq 0} \Big( \sum_{S \in \calW_i} \ol{\lambda_S} \chi_S \Big) t^i.$$
These operations turn $\calS[[t;T]]$ into a commutative $*$-algebra $(\calS[[t;T]], \odot,-)$. Indeed it is a $*$-algebra isomorphic to $\prod_{W\in \V}K$. In the next proposition we show that it can be identified with the center of the algebra $\gotR = \prod_{W \in \V} M_{|W|}(K)$, and also with a certain corner of $\gotR$. We first fix some notation: we will denote the projections $\pi(\chi_C) \in \gotR$ by $p_C$ for any clopen $C \subseteq X$. So for example $p_E = \pi(\chi_E) = (e_{00}(W))_W \in \gotR$, and $p_{T^{-1}(E)} = \pi(\chi_{T^{-1}(E)}) = (e_{|W|-1,|W|-1}(W))_W \in \gotR$.

\begin{proposition}\label{proposition-S.corner.R}
We have an isomorphism of $*$-algebras $\calS[[t;T]] \overset{\Psi}{\cong} Z(\gotR)$, the center of $\gotR$. In particular, we have a $*$-isomorphism $\calS[[t;T]] \cong p_E \gotR p_E$ given by $d \mapsto \Psi(d) p_E$, $d \in \calS[[t;T]]$.
\end{proposition}
\begin{proof}
Write an element $\sum_{i \geq 0} b_i t^i \in \calS[[t;T]]$ as
$$\sum_{i \geq 0} b_i t^i = \sum_{i \geq 0} \Big( \sum_{S \in \calW_i} \lambda_S \chi_S \Big) t^i = \sum_{W \in \V} (\lambda_{S(W)} \chi_{S(W)}) t^{|W|-1}.$$
Note that $Z(\gotR) = Z(\prod_{W \in \V} M_{|W|}(K)) = \prod_{W \in \V} K$. We define a map $\Psi : \calS[[t;T]] \ra Z(\gotR)$ by
$$\Psi \Big(\sum_{W \in \V} (\lambda_{S(W)} \chi_{S(W)}) t^{|W|-1} \Big) = (\lambda_{S(W)} \cdot h_W)_W.$$
It is straightforward to check that it is indeed an isomorphism of $*$-algebras. Since $Z(\gotR) \cong p_E \gotR p_E$ through $z \mapsto p_E z$, the result follows.
\end{proof}

Our next step is to prove the following formulas, which will be useful later.

\begin{lemma}\label{lemma-formulas.computations}
For $A,B \in \calS[[t;T]]$, the following formulas hold inside $\gotR$:
\begin{align*}
p_E \cdot \pi_+(A)^* \cdot p_{T^{-1}(E)} \cdot \pi_+(B) \cdot p_E & = \Psi(\ol{A} \odot B) p_E,\\
p_{T^{-1}(E)} \cdot \pi_+(A) \cdot p_E \cdot \pi_+(B)^* \cdot p_{T^{-1}(E)} & = \Psi( A \odot \overline{B}) p_{T^{-1}(E)}.
\end{align*}
\end{lemma}
\begin{proof}
We will only prove the first formula, being the second one analogous. Note that it is enough to prove that the $W$-component of the left-hand side and the right-hand side of the formula agree. More precisely, we have to check the equality
$$\big( p_E \cdot \pi_+(A)^* \cdot p_{T^{-1}(E)} \cdot \pi_+(B) \cdot p_E \big)_W = \big( \Psi(\ol{A} \odot B) p_E \big)_W$$
for any fixed $W \in \V$. Write $A = \sum_{i \geq 0} \Big(\sum_{S \in \calW_i} \lambda_S \chi_S \Big) t^i$ and $B = \sum_{j \geq 0} \Big(\sum_{S' \in \calW_j} \mu_{S'} \chi_{S'} \Big) t^j$. We first compute, for a fixed $S \in \calW_i$ and using Lemma \ref{lemma-special.terms}, the terms
$$e_{|W|-1,|W|-1}(W) \cdot (h_W \cdot \chi_S t^i) \cdot e_{00}(W) = \delta_{W,W(S)} e_{i0}(W) = \delta_{W,W(S)} e_{|W|-1,0}(W).\footnote{Note that the appearance of the term $\delta_{W,W(S)}$ already encodes the fact that the term is $0$ if $|W| \neq i+1$.}$$
Therefore
\begin{align*}
& \big( p_E \cdot \pi_+(A)^* \cdot p_{T^{-1}(E)} \cdot \pi_+(B) \cdot p_E \big)_W \\
& \qquad = e_{00}(W) \cdot \Big( \sum_{i \geq 0} \sum_{S \in \calW_i} \lambda_S h_W \cdot \chi_S t^i \Big)^* \cdot e_{|W|-1,|W|-1}(W) \cdot \Big( \sum_{j \geq 0} \sum_{S' \in \calW_j} \mu_{S'} h_W \cdot \chi_{S'}t^j \Big) \cdot e_{00}(W) \\
& \qquad = e_{00}(W) \cdot \big(\lambda_{S(W)} e_{|W|-1,0}(W) \big)^* \cdot e_{|W|-1,|W|-1}(W) \cdot \big(\mu_{S(W)} e_{|W|-1,0}(W) \big) \cdot e_{00}(W) \\
& \qquad = \ol{\lambda_{S(W)}} \mu_{S(W)} e_{00}(W) = \big( \Psi(\ol{A} \odot B) p_E \big)_W,
\end{align*}
so the result follows.
\end{proof}

We now proceed to generalize the above formulas. For this purpose, we first define an idempotent map $P : \calB_0[[t;T]] \ra \calS[[t;T]]$ as follows.

\begin{lemma}\label{lemma-the.idempotent.P}
 With the above notation, there exists an idempotent linear map $P : \calB_0[[t;T]] \ra \calS[[t;T]]$ such that for each $x \in \calB_0[[t;T]]$, we have
 $$p_{T^{-1}(E)} \cdot \pi_+(x) \cdot p_E = p_{T^{-1}(E)} \cdot \pi_+(P(x)) \cdot p_E.$$
\end{lemma}

\begin{proof}
For $i \geq 1$, let $V_i$ be the linear subspace of $\calB_i$ given by $\text{span}\{ \chi_S \mid S \in \calW_i\}$, and let $V_i'$ be the linear subspace of $\calB_i$ spanned by all the projections $\chi_C$, where 
\begin{align*}
(*)_i \quad \,  & C \text{ is a non-empty clopen subset of $X$ of the form \eqref{equation-form.of.bs}, with $s\ge i$,}\\
    & \text{ and such that either $s>i$ or $r>0$.}
\end{align*}
\begin{enumerate}[label=\underline{Claim \arabic*}:,leftmargin=*,labelindent=0em]
\item $\calB_i = V_i + V'_i$.
\begin{spaceleft}
\underline{Proof:} Recall that $\calB_i$ is spanned by all the characteristic functions $\chi_C$, where $C$ is a clopen subset of $X$ of the form \eqref{equation-form.of.bs} with $s \geq i$. If $s > i$ or $r > 0$, then $C$ is of the form $(*)_i$, so that $\chi_C \in V'_i$. So we can assume that $s = i$ and $r = 0$. Furthermore, if $T^i(E) \cap C \cap T^{-1}(E)$ is non-empty, then $C \in \calW_i$ and so $\chi_C \in V_i$. So we can further assume that $T^i(E) \cap C \cap T^{-1}(E) = \emptyset$.

We can write
$$C = \Big(T^i(E) \cap C \Big) \sqcup \Big(\bigsqcup_{Z \in \calP} T^i(Z) \cap C \Big).$$
If $T^i(E) \cap C = \emptyset$, then $\chi_C$ is a sum of terms of the form $(*)_i$, so that $\chi_C \in V_i'$. If $T^i(E) \cap C \neq \emptyset$, we can further decompose $C$ as
\begin{align*}
C & = \Big( \bigsqcup_{Z \in \calP} T^i(E) \cap C \cap T^{-1}(Z) \Big) \sqcup \Big(\bigsqcup_{Z \in \calP} T^i(Z) \cap C \Big)
\end{align*}
by using the assumption $T^i(E) \cap C \cap T^{-1}(E) = \emptyset$. Note that for each $Z \in \calP$, either $C \cap T^{-1}(Z)$ is empty or it is of the form $(*)_i$; in the latter case we can write any non-empty $T^i(E) \cap C \cap T^{-1}(Z)$ as 
$$T^i(E) \cap C \cap T^{-1}(Z) = \Big( C \cap T^{-1}(Z) \Big) \Big\backslash \Big( \bigsqcup_{Z' \in \calP} T^i(Z') \cap C \cap T^{-1}(Z) \Big).$$
Therefore $\chi_C$ is a linear combination of terms of the form $(*)_i$, and thus $\chi_C \in V'_i$.\quad\qedsymbol
\end{spaceleft}

\item $V_i \cap V'_i = \{0\}$.
\begin{spaceleft}
\underline{Proof:} Assume that $b \in V_i \cap V_i'$ and write $b= \sum_{S \in \calW_i} \lambda_S \chi_S$, with $\lambda_S\in K$. Since $b \in V_i'$ we have
$$ 0 = \Big( \sum_{S' \in \calW_i} \chi_{T^i(E) \cap S' \cap T^{-1}(E)} \Big) b = \sum_{S \in \calW_i} \lambda_S \chi_{T^i(E) \cap S \cap T^{-1}(E)}.$$
Since $\{\chi_{T^i(E)\cap S \cap T^{-1}(E)} \}_{S \in \calW_i}$ is a family of mutually orthogonal non-zero projections, we get that $b = 0$.\quad\qedsymbol
\end{spaceleft}
\end{enumerate}
Therefore $\calB_i = V_i \oplus V'_i$ for $i \geq 1$. In the base case $i=0$, we need to distinguish between two different scenarios, depending on whether the intersection $E \cap T^{-1}(E)$ is empty or not. 
\begin{enumerate}[label=\underline{Case \arabic*}:,leftmargin=*,labelindent=0em]
\item $E \cap T^{-1}(E) = \emptyset$.
\begin{spaceleft}
In this case we have $\calW_0=\emptyset$. We take $V'_0= \calB_0$ and $V_0=\{0\}$.
\end{spaceleft}
\item $E \cap T^{-1}(E) \neq \emptyset$.
\begin{spaceleft}
Define here $V'_0$ to be the linear subspace of $\calB_0$ spanned by all the projections $\chi_C$, where
$$(*)_0 \quad C \text{ is a non-empty clopen subset of $X$ of the form \eqref{equation-form.of.bs}, }$$
and let $V_0 = K \cdot \chi_{S_0 \cup T^{-1}(S_1)}$ and $V_0''= K \cdot \chi_{E \cap T^{-1}(E)}$. Analogous computations as in the case for $i \geq 1$ show that there is a decomposition $\calB_0 = V_0'' \oplus V'_0$. In particular, $V'_0$ has codimension $1$, so in order to obtain the decomposition $\calB_0 = V_0 \oplus V'_0$ it is enough to show that $\chi_{S_0 \cup T^{-1}(S_1)} \notin V'_0$. But if $\chi_{S_0 \cup T^{-1}(S_1)} \in V'_0$, we could write it as
$$\chi_{S_0 \cup T^{-1}(S_1)} = \sum_{r,s \geq 0} \sum_{\text{some } Z_i} \lambda_{r,s,Z_i} \chi_{T^{-r}(Z_{-r}) \cap \cdots \cap T^{-1}(Z_{-1}) \cap Z_0 \cap \cdots \cap T^{s-1}(Z_{s-1})},$$
and by multiplying the above equality by $\chi_{E \cap T^{-1}(E)}$ we would get 
$$\chi_{E \cap T^{-1}(E)} = \chi_{E \cap T^{-1}(E)} \cdot \chi_{S_0 \cup T^{-1}(S_1)} = 0,$$
a contradiction. Hence $\chi_{S_0 \cup T^{-1}(S_1)} \notin V'_0$, and we have the desired decomposition.
\end{spaceleft}
\end{enumerate}
We can now define $P$ as the projection onto the first component in the decomposition
$$\calB_0[[t;T]] =  \Big( \prod_{i \geq 0} V_i t^i \Big) \oplus \Big( \prod_{i \geq 0} V'_i t^i \Big) = \calS[[t;T]] \oplus \Big( \prod_{i \geq 0} V'_i t^i \Big).$$
We check the formula in the statement. Take $x = \sum_{i \geq 0} b_i t^i \in \calB_0[[t;T]]$. We can write it as
$$x = P(x) + \sum_{i \geq 0} \sum_{\substack{C \text{ as in } (*)_i \\  }} \lambda_C \chi_C t^i$$
in the case $E \cap T^{-1}(E) \neq \emptyset$, and as
$$x = P(x) + b_0 + \sum_{i \geq 1} \sum_{\substack{C \text{ as in } (*)_i \\  }} \lambda_C \chi_C t^i$$
if $E \cap T^{-1}(E) = \emptyset$. Note that in the latter case $\chi_{T^{-1}(E)} \cdot b_0 \cdot \chi_E = 0$; hence to prove the formula of the statement it is enough to check that, for a fixed $W = E \cap T^{-1}(Z_1) \cap \cdots \cap T^{-k+1}(Z_{k-1}) \cap T^{-k}(E) \in \V$ and $C$ of the form $(*)_i$, we have $h_W \cdot \chi_{T^{-1}(E)} \cdot \chi_C t^i \cdot \chi_E = 0$. We compute
$$h_W \cdot \chi_{T^{-1}(E)} \cdot \chi_C t^i \cdot \chi_E = \chi_{T^{k-1}(W) \cap T^{-1}(E) \cap C \cap T^i(E)} t^i.$$
This is zero for $C$ of the form $(*)_i$, since either $C \subseteq T^i(X \backslash E)$ or $C \subseteq T^{-1}(X \backslash E)$. The result follows.
\end{proof}

We can now generalize the formulas in Lemma \ref{lemma-formulas.computations}.

\begin{lemma}\label{lemma-formulas.computations.II}
For $x, y \in \calB_0[[t;T]]$, the following formulas hold:
\begin{align*}
p_E \cdot \pi_+(x)^* \cdot p_{T^{-1}(E)} \cdot \pi_+(y) \cdot p_E & = \Psi( \overline{P(x)} \odot P(y)) p_E,\\
p_{T^{-1}(E)} \cdot \pi_+(x) \cdot p_E \cdot \pi_+(y)^* \cdot p_{T^{-1}(E)} & = \Psi( P(x) \odot \overline{P(y)}) p_{T^{-1}(E)}.
\end{align*}
\end{lemma}

\begin{proof}
By Lemma \ref{lemma-the.idempotent.P} we have $p_{T^{-1}(E)} \cdot \pi_+(x) \cdot p_E = p_{T^{-1}(E)} \cdot \pi_+(P(x)) \cdot p_E$ for all $x \in \calB_0[[t;T]]$. Taking the involution on both sides, we get that $p_E \cdot \pi_+(x)^* \cdot p_{T^{-1}(E)} = p_{E} \cdot \pi_+(P(x))^* \cdot p_{T^{-1}(E)}$ for all $x \in \calB_0[[t;T]]$. Now we obtain, from Lemma \ref{lemma-formulas.computations},
$$p_E \cdot \pi_+(x)^* \cdot p_{T^{-1}(E)} \cdot \pi_+(y) \cdot p_E = p_E \cdot \pi_+(P(x))^* \cdot p_{T^{-1}(E)} \cdot \pi_+(P(y)) \cdot p_E = \Psi(\overline{P(x)}\odot P(y)) p_E,$$
as desired. The proof of the other equality is similar.
\end{proof}

Recall that $\calS[[t;T]]$ is a unital $*$-regular commutative algebra under the Hadamard product $\odot$, the unit being the element $e = \sum_{i \geq 0} (\sum _{S \in \calW_i} \chi_S) t^i$. We obtain:
 
\begin{proposition}\label{proposition-formulas.proj}
Take $s = \sum_{Z \in \calP} \chi_Z t= \chi_{X \setminus E}t \in \calB_1t$. Then $u := (1-s)^{-1}= 1+s+s^2+ \cdots \in \calD_+$ satisfies that $P(u)= e$, where $e$ is the unit element of $(\calS[[t;T]], \odot)$. As a consequence, we have the formulas
\begin{align*}
p_E \cdot \pi_+(u)^* \cdot p_{T^{-1}(E)} \cdot \pi_+(x) \cdot p_E & = \Psi(P(x)) p_E,\\
p_{T^{-1}(E)} \cdot \pi_+(x) \cdot p_E \cdot \pi_+(u)^* \cdot p_{T^{-1}(E)} & = \Psi(P(x)) p_{T^{-1}(E)}
\end{align*}
for all $x \in \calB_0[[t;T]]$. In particular, inside $\calR_{\calB}$, the ideal generated by $p_E$ coincides with the ideal generated by $p_{T^{-1}(E)}$.
\end{proposition}
\begin{proof}
We first note that
$$s^i = (\chi_{X \backslash E}t)^i = \chi_{X \backslash (E \cup \cdots \cup T^{i-1}(E))} t^i = \sum_{Z_0, Z_1, ..., Z_{i-1} \in \calP} \chi_{Z_0 \cap T(Z_1) \cap \cdots \cap T^{i-1}(Z_{i-1})} t^i$$
for $i \geq 1$. It follows from this formula and the computations done in Lemma \ref{lemma-the.idempotent.P} that $P(s^i) = \sum_{S \in \calW_i} \chi_S t^i$. To compute $P(1)$, note first that $P(1)=0$ if $E\cap T^{-1}(E)=\emptyset$. So assume that $E\cap T^{-1}(E)\ne \emptyset$, and observe that
$$1 = \chi_{S_0 \cup T^{-1}(S_1)} + \chi_{(X \backslash S_0) \cap (X \backslash T^{-1}(S_1))} = \chi_{S_0 \cup T^{-1}(S_1)} + \sum_{\substack{Z \in \calP \\ Z \cap T^{-1}(E) = \emptyset}} \sum_{\substack{Z' \in \calP \\ T^{-1}(Z') \cap E = \emptyset}} \chi_{Z \cap T^{-1}(Z')}.$$
By definition, the second part of this expression belongs to the complement $V'_0$, so that $P(1) = \chi_{S_0 \cup T^{-1}(S_1)}$. Putting everything together, it is clear that $P(u) = e$. The desired formulas follow from Lemma \ref{lemma-formulas.computations.II}. In particular, we have
$$p_E \cdot \pi_+(u)^* \cdot p_{T^{-1}(E)} \cdot \pi_+(u) \cdot p_E = \Psi(P(u)) p_E = \Psi(e) p_E = p_E,$$
$$p_{T^{-1}(E)} \cdot \pi_+(u) \cdot p_E \cdot \pi_+(u)^* \cdot p_{T^{-1}(E)} = \Psi(P(u)) p_{T^{-1}(E)} = \Psi(e) p_{T^{-1}(E)} = p_{T^{-1}(E)}.$$
One deduces from this and Proposition \ref{proposition-charac.division.closure} $(iii)$ that the ideal generated by $p_E$ coincides with the ideal generated by $p_{T^{-1}(E)}$ inside $\calR_{\calB}$.
\end{proof}

We now define $*$-algebras $\calQ$ and $\calE$ which resemble the ones defined for a particular case in \cite[Lemma 6.10]{AG} (see Theorem \ref{theorem-computation.*.reg.closure.0} for the exact relation between these constructions). The algebra $\calQ$ is $*$-regular, and the $*$-algebra $\calE$ is contained in the $*$-regular closure $\calR_{\calB}$ of $\calB$ in $\gotR$.

\begin{definition}\label{definition-algebraQ}
With the above notation, we define the $*$-algebra $\calQ$ as the $*$-regular closure of $P(\calD_+)$ in the $*$-regular algebra $(\calS[[t;T]], \odot, -)$. In other words, $\calQ$ is the smallest $*$-regular subalgebra of $\calS[[t;T]]$ containing $P(\calD_+)$. The $*$-algebra $\calE$ is defined as the subalgebra of $\gotR$ generated by $\calD$ and $\Psi(\calQ )p_E$. 
\end{definition}

Let $x \in \calD_+ \subseteq \calB_0[[t;T]]$, $x = \sum_{i \geq 0}b_i(x) t^i$. By Lemma \ref{lemma-the.idempotent.P}, we can decompose each $b_i(x) = b_i^{(1)}(x) + b_i^{(2)}(x) \in V_i \oplus V_i'$, being $b_i^{(1)}(x) = \sum_{S \in \calW_i} \lambda_S(x) \chi_S$. Then the relative inverse of $q := P(x) = \sum_{i \geq 0} \Big( \sum_{S \in \calW_i}\lambda_S(x)\chi_S \Big) t^i$ inside $\calQ$ is given by (see Observation \ref{observation-hadamard.prod})
$$\ol{q} = \sum_{i \geq 0} \Big( \sum_{\substack{S \in \calW_i \\ \lambda_S(x) \neq 0}} \lambda_S(x)^{-1} \chi_S \Big) t^i,$$
so that $q\ol{q} = \sum_{i \geq 0} \Big( \sum_{\substack{S \in \calW_i \\ \lambda_S(x) \neq 0}}\chi_S \Big) t^i$.

We are now ready to show the desired properties of the $*$-algebras $\calQ$ and  $\calE$.

\begin{proposition}\label{proposition-algebraQ.inside.calR}
We have an embedding $\Psi(\calQ) p_E \subseteq p_E \calR_{\calB} p_E$.
Moreover we have 
$$\calB\subseteq \calD \subseteq \calE \subseteq \calR_{\calB},$$ where $\calR_{\calB}$ is the $*$-regular closure of $\calB$ in $\gotR$.
\end{proposition}

\begin{proof}
Let $x \in \calD_+$, and set $u = (1-s)^{-1} \in \calD_+$. Due to Proposition \ref{proposition-formulas.proj}, we have
$$\Psi(P(x)) p_E = p_E \cdot \pi_+(u)^* \cdot p_{T^{-1}(E)} \cdot \pi_+(x) \cdot p_E.$$
By Proposition \ref{proposition-charac.division.closure}, we have that $\pi_+(\calD_+) \subseteq \calR_{\calB}$, so all the factors of the right-hand side of the above equality belong to $\calD\subseteq \calR_{\calB}$.  It follows that $\Psi(P(x)) p_E \in p_E \calR_{\calB} p_E$, and so $\Psi(P(\calD_+)) p_E \subseteq p_E \calR_{\calB} p_E$.

By Proposition \ref{proposition-S.corner.R}, the map $d \mapsto \Psi(d) p_E$ defines a $*$-isomorphism from $\calS[[t;T]]$ onto $p_E \gotR p_E$. It follows that $\Psi(\calQ) p_E$ is the $*$-regular closure of $\Psi(P(\calD_+)) p_E$ in $p_E \gotR p_E$.
Since  $\Psi(P(\calD_+)) p_E \subseteq p_E\calR_{\calB} p_E$ and $p_E\calR_{\calB}p_E$ is $*$-regular, we conclude that  $\Psi(\calQ) p_E\subseteq p_E\calR_{\calB}p_E\subseteq \calR_{\calB}$. From this it is obvious that $\calE \subseteq \calR_{\calB}$. This shows the result.
\end{proof}

We will show in Theorem \ref{theorem-computation.*.reg.closure.0} that we have the equality $\calE= \calR_{\calB}$ for the algebra $\calB$ studied in \cite[Section 6]{AG}.

\section{Group algebras arising as \texorpdfstring{$\Z$}{}-crossed product algebras}\label{section-group.alg}

We start this section by first showing how the Fourier transform (sometimes called Pontryagin duality) describes the group algebras of several semidirect product groups as $\Z$-crossed product algebras. In this setting, the Atiyah problem for the group algebra is translated to a problem on computing ranks inside the corresponding $\Z$-crossed product algebra.

For a topological, second countable, locally compact abelian group $H$ one can define its Pontryagin dual $\wh{H}$ as the set of continuous homomorphisms $\phi : H \ra \T$, also called characters. With the compact-convergence topology, $\wh{H}$ becomes a topological, metrizable, locally compact abelian group. If $H$ is a countable discrete group then $\wh{H}$ is compact, and if moreover $H$ is a torsion group then $\wh{H}$ is totally disconnected. We refer the reader to \cite[Chapter 4]{Fol16} for more information about Pontryagin duality.\\


Suppose now that $H$ is a countable discrete, torsion abelian group. Associated with $H$, we consider the subset $\calO \subseteq \N$ given by
$$\calO = \{ n \in \N \mid \text{there exists an element } g \in H \text{ of order } n\}.$$
This set inherits the structure of a lattice from that of $\N$.

\begin{lemma}\label{lemma-O.lattice}
The set $\calO$ is a sublattice of $(\N,\emph{div},\emph{gcd},\emph{lcm})$. Even more, if $n \in \calO$, then any divisor $d$ of $n$ belongs to $\calO$ too, so $\calO$ is a \emph{hereditary sublattice} of $\N$.
\end{lemma}
\begin{proof}
Given $a,b \in \calO$, there exist $h, g \in H$ such that $\text{o}(h) = a$ and $\text{o}(g) = b$. Take first $d$ to be any divisor of $a$, so we write $a = da'$ for some $a' \in \N$. Then the element $h^{a'} \in H$ has order exactly $d$, and so $d \in \calO$. In particular, we have shown that $\text{gcd}\{a,b\}$ belongs to $\calO$.

We must show now that $m := \text{lcm}\{a,b\}$ belong to $\calO$. If both $a$ and $b$ are coprime numbers, it is straightforward to show that the element $hg \in H$ has order exactly $ab = m$, so we assume that $a$ and $b$ are not coprime. If we let $\P$ be the set of prime numbers that appear in the factorizations of $a$ and $b$, we can write
$$a = \prod_{p \in \P} p^{\alpha_p}, \quad b = \prod_{p \in \P} p^{\beta_p},$$
being $\alpha_p,\beta_p$ non-negative integers. Define now
$$a' := \prod_{\substack{p \in \P \\ \alpha_p \geq \beta_p}} p^{\alpha_p}, \quad b' := \prod_{\substack{p \in \P \\ \alpha_p < \beta_p}} p^{\beta_p}$$
so that, by construction, $\text{gcd}\{a',b'\} = 1$ and $\text{lcm}\{a',b'\} = a'b' = m = \text{lcm}\{a,b\}$. Clearly, the order of the element $\wt{h} := h^{\frac{a}{a'}}$ is $a'$, and the order of the element $\wt{g} := g^{\frac{b}{b'}}$ is $b'$. We now use the previous case (since $a', b'$ are coprime) to conclude that the element $\wt{h}\wt{g} \in H$ has order 
exactly $a'b' = m$. This concludes the proof of the lemma.
\end{proof}

For this section, we take $K$ to be any field satisfying the following hypothesis:
\begin{enumerate}[$(a)$,leftmargin=1cm]
\item the characteristic of $K$ must be coprime with all $n \in \calO$;
\item $K$ must contain all the $n^{\text{th}}$ roots of unity, for any $n \in \calO$.
\end{enumerate}
We write
$$\calU(K) = \bigcup_{n \in \calO}\{ \text{all }n^{\text{th}}\text{ roots of unity in }K\} \subseteq K.$$
It is clear that $\calU(K)$ is a subgroup of the multiplicative group $K^{\times}$. We define the $K$-\textit{Pontryagin dual} of $H$ as the set of all homomorphisms
$$\wh{H} := \{\phi : H \ra K \mid \phi \text{ is a morphism of groups}\} = \{\phi : H \ra \calU(K) \mid \phi \text{ is a morphism of groups}\}.$$
This is a group under pointwise product. Note that, since $H$ is endowed with the discrete topology, any element $\phi \in \wh{H}$ will be automatically continuous, whatever would be the topology on $\calU(K)$. Therefore in the case $K = \C$ this definition coincides with the usual Pontryagin dual.

We want to mimic the properties of $\wh{H}$ for the case $K = \C$ to this general setting. Indeed we will show later that $\wh{H}$ does not depend on the particular field  $K$ up to isomorphism. 
To this aim, we need the notion of \textit{compatible family of primitive roots of unity}.

\begin{definition}\label{definition-comp.primitive.roots.unity}
Let $\calX := \{\xi_n\}_{n \in \calO}$ be a family of roots of unity in $K$ consisting of primitive ones (so for each $n \in \calO$, we choose a primitive $n^{\text{th}}$ root of unity $\xi_n$). The family $\calX$ is said to be \textit{compatible} if the equation
$$\xi_{n\cdot m}^n = \xi_m$$
holds for all $n,m \in \calO$ such that $n\cdot m \in \calO$.
\end{definition}

So for example for $K = \C$, the family $\xi_n := e^{\frac{2 \pi i}{n}}$ is compatible.

\begin{lemma}\label{lemma-existence.comp.prim.roots.unity}
A compatible family of primitive roots of unity in $K$ always exists.
\end{lemma}
\begin{proof} If $\calO$ is finite, then it follows from Lemma \ref{lemma-O.lattice} that $\calO$ is the set of divisors of some $n\in \N$. In this case, take a primitive $n^{\text{th}}$ root of unity $\xi_n$ and set $\xi_d := \xi_n^{n/d}$ for any divisor $d$ of $n$.

Now assume that $\calO$ is infinite. Enumerate $\calO = \{n_i\}_{i \in \N}$ such that $n_i < n_{i+1}$ for all $i \in \N$. Assume first that $K$ has characteristic $0$. We first construct a family $\{\eta_i\}_{i \in \N}$ such 
that each $\eta_i$ is a primitive $n_1 \cdots n_i^{\text{th}}$ root of unity, and that $\eta_{i+1}^{n_{i+1}} = \eta_i$. The elements $\eta_i$ may belong to an algebraic closure of $K$, but the elements
$\xi_n$ below belong necessarily to $K$ by our hypothesis.  

Let $\eta_1$ be a primitive $n_1^{\text{th}}$ root of unity, and assume we have constructed $\eta_i$ a primitive $n_1 \cdots n_i^{\text{th}}$ root of unity, such that
$$\eta_i^{n_i} = \eta_{i-1}.$$
To construct $\eta_{i+1}$, take first any primitive $n_1 \cdots n_i \cdot n_{i+1}^{\text{th}}$ root of unity $\omega_{i+1}$; then $\omega_{i+1}^{n_{i+1}}$ is a primitive $n_1 \cdots n_i^{\text{th}}$ root of unity, so of the form $\eta_i^j$ for some $j$ coprime with $n_1 \cdots n_i$. Let $l_0$ be the inverse of $j$ modulo $n_1 \cdots n_i$. If we can find an integer $k$ such that $l := l_0 + n_1 \cdots n_i \cdot k$ is coprime with $n_1 \cdots n_i \cdot n_{i+1}$ we will be done, since $\eta_{i+1} := \omega_{i+1}^l$ will be a primitive $n_1 \cdots n_{i+1}^{\text{th}}$ root of unity, and
$$\eta_{i+1}^{n_{i+1}} = \omega_{i+1}^{l \cdot n_{i+1}} = \eta_i^{j \cdot l} = \eta_i.$$
Let us now proceed to find an integer $k$ such that $l = l_0 + n_1 \cdots n_i \cdot k$ is coprime with $n_1 \cdots n_i \cdot n_{i+1}$. Assume first that $n_{i+1}$ is \textit{not} a prime number; by Lemma \ref{lemma-O.lattice}, its divisors must be in $\calO$, so they are already included in $n_1 \cdots n_i$. Since $l_0$ is coprime with $n_1 \cdots n_i$, it also is coprime with $n_{i+1}$. Thus we may take $k=0$, and we will be done. If $n_{i+1}$ \textit{is} a prime number $p$, then two cases can happen: (1) if $p$ does not divide $l_0$, then $l_0$ is coprime with both $n_1 \cdots n_i$ and $p = n_{i+1}$, and we may take $k=0$; (2) if $p$ does divide $l_0$, then it certainly cannot divide $l_0 + n_1 \cdots n_i$ (if that was the case, $p$ would divide $n_1 \cdots n_i$, which is impossible), we can thus take $k=1$.

Note that, because of our hypotheses on the characteristic of $K$, the above arguments work \textit{mutatis mutandis} for a field $K$ of arbitrary characteristic (assuming, of course, the required hypotheses).\\

Now, given such a family $\{\eta_i\}_{i \in \N}$ of primitive roots of unity, we define the desired compatible one. For each $n = n_i \in \calO$, let $n_j \in \calO$ be any element 
such that $n$ divides the product $n_1 \cdots n_j$, and consider
$$\xi_n := \eta_j^{\frac{n_1 \cdots n_j}{n}}.$$
We claim that this definition does not depend on the choice of $n_j$, for if $n_l \in \calO$ is another element such that $n$ divides $n_1 \cdots n_l$, then by assuming $n_l > n_j$ we compute
$$\eta_l^{n_{j+1} \cdots n_l} = \eta_{l-1}^{n_{j+1} \cdots n_{l-1}} = \cdots = \eta_j.$$
Here we have used the property $\eta_{i+1}^{n_{i+1}} = \eta_i$ for all $i \in \N$. The claim follows straightforwardly:
$$\eta_l^{\frac{n_1 \cdots n_l}{n}} = \eta_l^{\frac{n_1 \cdots n_j}{n} n_{j+1} \cdots n_l} = \eta_j^{\frac{n_1 \cdots n_j}{n}}.$$
It is easily checked that $\xi_n$ defines a primitive $n^{\text{th}}$ root of unity. We now verify the compatibility property: given $n,m \in \calO$ such that $n \cdot m \in \calO$, and $n_j \in \calO$ such that $n \cdot m$ divides $n_1 \cdots n_j$, we compute
$$\xi_{n \cdot m}^n = \eta_j^{\frac{n_1 \cdots n_j}{n \cdot m}n} = \eta_j^{\frac{n_1\cdots n_j}{m}} = \xi_m.$$
This concludes the proof of the lemma.
\end{proof}

Using Lemma \ref{lemma-existence.comp.prim.roots.unity}, we can indeed show that, with our hypothesis on $H$ and $K$, the group $\calU(K)$ does not depend on the field $K$,  as follows.

\begin{lemma}\label{lemma-UK.does.not.depend}
With the above hypothesis and notation, we have that $\calU(K)$ does not depend on $K$. In fact, the map $\Upsilon \colon \calU(K) \to \calU(\C)$ which sends $\xi_n^a$ to $e^{\frac{2 \pi i}{n}a}$ for $n \in \calO$ and $0 \leq a < n$ is a group isomorphism. 
\end{lemma}

\begin{proof}
We first observe that this map is well-defined. Suppose that $\xi_n^a = \xi_m^b$, for $n,m\in \calO$ and  $0\le a< n$, $0\le b < m$. 
Let $l= \text{lcm}\{n,m\}$, and write $l=nl_1=ml_2$. We have
$$\xi_{l}^{al_1}= \xi _{nl_1}^{al_1}= \xi_n^a = \xi _m^b = \xi_{ml_2}^{bl_2}= \xi_l^{bl_2},$$
so that $al_1-bl_2$ must be a multiple of $l$. Using this we obtain that  
$$\Upsilon (\xi_n^a)= e^{\frac{2 \pi i}{l}al_1} = e^{\frac{2 \pi i}{l}bl_2} = \Upsilon (\xi_m^b).$$
A similar computation shows that $\Upsilon$ is a group homomorphism. Clearly $\Upsilon $ is a bijection and so it is a group isomorphism. 
\end{proof}

Now we can translate the usual metric topology on $\calU(\C) \subseteq \C$ into a metric topology on $\calU(K)$. It turns out that this topology can also be defined directly in terms of the algebraic structure of $\calU(K)$, as follows.

Take a compatible family of primitive roots of unity $\calX = \{\xi_n\}_{n \in \calO}$, so that
$$\calU(K) = \bigcup_{n \in \calO}\{1,\xi_n,...,\xi_n^{n-1}\}.$$
Define a distance $d_{\xi}$ over $\calU(K)$ in the following way: for two elements $\xi_n^a, \xi_m^b \in \calU(K)$ with $0 \leq a < n$ and $0 \leq b < m$ (here both $n,m \in \calO$), define
$$d_{\xi}(\xi_n^a,\xi_m^b) := \begin{cases}
\Big| \frac{a}{n} - \frac{b}{m} \Big| & \text{ if } \Big| \frac{a}{n} - \frac{b}{m} \Big| \leq \frac{1}{2}; \\
1 - \Big| \frac{a}{n} - \frac{b}{m} \Big| & \text{ if } \Big| \frac{a}{n} - \frac{b}{m} \Big| > \frac{1}{2}.
\end{cases}$$
Note that $d_{\xi}$ is bounded by $\frac{1}{2}$. It is direct to prove that $d_{\xi}$ is indeed a well-defined distance. We note also that $d_{\xi}$ is translation invariant, in the sense that $d_{\xi}(g_1g,g_2g)= d_{\xi}(g_1,g_2)$ for $g_1,g_2,g\in \calU(K)$. This defines a metric topology on $\calU(K)$. We prove below that in the case $K = \C$, the subspace topology of $\calU(\C) \subseteq \C$ coincides with the metric topology induced by $d_{\xi}$ when taking the compatible family $\xi_n = e^{\frac{2 \pi i}{n}}$. It then follows that the isomorphism $\Upsilon$ from Lemma \ref{lemma-UK.does.not.depend} is indeed an isomorphism of topological groups.

\begin{lemma}\label{lemma-formula.norm.dist}
We have the formula
$$\left|  \xi_n^a - \xi_m^b \right|^2 = 4 \sin^2(\pi d_{\xi}(\xi_n^a,\xi_m^b)).$$
Here $\left| \cdot \right| $ denotes the complex norm.
\end{lemma}

\begin{proof}
It is a simple computation:
\begin{align*}
\left| \xi_n^a - \xi_m^b \right|^2 & = \Big[ \cos\Big(2 \pi\frac{a}{n}\Big) - \cos\Big(2 \pi\frac{b}{m}\Big) \Big]^2 + \Big[ \sin\Big(2 \pi\frac{a}{n}\Big) - \sin\Big(2 \pi\frac{b}{m}\Big) \Big]^2 \\
& = 2 - 2 \Big[ \cos\Big(2 \pi \frac{a}{n}\Big) \cos\Big(2 \pi \frac{b}{m}\Big) + \sin\Big(2 \pi \frac{a}{n}\Big) \sin\Big(2 \pi \frac{b}{m}\Big) \Big] \\
& = 2 \Big[ 1 - \cos\Big(2 \pi\frac{a}{n} - 2 \pi\frac{b}{m}\Big) \Big] = 4 \sin^2\Big(\pi \Big| \frac{a}{n} - \frac{b}{m}\Big| \Big) \\
& = 4 \sin^2(\pi d_{\xi}(\xi_n^a,\xi_m^b)).\hfill\qedhere
\end{align*}
\end{proof}

\begin{proposition}\label{proposition-equivalent.top.C}
For $K = \C$ and $\xi_n = e^{\frac{2 \pi i}{n}}$, the subspace topology for $\calU(\C) \subseteq \C$ coincides with the topology induced by the distance $d_{\xi}$.
\end{proposition}
\begin{proof}
For a real number $|x| \leq \frac{\pi}{2}$, there exist positive constants $c,C > 0$ such that $c|x| \leq |\sin(x)| \leq C|x|$. Applying Lemma \ref{lemma-formula.norm.dist} one gets positive constants $A,B > 0$ such that
$$A \cdot d_{\xi}(\xi_n^a,\xi_m^b) \leq \left| \xi_n^a - \xi_m^b \right| \leq B \cdot d_{\xi}(\xi_n^a,\xi_m^b).$$
The result follows.
\end{proof}

We can now endow $\wh{H}$ with the compact-convergence topology. Recall that, in our setting, this topology has as a basis the sets
$$B_{\{h_1,...,h_n\}}(\phi,\epsilon) := \{ \psi \in \wh{H} \mid d_{\xi}(\psi(h_i),\phi(h_i)) < \epsilon \text{ for all indices }i\},$$
where $\epsilon > 0$, $\phi \in \wh{H}$ and $\{h_1,...,h_n\}$ is a finite subset (hence compact) of $H$. It is easy to show that this indeed defines a basis for a topology $\tau_{cc}$ in $\wh{H}$. 

\begin{proposition} \label{proposition-properties.of.wH}
The topological group $(\wh{H},\tau_{cc})$ is a totally disconnected, compact and metrizable group.  
\end{proposition}

\begin{proof}
We have an isomorphism of topological groups $\wh{H}\cong \wh{H}_{\C}$ induced by $\Upsilon$, where $\wh{H}_{\C}$ is the usual Pontryagin dual of $H$, and it is well-known that $\wh{H}_{\C}$ has the stated properties. 
\end{proof}

Suppose now that $\Z$ acts on $H$ by automorphisms via $\rho : \Z \act H$. We write $G$ for the semi-direct product group $H \rtimes_{\rho} \Z$, so $G$ is generated by $t$ and by any set $S$ consisting of generators of $H$ \footnote{This crossed product construction can be generalized by replacing $\Z$ with any other countable discrete group $\Lambda$, as in \cite[Section 2]{Aus}. However, we will stick into the case $\Lambda = \Z$ for our purposes.}. We denote by $\wt{\rho} : \Z \act K H$ the action on the group algebra $K H$ extending $\rho$ by linearity, so that $K G \cong K H \rtimes_{\wt{\rho}} \Z$.

The action $\rho : \Z \act H$ induces another action $\wh{\rho} : \Z \act \wh{H}$ by homeomorphisms, defined in the usual way:
$$\wh{\rho}_n(\phi) := \phi \circ \rho_{-n} \quad \text{ for } n \in \Z \text{ and } \phi \in \wh{H}.$$
If we write $T := \wh{\rho}_1$ then the action $\wh{\rho}$ is generated by $T$, in the sense that
$$\wh{\rho}_n(\phi) = T^n(\phi) \quad \text{ for } \phi \in \wh{H} \text{ and } n \in \Z.$$
Observe that $T : \wh{H} \ra \wh{H}$ defines a homeomorphism of the totally disconnected, compact metrizable group $\wh{H}$. In the next proposition we establish the relationship between the group algebra $K G$ and the $\Z$-crossed product $C_K(\wh{H}) \rtimes_T \Z$ by means of the well-known Fourier transform.

\begin{proposition}\label{proposition-pontryagin.duality}
Assume the previous hypothesis and caveats; that is, with
$$\calO = \{ n \in \N \mid \text{there exists an element } g \in H \text{ of order } n\},$$
assume that the characteristic of $K$ is coprime with all $n \in \calO$, and that, for any $n \in \calO$, $K$ contains all the $n^{\text{th}}$ roots of 1. Let $\calX = \{\xi_n\}_{n \in \calO}$ be a compatible family of primitive roots of unity in $K$ (see Definition \ref{definition-comp.primitive.roots.unity} and Lemma \ref{lemma-existence.comp.prim.roots.unity}).

Then we can identify the group algebra $K G \cong K H \rtimes_{\wt{\rho}} \Z$ with $C_K(\wh{H}) \rtimes_T \Z$ via the Fourier transform $\pazF : K H \ra C_K(\wh{H})$, by sending an element $h \in H$ of order $n$ to the element
$$\sum_{j=0}^{n-1}\xi_n^{-j} \chi_{U_{h,j}},$$
where $U_{h,j} = \{ \phi \in \wh{H} \mid \phi(h) = \xi_n^j \}$ and $\chi_{U_{h,j}}$ denotes the characteristic function of the clopen $U_{h,j}$, and then extending it to a map $\pazF : K G \ra C_K(\wh{H}) \rtimes_T \Z$ by sending the generator $t$ of $\Z$ on $K G$ to the generator $\delta_t$ of $\Z$ on the $\Z$-crossed product.

If moreover $K$ is endowed with an involution $-$ satisfying the compatibility condition $\ol{\xi_n} = \xi_n^{-1}$ for all $n \in \calO$, then the Fourier transform preserves the involutions on both $K G$ and the $\Z$-crossed product.
\end{proposition}

See also \cite{Aus} and \cite{Gra14}, where the authors state analogous results. For a proof of Proposition \ref{proposition-pontryagin.duality}, see \cite{Claramunt}. Recall that the involution on $K G$ is defined by the rule $(\lambda g)^* = \overline{\lambda} g^{-1}$ for $\lambda \in K, g \in G$ and extended by linearity, and the involution on the $\Z$-crossed product is given by
$$(f \cdot \delta_t)^* = (f^* \circ T^{-1}) \cdot \delta_{t^{-1}},$$
with $f^*(\phi) = \overline{f(\phi)}$ for $f \in C_K(\wh{H})$ and $\phi \in \wh{H}$, and again extended by linearity.

\begin{remark}\label{remark-involutions}
The final condition in Proposition \ref{proposition-pontryagin.duality} that there is an involution on $K$ such that $\ol{\xi_n} = \xi_n^{-1}$ for all $n \in \calO$ is a non-trivial one when the field $K$ has characteristic $p>0$. When $\calO = \{1,2\}$, any involution --in particular the identity involution-- works for any field $K$ of characteristic $p\ne 2$ (in order to have that $p$ is coprime to $2$). If $\calO$ is the set of divisors of $p^n+1$ for some prime $p$ and some $n\ge 1$, then one can take the field $\F_{p^{2n}}$ of $p^{2n}$ elements, with the involution $\varphi^n$, where $\varphi$ is the Frobenius automorphism of $\F_{p^{2n}}$. If $\calO $ is infinite, then there is no field of characteristic $p>0$ with an involution with the property that $\ol{\xi_n}=\xi_n^{-1}$ for all $n\in \calO$.
\end{remark}

Let now $K \subseteq \C$ be a subfield of $\C$ closed under complex conjugation, which will be the involution on $K$. Recall that $\rk_{K G}$ denotes the canonical rank function on $K G$ given by the restriction of the rank function naturally arising from $\calU(G)$ (see Section \ref{subsection-Betti.numbers.Atiyah.problem}). Our question now is whether we can find a measure $\wh{\mu}$ on the space $\wh{H}$ such that, when applying the construction explained in Theorem \ref{theorem-rank.function}, we end up with a rank function $\rk_{\calA}$ on $\calA = C_K(\wh{H}) \rtimes_T \Z$ that coincides with $\rk_{K G}$ under the Fourier transform $\pazF$. The answer to this question is affirmative in the case $\rk_{K G}$ is extremal, and in fact $\wh{\mu}$ coincides with the normalized Haar measure on $\wh{H}$, as we show in the next proposition.

\begin{proposition}\label{proposition-measure.H}
Let $K \subseteq \C$ be a subfield closed under complex conjugation and containing all the $n^{\text{th}}$ roots of $1$, for $n \in \calO$. Then from $\rk_{K G}$ we can construct a full $T$-invariant probability measure $\wh{\mu}$ on $\wh{H}$, which coincides with the normalized Haar measure on $\wh{H}$.

If moreover $\rk_{K G}$ is extremal in $\Ps(K G)$ then $\wh{\mu}$ is ergodic, and when applying the construction from Theorem \ref{theorem-rank.function} to $\wh{\mu}$ we end up with a Sylvester matrix rank function $\rk_{\calA}$ on $\calA = C_K(\wh{H}) \rtimes_T \Z$ such that $\rk_{K G} = \rk_{\calA} \circ \pazF$.
\end{proposition}
\begin{proof}
We first define a finitely additive probability measure $\ol{\mu}_{K G}$ on the algebra $\K$ of clopen subsets of $\wh{H}$ by the rule $\ol{\mu}_{K G}(U) = \rk_{K G}(\pazF^{-1}(\chi_U))$ for every clopen subset $U$ of $\wh{H}$ which, by the same argument as in the proof of \cite[Proposition 4.8]{AC}, can be uniquely extended to a Borel probability measure $\mu_{K G}$ on $\wh{H}$. Invariance of $\mu_{K G}$ follows from the fact that $t$ is an invertible element, and since $\rk_{K G}$ is a faithful rank function it follows that $\mu_{K G}$ is full. 
If moreover $\rk_{K G}$ is extremal, then again an argument similar to the one given in the proof of \cite[Proposition 4.10]{AC} proves that $\mu_{K G}$ is ergodic. Now Theorem \ref{theorem-rank.function} implies that $\rk_{K G} = \rk_{\calA} \circ \pazF$, as required.

Finally, to prove that $\mu_{K G}$ coincides with the normalized Haar measure $\wh{\mu}$ on $\wh{H}$, just note that $\pazF^{-1}(\chi_U)$ is a projection in $K G$ for any clopen $U \subseteq \wh{H}$, so its rank coincides with its trace and we obtain
\begin{equation*}
\mu_{K G}(U) = \tr_{K G}(\pazF^{-1}(\chi_U)) = \pazF^{-1}(\chi_U)(e) = \int_{\wh{H}} \chi_U(\phi) \phi(e) d\wh{\mu}(\phi) = \wh{\mu}(U).\hfill\qedhere
\end{equation*}
\end{proof}

\begin{remark}\label{remarks-austin}
An important observation is that, once we have proven that the Haar measure $\wh{\mu}$ on $\wh{H}$ is $T$-invariant, this property does not depend on the base field $K$ anymore. So, by assuming now that $K$ is any field of arbitrary characteristic $p$ (with $p$ not dividing any natural number $n \in \calO$) and containing all the $n^{\text{th}}$ roots of unity for any $n \in \calO$, \textit{and} by assuming ergodicity of $\wh{\mu}$, we can invoke Theorem \ref{theorem-rank.function} to obtain a `canonical' Sylvester matrix rank function on $K G$, by simply defining $\rk_{K G} := \rk_{\calA} \circ \pazF$.
\end{remark}

We can use Proposition \ref{proposition-measure.H} to prove that the $*$-regular closure of the group algebra $K G$ inside $\calU(G)$, which we denoted by $\calR_{K G}$, can be identified with $\calR_{\calA}$, the $*$-regular closure of $\calA$ inside the rank-completion $\gotR_{\rk}$ of $\calA$ with respect to its rank function $\rk_{\calA}$ (recall Theorem \ref{theorem-rank.function}).

\begin{theorem}\label{theorem-identify.*reg.closure}
Consider the same notation and hypotheses as in Proposition \ref{proposition-measure.H}, and assume that $\rk_{K G}$ is extremal in $\Ps(K G)$. We then obtain a $*$-isomorphism $\calR_{K G} \cong \calR_{\calA}$. In fact, we have commutativity of the diagram
\begin{equation*}
\vcenter{
	\xymatrix{
	\calA \ar[d]^*[@!270]{\cong} \ar@{^{(}->}[r] & \calR_{\calA} \ar[d]^*[@!270]{\cong} \ar@{^{(}->}[r] & \gotR_{\rk} \ar@{^{(}->}[d] \\
	K G \ar@{^{(}->}[r] & \calR_{K G} \ar@{^{(}->}[r] & \calU(G).
	}
}\label{diagram-comm.diag.4}
\end{equation*}
Moreover, the rank-completions of both $KG$ and $\calR_{KG}$ are $*$-isomorphic to $\calM_K$, the von Neumann continuous factor over $K$.
\end{theorem}
\begin{proof}
Since $\calU(G)$ is complete with respect to the $\rk_{\calU(G)}$-metric, Proposition \ref{proposition-measure.H} together with Theorem \ref{theorem-rank.function} tell us that $\gotR_{\rk}$ embeds in $\calU(G)$, making the previous diagram commutative. In turn, since $\gotR_{\rk}$ is itself $*$-regular, we see that
$$\calR_{K G} = \calR(K G,\calU(G)) \cong \calR(\calA,\gotR_{\rk}) = \calR_{\calA}$$
as $*$-algebras, as required. The last part follows from Proposition \ref{proposition-reg.closure.completion}.
\end{proof}

\section{The lamplighter group algebra}\label{section-lamplighter.algebra}

In this section we apply the constructions given in Sections \ref{section-approx.crossed.product} and \ref{section-*.regular.closure} to study the lamplighter group algebra. This algebra is of great relevance because, among other things, it gave the first counterexample to the Strong Atiyah Conjecture, see for example \cite{GZ}, \cite{DiSc} and the Introduction. 

\begin{definition}\label{definition-lamplighter}
The lamplighter group $\Gamma$ is the wreath product of the finite group $\Z_2$ of two elements by $\Z$. In other words,
$$\Gamma = \Z_2 \wr \Z = \Big( \bigoplus_{i \in \Z} \Z_2 \Big) \rtimes_{\sigma} \Z$$
where the semidirect product is taken with respect to the Bernoulli shift $\sigma : \Z \act \bigoplus_{i \in \Z} \Z_2$ defined by
$$\sigma_n(x)_i = x_{i+n} \quad \text{ for } x = (x_i) \in \bigoplus_{i \in \Z} \Z_2.$$
In terms of generators and relations, if we denote by $t$ the generator corresponding to $\Z$, and by $a_i$ the generator corresponding to the $i^{\text{th}}$ copy of $\Z_2$, we have the presentation
\begin{equation*}
\Gamma = \langle t, \{ a_i \}_{i \in \Z} \mid a_i^2,\text{ } a_i a_j a_i a_j,\text{ } ta_it^{-1}a_{i-1} \text{ for } i, j \in \Z \rangle.
\end{equation*}
\end{definition}

Now the Fourier transform (if $K$ is any field with involution of characteristic different from $2$) gives a $*$-isomorphism $K \Gamma \cong C_K(X) \rtimes_T \Z$, where $X = \{0,1\}^{\Z}$ is the Cantor set and $T$ is the shift map, namely $T(x)_i = x_{i+1}$ for $x \in X$. The isomorphism is given by the identifications
$$1 \mapsto \chi_X, \quad t \mapsto t, \quad a_i \mapsto \chi_{U_i} - \chi_{X \backslash U_i}$$
where $U_i = \{ x \in X \mid x_i = 0\}$. Note that, in particular, the elements $e_i = \frac{1 + a_i}{2}$ are idempotents in $K \Gamma$, and so are $f_i = 1 - e_i$. They correspond to the characteristic functions of the clopen sets consisting of all the elements in $X$ having a $0$ (resp. a $1$) at the $i^{\text{th}}$ component, respectively.

\subsection{The approximating algebras \texorpdfstring{$\calA_n$}{} for the lamplighter group algebra}\label{subsection-approximating.algebras.lamplighter}

We now give a concrete family of pairs $\{(E_n,\calP_n)\}_{n \geq 0}$ as in Section \ref{section-approx.crossed.product}, together with the corresponding algebras $\calA_n = \calA(E_n,\calP_n)$ for the algebra $\calA = C_K(X) \rtimes_T \Z$ corresponding to the lamplighter group.



We will follow the same notation as in \cite[Section 3]{Gra14}: given $\epsilon_{-k},...,\epsilon_l \in \{0,1\}$, the cylinder set $\{ x = (x_i) \in X \mid x_{-k} = \epsilon_{-k},...,x_l = \epsilon_l \}$ will be denoted by $[\epsilon_{-k} \cdots \underline{\epsilon_0} \cdots \epsilon_l]$. It is then clear that a basis for the topology of $X$ is given by the collection of clopen sets consisting of all the cylinder sets.

By \cite[Example 3.1]{KM}, the usual $\big( \frac{1}{2}, \frac{1}{2} \big)$-product measure on $X$ is an ergodic, full and shift-invariant probability measure. In fact,
\begin{align*}
\rk_{K \Gamma}( \pazF^{-1}(\chi_{[\epsilon_{-k} \cdots \underline{\epsilon_0} \cdots \epsilon_l]}) ) 
& = \frac{1}{2^{l+k+1}} = \mu([\epsilon_{-k} \cdots \underline{\epsilon_0} \cdots \epsilon_l]).
\end{align*}
It follows from Theorem \ref{theorem-rank.function} that $\rk_{K \Gamma} \circ \pazF^{-1}$ coincides with $\rk_{\calA}$, where $\calA = C_K(X) \rtimes_T \Z$. In particular, the set of $\ell^2$-Betti numbers arising from $\Gamma$ with coefficients in $K$ can be also computed by means of $\rk_{\calA}$:
$$\calC(\Gamma,K) = \bigcup_{k \geq 1} \{k - \rk_{\calA}(A) \mid A \in M_k(\calA) \},$$
which is closely related to $\calC(\calA) = \rk_{\calA}\Big(\bigcup_{k \geq 1} M_k(\calA)\Big)$. In fact, they are both subsemigroups of $(\R^+,+)$ which generate the same subgroup $\calG(\Gamma,K) = \calG(\calA)$, see Subsection \ref{subsection-Betti.numbers.Atiyah.problem}.

Let us start our analysis of $K \Gamma \cong C_K(X) \rtimes_T \Z$ using the dynamical approximation from Section \ref{section-approx.crossed.product}. For $n \geq 0$, we take $E_n = [1 \cdots \underline{1} \cdots 1]$ (with $2n+1$ one's) for the sequence of clopen sets, whose intersection gives the point $y = (...,1,\underline{1},1,...) \in X$ which is a fixed point for the shift map $T$ \footnote{It can also be done by taking an even number of one's at each level $n$; we are taking an odd number for notational convenience.}. We take the partitions $\calP_n$ of the complements $X \backslash E_n$ to be the obvious ones, namely
$$\calP_n = \{[\epsilon_{-n} \cdots \underline{\epsilon_0} \cdots \epsilon_n]\, \mid \, \epsilon_i\in \{ 0,1\}\}\setminus \{E_n\}.$$
Write $\calA_n := \calA(E_n, \calP_n)$ for the unital $*$-subalgebra of $\calA = C_K(X) \rtimes_T \Z$ generated by the partial isometries $\chi_Z t$, $Z \in \calP_n$. It is easily seen that $\calA_n$ coincides with the unital $*$-subalgebra of $\calA$ generated by the partial isometries $s_i = e_i t$ for $-n \leq i \leq n$, where recall that each $e_i$ is the projection in $K \Gamma$ given by $\frac{1 + a_i}{2}$ (equivalently, the characteristic function of the clopen set $[0_i]$), and we put $f_i = 1 - e_i$. 
We have, for each $n \geq 0$, inclusions $\calA_n \subseteq \calA_{n+1}$. The quasi-partition $\ol{\calP}_n$ consists of the translates of the sets $W \in \V_n$ of the following types:
\begin{enumerate}[a),leftmargin=1cm]
\item $W_0 = [1 1 \cdots \underline{1} \cdots 1 1 1]$ of length $1$ (there are $2n+2$ one's);
\item $W_1 = [1 1 \cdots \underline{1} \cdots 1 1 0 1 1 \cdots 1 \cdots 1 1]$ of length $2n+2$ (there are $4n+2$ one's, and a zero);
\item $W(*,*,...,*,*) = [1 1 \cdots \underline{1} \cdots 1 1 0 * * \cdots * * 0 1 1 \cdots 1 \cdots 1 1]$ of length $(2n+3)+l$,
\end{enumerate}
where in the last type $l \geq 0$ is the number of $*$, and each $*$ can be either a zero or a one, but with at most $2n$ consecutive one's. It can be checked by hand that indeed $\ol{\calP}_n$ forms a quasi-partition of $X$, namely that
$$\sum_{W \in \V_n} |W| \mu(W) = 1.$$
To this aim, we first need a definition. We write $m = 2n+1$.

\begin{definition}\label{definition-fibonacci.numbers}
For $k \in \Z^+$ we define the $k^{\text{th}}$ $m$-acci number, denoted by $\text{Fib}_m(k)$, recursively by setting
$$\text{Fib}_m(0) = 0, \quad \text{Fib}_m(1) = \text{Fib}_m(2) = 1, \quad \text{Fib}_m(3) = 2, \quad \dots , \quad \text{Fib}_m(m-1) = 2^{m-3},$$
and for $r \in \Z^+$,
$$\text{Fib}_m(r+m) = \text{Fib}_m(r+m-1) + \cdots + \text{Fib}_m(r).$$
This sequence is also known in the literature as the $m$-step Fibonacci sequence, see for example \cite{Flores} and \cite{NoePost}.
\end{definition}

\begin{lemma}\label{lemma-fibonacci.numbers}
For $k \geq 2$, $\emph{Fib}_m(k)$ is exactly the number of possible sequences $(\epsilon_1,...,\epsilon_l)$ of length $l = k-2$ that one can construct with zeroes and ones, but having at most $m-1$ consecutive one's.
\end{lemma}
\begin{proof}
A simple combinatorial argument yields the result.
\end{proof}

By using the summation rules
$$\sum_{k \geq 1} \frac{\text{Fib}_m(k)}{2^k} = 2^{m-1}, \qquad \sum_{k \geq 1} \frac{k \text{Fib}_m(k)}{2^k} = 2^{2m}-(m+1)2^{m-1},$$
whose proofs can be found in \cite[Lemma 3.2.3]{Claramunt}, we compute
$$\sum_{W \in \V_n} |W| \mu(W) = \frac{1}{2^{m+1}} + \frac{1}{2^{2m}} \Big( \sum_{k \geq 1} \frac{m \text{Fib}_m(k)}{2^k} + \sum_{k \geq 1} \frac{k\text{Fib}_m(k)}{2^k} \Big) = \frac{m+1}{2^{m+1}} + \Big( 1 - \frac{m+1}{2^{m+1}} \Big) = 1,$$
as we already know.

Recall from Section \ref{section-approx.crossed.product} that we have faithful $*$-representations $\pi_n : \calA_n \hookrightarrow \gotR_n$, $x \mapsto (h_W \cdot x)_W$. In our situation, we have the concrete expression $\gotR_n = K \times  \prod_{k \geq 1} M_{m+k}(K)^{\text{Fib}_m(k)}$.

If now $K$ is a subfield of $\C$ closed under complex conjugation, then by Theorem \ref{theorem-identify.*reg.closure} we can identify $\calR_{K \Gamma} \cong \calR_{\calA}$, and in fact the $*$-regular closure of each $\calA_n$ inside $\calU(\Gamma)$ coincides with $\calR_n = \calR(\calA_n,\gotR_n)$, and the same for $\calR_{\infty} = \calR(\calA_{\infty}, \gotR_{\infty})$. In particular, Proposition \ref{proposition-reg.closure.completion} applies to give the following result, already proved by Elek in \cite{Elek16} for $K = \C$.

\begin{corollary}\label{corollary-vN.cont.factor}
Let $K$ be a subfield of $\C$ closed under complex conjugation, and let $\calR_{\rk}$ to be the rank-completion of $\calR_{K \Gamma}$ inside $\calU(\Gamma)$ with respect to $\rk_{\calU(\Gamma)}$. Then $\calR_{\rk} \cong \calM_K$ as $*$-algebras over $K$, where $\calM_K$ denotes the von Neumann continuous factor over $K$.
\end{corollary}

\subsection{The algebra of special terms for the lamplighter group algebra}\label{subsection-lamplighter.alg.reg.closure}

We now interpret the results in Subsection \ref{subsection-study.RB} for the lamplighter group algebra. In particular, we show that the corresponding algebra of special terms $\calS_n[[t;T]]$ is an integral domain. Our notation here is a little bit different from that section: we write $\calA_{n,0}[[t;T]]$ instead of $\calB_0[[t,T]]$ to denote the set of infinite sums
$$\sum_{i \geq 0} b_i (\chi_{X \backslash E_n}t)^i = \sum_{i \geq 0} b_i t^i , \quad \text{where } b_i \in \calA_{n,i} := \chi_{X \backslash (E_n \cup \cdots \cup T^{i-1}(E_n))} \calA_{n,0}$$
with $\calA_{n,0} = C_K(X) \cap \calA_n$. We then have a representation $\pi_n : \calA_{n,0}[[t;T]] \to \gotR_n$, $\pi_n(a) = (h_W \cdot a)_W$.

Following Definition \ref{definition-divisions.closure}, we denote by $(\calD_n)_+$ the division closure of $(\calA_n)_+ = \bigoplus_{i \geq 0} \calA_{n,i} t^i$ in $\calA_{n,0}[[t;T]]$. We write $\calS_n[[t;T]]$ to denote the subspace of $\calA_{n,0}[[t;T]]$ consisting of those elements $\sum_{i \geq 0} b_i t^i$ such that each $b_i$ belongs to $\text{span}\{ \chi_S \mid S \in \calW_i\}$ (see Definition \ref{definition-special.terms}). These are easy to describe here: noting that $E_n \cap T^{-1}(E_n) \neq \emptyset$, we have that the special term of degree $0$ is given by
$$S_0 = T^{-1}(S_1) = [\underbrace{1 \cdots \underline{1} \cdots 1 1}_{2n}],$$
which corresponds to $\chi_{S_0} = f_{-n+1} \cdots f_0 \cdots f_{n-1} f_n$;
the special one of degree $i = 2n+1$ is
$$S = [\underbrace{1 1 \cdots 1 \cdots 1}_{2n} 0 \underbrace{1 \cdots \underline{1} \cdots 1 1}_{2n}],$$
corresponding to $\chi_S t^{2n+1} = f_{-3n} f_{-3n+1} \cdots f_{-2n} \cdots f_{-n-1} e_{-n} f_{-n+1} \cdots f_0 \cdots f_{n-1} f_n t^{2n+1}$; finally, 
with degree $i \geq 2n+2$, we have the elements
$$S = [\underbrace{1 1 \cdots 1 \cdots 1}_{2n} 0 \underbrace{* * \cdots * *}_{i - (2n+2)} 0 \underbrace{1 \cdots \underline{1} \cdots 1 1}_{2n}]$$
corresponding to
$$\chi_S t^i = f_{-n-i+1} f_{-n-i+2} \cdots f_{-i+1} \cdots f_{-i+n} e_{-i+n+1} (\ast)_{-i+n+2} \cdots (\ast)_{-n-1} e_{-n} f_{-n+1} \cdots f_0 \cdots f_{n-1} f_n t^i,$$
with $(\ast)_j \in \{e_j,f_j\}$ having no more than $2n$ consecutive $f_j$'s. The next lemma shows that the lamplighter group algebra has some special properties that are reflected in $\calS_n[[t;T]]$.

\begin{lemma}\label{lemma-special.elements.are.algebra}
The space $\calS_n[[t;T]]$ becomes a subalgebra of $\calA_{n,0}[[t;T]]$, and even an integral domain.
\end{lemma}

\begin{proof}
We show that if $S \in \calW_i$, $S'\in \calW_j$ then $S\cap T^i(S') \in \calW_{i+j}$ (here $i,j \geq 2n+2$, the other cases can be also checked in a similar way).
We have 
$$ \chi_S = f_{-n-i+1} \cdots f_{-i+1} \cdots f_{-i+n} e_{-i+n+1} a_{-i+n+2} \cdots a_{-n-1} e_{-n} f_{-n+1} \cdots f_0 \cdots f_n,$$
$$\chi_{S'} = f_{-n-j+1} \cdots f_{-j+1} \cdots f_{-j+n} e_{-j+n+1} b_{-j+n+2} \cdots b_{-n-1} e_{-n} f_{-n+1} \cdots f_0 \cdots f_n,$$
with $a_k, b_k \in \{e_k,f_k\}$ with no more than $2n$ consecutive $f_k$'s, so that 
\begin{align*}
\chi_S t^i \cdot \chi_{S'} t^j & = \chi_{S \cap T^i(S')} t^{i+j} = f_{-n-j-i+1} \cdots f_{-j-i+1} \cdots f_{-j-i+n} e_{-j-i+n+1} b_{-j-i+n+2} \cdots b_{-n-i-1} e_{-n-i} \\
& \quad \cdot f_{-n-i+1} \cdots f_{-i} f_{-i+1} \cdots f_{-i+n} e_{-i+n+1} a_{-i+n+2} \cdots a_{-n-1} e_{-n} f_{-n+1} \cdots f_0 \cdots f_n t^{i+j}.
\end{align*}
Now it is clear that $S \cap T^i(S') \in \calW_{i+j}$.  This shows that $\calS_n[[t;T]]$ is a subalgebra of $\calA_{n,0} [[t;T]]$. To show that $\calS_n[[t;T]]$ is a domain, consider two non-zero elements $a,b \in \calS_n[[t;T]]$ , and let $\chi_S t^i$ and $\chi_{S'} t^j$ be terms in the support of $a$ and $b$ respectively, of smallest degree. By the computation above $\chi_S t^i \cdot \chi_{S'} t^j = \chi_{S \cap T^i(S')} t^{i+j}$ is a non-zero term of smallest degree in $ab$. This shows that $ab \ne 0$. Note that the special term $\chi_{S_0 \cup T^{-1}(S_1)} = \chi_{S_0} = \chi_{T^{-1}(S_1)}$ is the unit of the algebra $\calS_n[[t;T]]$.
\end{proof}

Define $\calS_n[t;T] \subseteq \calS_n[[t;T]]$ as the set of elements of $\calS_n[[t;T]]$ with finite support, i.e. of the form $\sum_{i=0}^r b_i t^i$ with $b_i$ belonging to the linear span of the special elements of degree $i$, and $r$ a positive integer.

\begin{proposition}\label{proposition-free.algebra}
For $n \geq 1$, $\calS_n[t;T]$ is a free $K$-algebra with infinitely many generators, and $\calS_n[[t;T]]$ is a free power series $K$-algebra with infinitely many generators.
\end{proposition}

\begin{proof}
We say that a special term $\chi_S t^i$ of the form
$$f_{-n-i+1}\cdots f_{-i+1} \cdots f_{-i+n} e_{-i+n+1} a_{-i+n+2} \cdots a_{-n-1} e_{-n} f_{-n+1}\cdots f_0 \cdots f_n t^i$$
is \textit{pure} if there are no more than $2n-1$ consecutive $f_j$'s in the $a_{-i+n+2}\cdots a_{-n-1}$ sequence. Denote by $\calP_u$ the set of pure elements. Then every special term $\chi_S t^i$ can be written uniquely as a product of pure terms, so we obtain an isomorphism
$$K\langle \langle x_b \mid b \in \calP_u \rangle \rangle \cong \calS_n[[t;T]], \quad x_b \mapsto b, \text{ } 1 \mapsto \chi_{[1 \cdots \underline{1} \cdots 1 1]}$$
which restricts to an isomorphism $K\langle x_b \mid b \in \calP_u \rangle \cong \calS_n[t;T]$.
\end{proof}

In the next subsection we provide the description of $\calS_n[[t;T]]$ for the case $n = 0$.

We now observe that, for $n\ge 1$, the $*$-regular algebra $\calQ_n$ and so the $*$-algebra $\calE_n$, corresponding to the algebra $\calA_n$ (described in Definition \ref{definition-algebraQ}), contain a well-known large $*$-subalgebra.

With the notation used in the proof of Proposition \ref{proposition-free.algebra}, denote by $K_{\mathrm{rat}}\langle x_b\mid b\in \calP_u\rangle $ the algebra of non-commutative rational series, which is by definition the division closure of $K\langle x_b\mid b\in \calP_u\rangle $ in $K \langle\langle  x_b\mid b\in \calP_u\rangle\rangle $, see \cite{BR}. Note that
$$K_{\mathrm{rat}}\langle x_b \mid b \in \calP_u\rangle =\bigcup_F K_{\mathrm{rat}}\langle x_b\mid b\in F\rangle,$$
where $F$ ranges over all the finite subsets of $\calP_u$. We see $K_{\mathrm{rat}}\langle x_b\mid b\in \calP_u\rangle$ as a subalgebra of $\calS_n[[t;T]]$ via the identification $K \langle \langle x_b \mid b\in \calP_u\rangle \rangle \cong \calS_n[[t;T]]$ provided by Proposition \ref{proposition-free.algebra}. By \cite[Theorems 1.5.5 and 1.7.1]{BR}, the algebra $K_{\mathrm{rat}}\langle x_b\mid b\in \calP_u\rangle$ is even a $*$-subalgebra of $(\calS_n[[t;T]],\odot,-)$. We will denote the algebra of rational series endowed with the Hadamard product $\odot$ by $K_{\mathrm{rat}}\langle x_b\mid b\in \calP_u\rangle^{\circ}$. Note that, since $\calP_u$ is infinite, the algebra $K_{\mathrm{rat}}\langle x_b\mid b\in \calP_u\rangle^{\circ}$ is not unital, but it has a local unit, namely the family $\{(1-\sum_{b\in F} x_b)^{-1}\}_F$, where $F$ ranges over all the finite subsets of $\calP_u$.  

\begin{proposition}\label{proposition-rational.series}
Let $n\ge 1$ and let $\calQ_n$ be the $*$-regular closure of $P((\calD_n)_+)$ in $(\calS_n[[t;T]],\odot,-)$. Then $\calQ_n$ contains the $*$-regular closure of $K_{\mathrm{rat}}\langle x_b\mid b\in \calP_u\rangle^{\circ}$ in $(\calS_n[[t;T]],\odot,-)$.
\end{proposition}
\begin{proof}
Using the identifications given in the proof of Proposition \ref{proposition-free.algebra},	we only need to check that the division closure $\mathcal D$ of $\calS_n[t;T]$ in $\calS_n[[t;T]]$ is contained in $P((\calD_n)_+)$.

Let us denote by $g$ the unit of $\calS_n [[t;T]]$, that is $g = \chi_{S_0\cap T^{-1}(S_1)} = \chi_{S_0} = \chi_{T^{-1}(S_1)}$. Then the division closure of $(1-g)K+\calS_n[t;T]$ in $(1-g)K+\calS_n[[t;T]]$ is precisely $(1-g)K+\mathcal D$. Moreover we have inclusions of unital algebras
$$(1-g)K+\calS_n[t;T]\subseteq (\calD_n)_+\cap ((1-g)K+\calS_n[[t;T]])\subseteq (1-g)K + \calS_n[[t;T]],$$
and $(\calD_n)_+\cap ((1-g)K+\calS_n[[t;T]])$ is inversion closed in $(1-g)K+ \calS_n[[t;T]]$, so
$$(1-g)K+\mathcal D \subseteq (\calD_n)_+\cap ((1-g)K+\calS_n[[t;T]])\subseteq (\calD_n)_+.$$
We thus get $\mathcal D = P(\mathcal D)\subseteq P((\calD_n)_+)$, as desired.
\end{proof}

To close this subsection, we compute the $*$-regular closure of $K_{\mathrm{rat}}\langle X \rangle^{\circ}$ in $K\langle \langle X\rangle \rangle^{\circ}$. To this end, we first analyze the $*$-regular closure in the setting of commutative rings.

For any unital ring $T$, we denote by $\mathbf{B}(T)$ the Boolean algebra of central idempotents of $T$. Recall that $e\wedge f= ef$ and $e\vee f= e+f-ef$ for $e,f\in \mathbf{B}(T)$.

Let $R$ be a commutative unital $*$-regular ring and let $S$ be a unital $*$-subring of $R$. Observe that the idempotents of $R$ are necessarily self-adjoint. We want to obtain a simplified form of the construction of the $*$-regular closure of $S$ in $R$.

Write
\begin{equation}\label{equation-definition.of.E.idempotents}
E(S)=\{ e\in \mathbf{B}(R) \mid \text{ there exists }a\in S \text{ such that } aR=eR \}.
\end{equation}
We can think of the elements of $E(S)$ as being the supports of the elements of $S$. Note that $\mathbf{B}(S)\subseteq E(S)\subseteq \mathbf{B}(R)$ and that for $a\in S$ and $e\in E(S)$ such that $aR=eR$, we have $(1-e)R= \text{Ann}_R(a)$, the annihilator of $a$ in $R$.

\begin{lemma}\label{lemma-Boolean.prop}
Let $R$ be a commutative unital $*$-regular ring and let $S$ be a unital $*$-subring of $R$. The set $E(S)$ is closed under meets in $\mathbf{B}(R)$. Consequently, $E(S)$ is a Boolean subalgebra of $\mathbf{B}(R)$ if and only if $E(S)$ is closed under complements.
\end{lemma}
\begin{proof}
Let $a,b\in S$, and suppose that $aR=eR$ and $bR=fR$ for $e,f\in \mathbf{B}(R)$. We then have:
$$(e\wedge f)R= efR = (eR)(fR)= (aR)(bR)= (ab)R,$$
which shows that $e\wedge f \in E(S)$. The second part follows from the De Morgan laws.
\end{proof}

We can obtain now the description of the $*$-regular closure.

\begin{proposition}	\label{proposition-commutative.reg.closure}
Let $R$ be a unital commutative $*$-regular ring and let $S$ be a unital $*$-subring of $R$. Let $E(S)\subseteq \mathbf{B}(R)$ be the set defined in \eqref{equation-definition.of.E.idempotents}, and let $\mathbf{B}_S$ be the smallest Boolean subalgebra of $\mathbf{B}(R)$ containing $E(S)$.  Let $S\mathbf{B}_S$ be the subring of $R$ generated by $S$ and $\mathbf{B}_S$. Then the $*$-regular closure  $\calR$ of $S$ in $R$ coincides with the classical ring of quotients $Q_{\emph{cl}}(S\mathbf{B}_S)$ of $S\mathbf{B}_S$. Moreover we have $\mathbf{B}(\calR)= \mathbf{B}_S$.
\end{proposition}
\begin{proof}
We first show that $Q_{{\rm cl}}(S\mathbf{B}_S)$ naturally embeds in $R$. Observe that every element $x \in S\mathbf{B}_S$ can be written in the form $x = \sum_{i=1}^n e_is_i$, where $s_i \in S$ for all $i$ and $(e_i)$ is a sequence of pairwise orthogonal elements of $\mathbf{B}_S$. Now let $f_i \in E (S)$ such that $s_iR = f_iR$. We have $(e_is_i)R = (e_if_i)R$, and $h_i := e_if_i = e_i \wedge f_i \in \mathbf{B}_S$, because $\mathbf{B}_S$ is a Boolean subalgebra of $\mathbf{B}(R)$. Hence, by replacing each $e_i$ by $h_i$, we can further assume that $(e_is_i)R = e_iR$. It follows that the support projection of $x$ in $R$ is exactly $\sum_{i=1}^n e_i = \vee_{i=1}^n e_i \in \mathbf{B}_S$. So we have shown that the support projection in $R$ of any element of $S\mathbf{B}_S$ belongs to $\mathbf{B}_S$.

Let $z$ be a non-zero-divisor of $S\mathbf{B}_S$, and write $zR = eR$ for some idempotent $e \in R$. By what we proved above we have $e\in \mathbf{B}_S$. Since $(1-e)z = 0$, we conclude that $e = 1$. Hence $z$ is invertible in $R$ and therefore there is a unique embedding of $Q_{{\rm cl}}(S\mathbf{B}_S)$ into $R$ extending the canonical embedding $S\hookrightarrow R$.

Let $\calR$ denote the $*$-regular closure of $S$ in $R$. We first show that $Q_{{\rm cl}}(S\mathbf{B}_S)\subseteq \calR$. It is clear that $E(S) \subseteq \calR$, so that $\mathbf{B}_S\subseteq \mathbf{B}(\calR)$ and $S\mathbf{B}_S\subseteq \calR$. Now since all non-zero-divisors of $S\mathbf{B}_S$ are invertible in $\calR$, we get that $Q_{{\rm cl}}(S\mathbf{B}_S)\subseteq \calR$.

For the other inclusion  $\calR \subseteq Q_{{\rm cl}}(S\mathbf{B}_S)$, it suffices to show that $ Q_{{\rm cl}}(S\mathbf{B}_S)$ is a $*$-regular subring of $R$. First notice that since $S$ is a $*$-subring of $R$, $S\mathbf{B}_S$ is a $*$-subring of $R$ too. It follows that $Q_{{\rm cl}}(S\mathbf{B}_S)$ is a $*$-subring of $R$. Now let $x= ab^{-1}\in  Q_{{\rm cl}}(S\mathbf{B}_S)$, where $a \in S\mathbf{B}_S$, and $b$ a non-zero-divisor in $S\mathbf{B}_S$. We proved above that the support projection, say $e$, of $a$ belongs to $\mathbf{B}_S$. Hence $a + (1-e)$ is a non-zero-divisor in $S\mathbf{B}_S$, so that it is invertible in $ Q_{{\rm cl}}(S\mathbf{B}_S)$. Now it is easily checked that the quasi-inverse of $x$ is $eb(a+(1-e))^{-1}\in Q_{{\rm cl}}(S\mathbf{B}_S)$, so that $Q_{{\rm cl}}(S\mathbf{B}_S)$ is a $*$-regular ring.
\end{proof}

Let $X$ be an infinite countable set and write $X = \bigcup_{n\ge 1} X_n$, where $X_n \subseteq X_{n+1}$ and $|X_n| = n$ for all $n$. Let $\calR^n_{\text{rat}}$ be the $*$-regular closure of $K_{\mathrm{rat}} \langle X_n \rangle^{\circ}$ in $K \langle \langle X_n\rangle \rangle^{\circ}$, and let $\calR_{\text{rat}}$ be the $*$-regular closure of $K_{\mathrm{rat}} \langle X \rangle^{\circ}$ in $K\langle \langle X \rangle \rangle^{\circ}$. Then we have
$$\calR_{\text{rat}} = \Big( \bigcup_{n\ge 1} \calR^n_{\text{rat}}  \Big) + 1\cdot K.$$
Therefore, it is enough to compute the $*$-algebras $\calR^n_{\text{rat}}$.

For a finite set $X$, denote by $X^*$ the free monoid generated by $X$. For $a = \sum_{w\in  X^*} a_w w \in K \langle \langle X \rangle \rangle$, denote by $\text{supp}(a)$ the {\it support} of $a$, that is, the set of all $w \in X^*$ such that $a_w \ne 0$. The {\it annihilator} $\text{ann}(a)$ is defined as the complement of $\text{supp}(a)$ in $X^*$. The Boolean algebra $\mathbf{B}(K\langle \langle X\rangle \rangle^{\circ})$ is isomorphic to the power set $P(X^*)$ of $X^*$. A subset of $X^*$ is usually called a {\it language} (see \cite[p. 4]{BR}).

The following proposition is the key to obtain some new irrational $\ell^2$-Betti numbers arising from the lamplighter group, see \cite[Subsection 4.3]{AC2}.

\begin{proposition}\label{proposition-reg.closure.rational.series}
Let $K$ be a field with involution, and let $X$ be a finite non-empty set. Let $\mathfrak K$ be the set of all the supports of elements from $K_{{\rm rat}}\langle X\rangle$, and let $\mathbf B$ be the Boolean subalgebra of $P(X^*)$ generated by $\mathfrak K$. Then the support of any element in the $*$-regular closure $\calR (K_{{\rm rat}}\langle X \rangle^{\circ},K\langle \langle X\rangle \rangle ^{\circ})$ of $K_{{\rm rat}}\langle X \rangle^{\circ}$ in $K\langle \langle X\rangle \rangle ^{\circ}$ belongs to $\mathbf B$. Moreover for each $L\in \mathbf B$ the characteristic series $\ul{L}$ belongs to  $\calR (K_{{\rm rat}}\langle X \rangle^{\circ},K\langle \langle X\rangle \rangle ^{\circ})$.
\end{proposition}
\begin{proof}
Note that the set $\mathfrak K$ corresponds exactly with the set $E(S)$ described in Proposition \ref{proposition-commutative.reg.closure}. Hence the result follows from that proposition.
\end{proof}

\begin{remark}\label{remark-noncomm.series}
By \cite[Section 3.4]{BR}, if $K$ is a field of characteristic zero, and $|X|>1$, then the set $\mathfrak K$ of Proposition \ref{proposition-reg.closure.rational.series} is not closed under complementation, and so it is {\it not} a Boolean subalgebra of subsets of $X^*$. It follows that $\mathbf{B}$ is strictly larger than $\mathfrak K$, and in particular properly contains the Boolean algebra of all the rational languages over $X$, see \cite[Chapter 3]{BR}.
\end{remark}

\subsection{Analysis of the algebra $\calA_0$}\label{subsection-analysis.A.G.example}

Here we analyze the example in \cite[Section 6]{AG} in the light of the theory developed in the present paper. The algebra constructed there coincides with the algebra $\calA_0$, the first of the approximating algebras considered in the preceding two sections. In \cite{AG}, a concrete description of the $*$-regular closure of the $*$-algebra $\calA_0$ is obtained, and it is shown that $\calG(\calA_0) = \Q$. Recall that $\calG(\calA_0)$ denotes the subgroup of $\R$ generated by the set of $\ell^2$-Betti numbers $\calC(\calA_0)$ arising from $\calA_0$. Note that $\calA_0$ coincides with the semigroup algebra $K \pazF$ of the monogenic free inverse monoid $\pazF$ (see \cite{AG}).

We will denote by $\calD_+$ and $\calD _-$ the subalgebras of $\calA_{0,0}[[t;T]]$ and $\calA_{0,0}[[t^{-1};T^{-1}]]$ introduced in Definition \ref{definition-divisions.closure}.

Now recall from \cite{AG} that the algebra $\Sigma^{-1}\calA_0$ embeds naturally in $\gotR_0 = \prod_{i \geq 1} M_i(K)$. Here $\Sigma$ is the set of all the polynomials of the form $f(s)$ for $f(x)\in K[x]$ with $f(0) = 1$, and $s = \chi_{X\setminus E_0}t = \chi_{[\underline{0}]}t$. Denoting by $\pazB$ the image of $\Sigma^{-1}\calA_0$ in $\gotR_0$, the $*$-subalgebra of $\gotR_0$ generated by $\pazB + \pazB^*$ will be denoted by $\pazT$. This algebra was considered in \cite{AG} with the notation $\calD$ (see \cite[Proposition 6.8]{AG}). We will show below that indeed $\pazT = \calD$, where, as in Section \ref{section-*.regular.closure}, $\calD$ denotes the $*$-subalgebra of $\gotR_0$ generated by $\pi_+(\calD_+) + \pi_-(\calD_-)= \pi_+(\calD_+) + \pi_+(\calD_+)^*$. Thus in the end the algebras denoted by $\calD$ in each of the two papers do agree.

\begin{proposition}\label{proposition-agreement.of.D.and.T}
With the above notation, we have $\pazT = \calD$.
\end{proposition}

\begin{proof}  
Note that for any $f(s) \in \Sigma$, we clearly have that $f(s)^{-1} \in \calD_+$. Since obviously $\calA_0 \subseteq \calD$, we get that $\pazB = \Sigma^{-1} \calA_0 \subseteq \calD$, and since $\calD$ is a $*$-algebra we get $\pazT \subseteq \calD$.

We may consider the algebra $\calA_{0,0}((t;T))$ consisting of skew Laurent power series $\sum_{i \geq -n} b_it^i$, where $n\in \Z^+$ and $b_i\in \calB_i$ for all $i$. Observe that $\pazB\subseteq \calA_{0,0}((t;T))$. Set $\pazB_+ := \pazB \cap \calA_{0,0}[[t;T]]$.  We aim to show that $\pi_+(\calD_+) \subseteq \pazB_+$. This would imply that $\calD \subseteq \pazT$, and thus we would get the desired equality $\calD = \pazT$.

Consider the natural onto homomorphism $\rho \colon \pazB \to K(x)$ sending  $s$ to $x$ and $s^*$ to $x^{-1}$. The kernel of $\rho$ is the ideal of $\pazB$ generated by $1-ss^*$ and $1-s^*s$. We denote by $\rho_+$ the restriction of the map $\rho$ to $\pazB_+$. 

To show that $\pi_+(\calD_+) \subseteq \pazB_+$, it is enough to prove that $\pazB_+$ is closed under inverses in $\calA_{0,0}[[t;T]]$, because $\calD_+$ is the division closure of $\calA_{0,0}[t;T]$ in $\calA_{0,0}[[t;T]]$. 

Observe that $\rho$ restricts to a surjective homomorphism $\rho_0 \colon \calA_{0,0} \to K$, which can be described as follows. Given any expression $z = \sum_{U \in \calP} \lambda_U \chi_U$, where $\calP$ is a partition of $X$ into clopen basic subsets as in \cite[Lemma 3.8]{AC}, there is a unique $U_0 \in \calP$ such that $\chi_{U_0}$ does not belong to the kernel of $\rho$, namely $U_0 = [0 \stackrel{i}{\cdots} 0]$ for some $i$. Then we have $\rho_0(z) = \lambda_{U_0} = z ((...,0,\underline{0},0,...))$. Note that $\rho_0 (T^j(z)) = \rho_0(z)$ for all $j \in \Z$. This enables us to extend $\rho$ to a well-defined homomorphism $\calA_{0,0}((t;T)) \to K((x))$, also denoted by $\rho$, which is given by
$$\rho(\sum_{i=-n}^{\infty} a_it^i)= \sum_{i=-n}^{\infty} \rho_0(a_i) x^i.$$

We have the following commutative diagram
\begin{equation*}
\vcenter{
	\xymatrix{
	\pazB_+ \ar[r] \ar[d]^{\rho_+} & \calA_{0,0}[[t;T]] \ar[d]^{\rho_+} \\
	K[x]_{(x)} \ar[r]  & K[[x]]\, ,
	}
}\label{diagram-comm.extra1}
\end{equation*}
where $K[x]_{(x)}$ denotes the localization of $K[x]$ at the maximal ideal $(x)$, and $\rho_+$ from the right-hand side is the restriction of $\rho : \calA_{0,0}((t;T)) \to K((x))$ to $\calA_{0,0}[[t;T]]$. The image of the map $K[x]_{(x)} \to K[[x]]$ is of course the algebra of rational series. Now let $z = \sum_{i \geq 0} b_it^i \in \pazB_+$ be invertible in $\calA_{0,0}[[t;T]]$. By Lemma \ref{lemma-invertibility.PS}, we can assume without loss of generality that $b_0 = 1$. Then $\rho_+(z)$ must be invertible in $K[x]_{(x)}$, so that there are $f(x), g(x) \in K[x]$ such that $f(0) = g(0) = 1$ and
$\rho_+(z) = f(x)g(x)^{-1}$. Now $f(s)g(s)^{-1} \in \pazB_+$ and we have
$$z= f(s)g(s)^{-1} + y,$$
where $y$ belongs to the ideal $I_{>0} := I \cap \big( \bigoplus_{i \geq 1} \calA_{0,i}t^i \big)$, where $I$ is the ideal of $\pazB$ generated by $1-ss^*$ and $1-s^*s$. Therefore we have that $zg(s) f(s)^{-1} = 1+ y'$, where $y' \in I_{>0}$, and we need to show that $1 + y'$ is invertible. Indeed, we will show that $I_{>0}$ is a nil-ideal, that is, that every element of $I_{>0}$ is nilpotent.

By \cite[Lemma 4.7]{AG}, each element of $I$ can be expressed as a (finite) linear combination of terms of the following forms:
\begin{enumerate}
	\item[(A)] $f^{-1}s^i(1-ss^*)(s^*)^j$, for $i,j \ge 0$ and $f \in \Sigma$,
	\item[(B)] $(s^*)^i(1-s^*s) s^jf^{-1}$, for $i,j \ge 0$ and $f \in \Sigma$,
	\item[(C)] $(s^*)^i(1-s^*s)s^jf^{-1}(1-ss^*)(s^*)^k$, for $i,j,k \ge 0$ and $f \in \Sigma$,
	\item[(D)] elements from $\soc(A)$.
\end{enumerate}
Let $y$ be an element in $I_{>0}$. Let $\{h_n\}_n$ be the canonical central projections of $\calA_0$. Since the constant term of $y$ is $0$, we see that each matrix $h_n \cdot y$ must be a strictly 
lower triangular matrix. We want to show that there is a fixed integer $R$ such that $(h_n \cdot y)^R = 0$ for all $n$. Now taking into account the forms (A), (B),(C), (D) above, we see that there are positive integers $N$ and $r$ such that for all $n \ge N$ the matrix $h_n \cdot y$ has the property that it is strictly lower triangular and that $(h_n \cdot y)_{ij}= 0$ for all $(i,j)$ such that $i \leq n-r$ and $j \geq r+1$. The result then follows from the next straightforward lemma.
\end{proof}

\begin{lemma}\label{lemma-nilpotent.matrix}
Let $0 < r < n$. Suppose that $A = (a_{ij})_{1 \le i,j \le n}$ is a strictly lower triangular matrix such that $a_{ij}= 0$ whenever $(i,j)$ satisfies that $i \leq n-r$ and $j \geq r+1$. Then $A^{2r+1}=0$.
\end{lemma}

\begin{proof}
For $1 \le t \le r+1$, we show by induction that $A^t=(c_{ij})$, with $c_{ij}= 0$ whenever $i \le j+t-1$ and whenever $(i,j)$ satisfies that $i\le n-r+t-1$ and $j\ge r+1$, and also whenever $(i,j)$ satisfies that $i \le n-r$ and $j \ge r-t+2$. Moreover if $t \ge 3$ then 
$$c_{n-r+1,r+3-t'} = c_{n-r+2, r+4-t'} = \cdots = c_{n-r+t'-2,r}=0$$
for $3\le t'\le t$.

Assume the result holds for $t \le r$. We will show it holds for $t+1$. Write $D:=A^{t+1}$. It is a simple matter to show that $d_{ij}= 0$ whenever $i \le j+t$. Now assume that the pair $(i,j)$ satisfies that $i \le n-r+t$ and that $j \ge r+1$. Suppose that there is a non-zero term $c_{ik}a_{kj}$ contributing to the term $d_{ij}$. Then $i > k+t-1$ and $k > j$. Therefore we have
$$k+t-1 < i \le n-r+t,$$
which implies that $k-1 < n-r$ so that $k \le n-r $. But now since $k \le n-r$ and $j \ge r+1$ we have that $a_{kj} = 0$ by hypothesis. So all the products $c_{ik}a_{kj}$ are $0$ and we get that $d_{ij}= 0$.

Now assume that $(i,j)$ satisfies that $i \le n-r$ and $j \ge r-t+1$. Proceeding as above the term $c_{ik}a_{kj}$ is either $0$ or $i > k+t-1$ and $k > j$. We get $k > j \ge r-t+1$ and so $k \ge r-t+2$. Since  $i \le n-r $ and $k \ge r-t+2$, we get that $c_{ik}=0$ by the induction hypothesis, and so $c_{ik}a_{kj}= 0$. Therefore $d_{ij}=0$.

It remains to show the last statement. Suppose first that $t = 3$. We have to show that $d_{n-r+1,r} = 0$. Assume there is a term $c_{n-r+1,k}a_{k,r}$ which is non-zero, where here $A^2= (c_{ij})$. Then we must have $n-r+1 > k+1$ and $k > r$ and thus $k \ge r+1$ so that $c_{n-r+1, k} = 0$ by what we have proved before, and we get a contradiction. Hence $d_{n-r+1,r} = 0$. Using induction, we assume the result true for all $3 \le t'\le t \le r$ and we show it for $t+1$. Setting $A^{t+1} = (d_{ij})$, we need to show that $d_{n-r+s,r+s-t+1} = 0$ for $1 \le s \le t-1$ (this is the case $t' = t+1$ of the statement, the cases where $3 \le t' \le t$ follow the same pattern). Write $A^t = (c_{ij})$ and let $s$ be an integer such that $1\le s \le t-1$. Let $c_{n-r+s,k}a_{k,r+s-t+1}$ be non-zero. By the induction hypothesis and what we proved before, we known that $c_{n-r+s,j}=0$ for $j\ge r+s-t+2$. Therefore we get that $k<r+s-t+2$. We also have that $k>r+s-t+1$, since $a_{k,r+s-t+1}\ne 0$, so $k\ge r+s-t+2$, and we get a contradiction. 

Therefore we have proven that $A^{r+1}$ is a matrix consisting of a $r \times r$ lower diagonal matrix at the left lower corner, and the rest of the entries are $0$. It follows that $A^{2r+1} = 0$. 
\end{proof}

We now proceed to describe the special elements $\calS_0[[t;T]]$ at this level. There is exactly one special term for each degree $i \geq 0$. For $i = 0$, it is given by $S_0 = T^{-1}(S_1) = X$ and the corresponding element inside the algebra is $\chi_X = 1$. For $i \geq 1$ we get the element $S = [0 \stackrel{i}{\cdots} \underline{0}]$, corresponding to $\chi_S t^i = \chi_{[0 \stackrel{i}{\cdots} \underline{0}]}t^i$. Note that $s^i = \chi_{[0 \stackrel{i}{\cdots} \underline{0}]}t^i$ for $i > 0$, where $s = \chi_{X \backslash E_0}t$, and so $\calS_0[[t;T]]$ can be isomorphically identified with the algebra of formal power series $K[[x]]$, the isomorphism sending $s \mapsto x$.

The isomorphism given in Proposition \ref{proposition-S.corner.R} coincides exactly with the isomorphism $\psi$ given in \cite[paragraph preceding Proposition 6.8]{AG}, and formulas (6.3) and (6.4) from \cite{AG} are deduced from Lemma \ref{lemma-formulas.computations}.

As in \cite{AG}, we denote by $\calR^{\circ}$ the subalgebra of rational series in $ K [[x]]$, endowed with the Hadamard product $\odot$. It is not closed under inversion in $(K[[x]],\odot)$. Let $\pazQ$ be the classical ring of quotients of $\calR^{\circ}$ in $(K[[x]],\odot)$, which coincides with its $*$-regular closure $\calR(\calR^{\circ},K[[x]])$ (see \cite[Lemma 6.10]{AG}, where it is denoted by $\calQ$). The algebra $\pazQ$ is not to be confused with our algebra $\calQ$ which, by Definition \ref{definition-algebraQ}, is the $*$-regular closure of $P(\calD_+)$ inside $(K[[x]],\odot)$.

\begin{lemma}\label{lemma-Q'.inside.Q}
We have $\pazQ \subseteq \calQ$.
\end{lemma}
\begin{proof}
Once we show the inclusion $\calR^{\circ} \subseteq P(\calD_+)$ we will be done since $\calQ$, being the $*$-regular closure of $P(\calD_+)$ in $(K[[s]],\odot)$, will contain $\pazQ$.

So let $p(s) \in \calR^{\circ}$. By definition, $p(s) \in K[[s]]$, so that $P(p(s)) = p(s)$. (Recall that $P$ is the idempotent map given in Lemma \ref{lemma-the.idempotent.P}.) Hence $p(s) \in P(\calD_+)$, as required.
\end{proof}

We are now in a position to prove that the $*$-subalgebra $\calE$ of $\calR_0 = \calR(\calA_0,\gotR_0)$ generated by $\calD$ and $\Psi(\calQ)(1-ss^*)$ coincides with the $*$-subalgebra of $\calR_0$ generated by $\pazT$ and $\Psi(\pazQ)(1-ss^*)$, which in turn coincides with $\calR_0$ by \cite[Theorem 6.13]{AG}. Let $\pazE$ denote the latter $*$-subalgebra.

\begin{theorem}\label{theorem-computation.*.reg.closure.0}
We have $\calE = \pazE = \calR_0$ and $\pazQ=\calQ$.
\end{theorem}

\begin{proof}
By Proposition \ref{proposition-agreement.of.D.and.T}, we have $\calD=\pazT$, so $\calR_0$ coincides with the $*$-subalgebra $\pazE$ generated by $\calD$ and $\Psi (\pazQ)(1-ss^*)$ by \cite[Theorem 6.13]{AG}.

Now the inclusion $\pazE \subseteq \calE$ is clear by Lemma \ref{lemma-Q'.inside.Q}. Since $\calE \subseteq \calR_0=\pazE$, we get the equality  $\calE = \pazE = \calR_0$.
 
Finally $\Psi(\pazQ) (1-ss^*) = (1-ss^*)\calR_0(1-ss^*)$ by \cite[Theorem 6.13]{AG}. Since $$\Psi (\pazQ)(1-ss^*)\subseteq \Psi (\calQ)(1-ss^*)\subseteq (1-ss^*)\calR_0(1-ss^*)$$ by Lemma \ref{lemma-Q'.inside.Q} and Proposition \ref{proposition-algebraQ.inside.calR}, and the map $x\mapsto \Psi(x)(1-ss^*)$ is injective, we conclude that $\pazQ = \calQ$.  
\end{proof}

\end{document}